\documentclass[11pt,reqno]{amsart}

\usepackage{amssymb,mathrsfs,enumerate}
\usepackage{amsfonts,amsmath,amssymb, amsthm, mathtools, thmtools, amsxtra, amscd}
\usepackage{url}
\usepackage{enumitem}
\usepackage{thm-restate}

\usepackage{graphicx}
\usepackage{colortbl,here}
\usepackage{extpfeil}
\usepackage{mathtools}
\usepackage{epsfig}
\usepackage{comment}
\usepackage{subcaption}
\usepackage{caption}
\usepackage{capt-of}
\usepackage{nicematrix}

\usepackage{hyperref}
\hypersetup{colorlinks,breaklinks,
             linkcolor=blue,urlcolor=blue,
             anchorcolor=blue,citecolor=blue}

\usepackage[nameinlink, capitalise, noabbrev]{cleveref}

\usepackage{cancel}

\usepackage[
maxcitenames=3, 
maxbibnames=10,  
sorting=nyt,
giveninits,
url=false,
doi=false,
sortcites,
backref,
backend=biber
]{biblatex}

\DeclarePairedDelimiter\floor{\lfloor}{\rfloor}

\def\R {\mathbb{R}}

\def\cC {\mathcal{C}}

\def\cI {\mathcal{I}}

\def \rpsi_i {|\psi_i \rangle}
\def \lpsi_i {\langle \psi_i|}
\def \lrpsi_i{\langle \psi_i | \psi_i \rangle}
\def \rpsi_k {|\psi_k \rangle}
\def \lpsi_k {\langle \psi_k|}
\def \lrpsi_k{\langle \psi_k | \psi_k \rangle}

\newcommand{\Id}{\operatorname{Id}}

\renewcommand{\R}{\mathbb R}

\newcommand{\N}{\mathbb N}

\newcommand{\aLip}{\|a\|_{\text{Lip}}}
\newcommand{\phiLip}{\|\phi\|_{\text{Lip}}}

\newcommand{\ba}{\begin{aligned}}
\newcommand{\ea}{\end{aligned}}

\newcommand{\be}{\begin{equation}}
\newcommand{\ee}{\end{equation}}

\declaretheorem[numberwithin = section]{theorem}
\declaretheorem[sibling = theorem]{lemma}
\declaretheorem[sibling = theorem]{corollary}
\declaretheorem[sibling = theorem]{definition}
\declaretheorem[sibling = theorem]{proposition}
\declaretheorem[sibling = theorem]{assumption}
\declaretheorem[sibling = theorem]{remark}

\numberwithin{equation}{section}
\numberwithin{theorem}{section}

\begin{document}

\title[Collective dynamics of the discrete Motsch-Tadmor model]{Mono-cluster flocking and uniform-in-time stability of the discrete Motsch-Tadmor model}

\author[Ha]{Seung-Yeal Ha}
\address[Seung-Yeal Ha]{\newline Department of Mathematical Sciences and Research Institute of Mathematics, \newline
	Seoul National University, Seoul, 08826, Republic of Korea}
\email{syha@snu.ac.kr}

\author[Hoffmann]{Franca Hoffmann}
\address[Franca Hoffmann]{\newline Computing and Mathematical Sciences \newline California Institute of Technology, CA 91125 Pasadena, USA}
\email{franca.hoffmann@caltech.edu}

\author[Kim]{Dohyeon Kim}
\address[Dohyeon Kim]{\newline Computing and Mathematical Sciences \newline California Institute of Technology, CA 91125 Pasadena, USA}
\email{dohyeon@caltech.edu}

\author[Yoon]{Wook Yoon}
\address[Wook Yoon]{\newline Department of Mathematical Sciences \newline Seoul National University, Seoul 08826, Republic of Korea}
\email{ynwk178@snu.ac.kr}

\thanks{\textbf{Acknowledgment.} The work of S.-Y. Ha is supported by the National Research Foundation of Korea (NRF-2020R1A2C3A01003881), and the work of FH and DK is supported by start-up funds at the California Institute of Technology and by NSF CAREER Award 2340762.}

\begin{abstract}
The Motsch-Tadmor (MT) model is a variant of the Cucker-Smale model with a normalized communication weight function. The normalization poses technical challenges in analyzing the collective behavior due to the absence of conservation of momentum. We study three quantitative estimates for the discrete-time MT model considering the first-order Euler discretization. First, we provide a sufficient framework leading to the asymptotic mono-cluster flocking. The proposed framework is given in terms of coupling strength, communication weight function, and initial data. Second, we show that the continuous transition from the discrete MT model to the continuous MT model can be made uniformly in time using the finite-time convergence result and asymptotic flocking estimate. Third, we present uniform-in-time stability estimates for the discrete MT model. We also provide several numerical examples and compare them with analytical results. 
\end{abstract}




\maketitle

\begin{center}
   \textit{To Eitan Tadmor on his 70th birthday, with friendship and admiration.} 
\end{center}


%
%
%
%
\section{Introduction} \label{sec:1}
\setcounter{equation}{0}
Emergent collective behaviors are ubiquitous in many biological complex systems such as aggregation of bacteria, flocking of birds, swarming of fish, synchronization of fireflies, and herding of sheep, or others. Among such systems, we are mainly interested in the flocking in which self-organized particles (agents) evolve into an ordered motion with the same velocity using limited environmental information and communication between particles. To model such flocking behaviors, Cucker and Smale \cite{CS-1,CS} introduced an analytically treatable particle model that resembles the Newton-type $N$-body system with the weighted average of relative velocities as an internal force. Let us first introduce the key components of the Cucker-Smale (CS) model, which will then lead us to its variant, the Motsch-Tadmor (MT) model.

Let $x_i$ and $v_i$ be the spatial position and velocity of the $i$-th particle. Then, the dynamics of the CS particle are governed by the following first-order system of ordinary differential equations for $(x_i, v_i)$:
\begin{equation}
\begin{cases} \label{eq: CS}
 \displaystyle {\dot x}_i  = v_i\,, \quad t >0,~~ i \in [N]:= \{1, \cdots, N \},  \\
 \displaystyle {\dot v}_i  = \frac{\kappa}{N} \sum_{j = 1}^N a(\|x_j - x_i\|) (v_j - v_i)\,.
\end{cases}
\end{equation}
Here, $\kappa > 0$ and $N$ denote the uniform coupling strength and system size, respectively, and the scalar quantity $a_{ij} = a(\|x_j - x_i\|)$ represents the communication weight between the $i$-th and $j$-th particles. 
For instance in \cite{CS}, the following explicit form of communication weight function was used:
\[
    a(r) = \frac{1}{(1 + r^2)^{\beta}}\quad \text{ for some } \beta \geq 0\,.
\]
The CS model \eqref{eq: CS} has been extensively studied from various perspectives, to name a few, stochastic effects \cite{10.1063/1.3496895,Ha2009EmergenceOT}, collective dynamics of an infinite CS ensemble \cite{B, BM, WX22, WX23, WX23-1,WX24},
asymptotic flocking dynamics of the kinetic CS model 
\cite{CFRT}, time-delay \cite{CHOI201849}, general and specific network topologies \cite{LX, DQ,S}, rigorous mean-field limit \cite{H-Liu} and external forcing \cite{ST}, etc. Right after Cucker-Smale's seminal work \cite{CS}, Motsch and Tadmor \cite{Motsch_2011} introduced a generalized CS model by replacing the factor $N$ appearing in the momentum equation $\eqref{eq: CS}_2$ by the sum of weights exerted on the $i$-th particle:  

\begin{equation}
\begin{cases} \label{eq: conti_MT}
\displaystyle {\dot x}_i = v_i, \quad t >0,~~i \in [N], \\
 \displaystyle {\dot v}_i =  \kappa \sum_{j = 1}^N \phi_{ij}(v_j - v_i),
\end{cases}
\end{equation}
where the normalized communication weight function $\phi_{ij}$ is given by 
\[ 
\displaystyle {\phi}_{ij} := \frac{a(\|x_i - x_j \|)}{\sum_{k=1}^{N} a(\|x_i - x_k \|)}.
 \] 
 Then, it is easy to see that $\phi_{ij}$ is non-negative, unit sum with respect to the second index, and non-symmetric in indices $i$ and $j$: there may exist $i_0, j_0$ such that $\phi_{i_0j_0} \neq \phi_{j_0i_0}$ and
 \[ \phi_{ij} \geq 0, \quad \sum_{j=1}^{N} \phi_{ij} = 1, \quad \forall~i, j \in [N]. \]
From now on, we call the system \eqref{eq: conti_MT} as the Motsch-Tadmor model (see also \cref{sec:2.1}). Compared to the vast aforementioned literature for the CS model, the emergent dynamics of \eqref{eq: conti_MT} have been investigated only in recent literature, e.g., mono-cluster and multi-cluster flocking \cite{jin2018flocking,liu2015motsch, Motsch_2011, MT14, HJZ} and collective dynamics of an infinite MT ensemble \cite{HWX}.
 
In general, to derive a numerical scheme from the continuous model, we often implement a suitable time-discretization. Given that, we employ the first-order forward Euler scheme with a time-step $h = \Delta t$, and replace derivatives ${\dot x}_i$ and ${\dot v}_i$ by the corresponding Newton quotients in this work: 
 \begin{equation}\label{eq: discrete_MT}
	\begin{cases}
		\displaystyle x_i(n+1) = x_i(n)+hv_i(n)\,, \quad n = 0, 1, \ldots, \\
		\displaystyle v_i(n+1) = v_i(n)+h \kappa\sum_{j=1}^{N} {\phi}_{ij}(n) \big(v_j(n)-v_i(n)\big),
	\end{cases}
\end{equation}
where the normalized communication weight $\phi_{ij}(n)$ is given by
\begin{equation} \label{New-1-1}
{\phi}_{ij}(n) := \frac{a(\|x_i(n) - x_j(n) \|)}{\sum_{k=1}^{N} a( \|x_i(n) - x_k(n) \|)}.
\end{equation}
Then, it is easy to see that 
\[ \sum_{j=1}^{N} {\phi}_{ij}(n)=1 \quad \mbox{and} \quad  \sum_{j \neq i} \phi_{ij}(n) \leq 1. \] 
In this setting, the following are three questions that we address in this paper: 
\begin{itemize}
\item
(Flocking dynamics):~Does the proposed discrete model \eqref{eq: discrete_MT} exhibit emergent behaviors?  More precisely, what conditions for system parameters and initial data are needed to show the emergent flocking dynamics?
\vspace{0.2cm}
\item
(Uniform-in-time transition):  Formally, the discrete model \eqref{eq: discrete_MT} converges to the continuous model \eqref{eq: conti_MT} as the time-step tends to zero. Can we verify this smooth transition uniformly in time?
\vspace{0.2cm}
\item
(Uniform-in-time stability):~Can we bound the distance between two solutions of the discrete model \eqref{eq: discrete_MT} with that of two initial data vectors uniformly in time?
\end{itemize}
The main purpose of this paper is to provide affirmative answers to the aforementioned questions. Before we move on to the descriptions of the main results, we introduce several handy notations for the analysis. These will be used throughout the paper. For $\forall~i, j \in [N]$, we set 
\begin{equation}
\begin{cases} \label{New-1}
 \displaystyle x_i = (x_i^1, \ldots, x_i^d), \quad v_i = (v_i^1, \ldots, v_i^d), \quad X:= \{ x_i \}_{i=1}^N, \quad V:= \{v_i \}_{i=1}^N, \\
 \displaystyle \|x_i \| := \Big( \sum_{k = 1}^d |x_i^k|^2 \Big)^{\frac{1}{2}}, \quad \|v_i \| := \Big( \sum_{k = 1}^d |v_i^k|^2 \Big)^{\frac{1}{2}}, \\
 \displaystyle  \Delta^x_{ij} := x_i - x_j,  \quad  \Delta^v_{ij} := v_i - v_j, \quad  \Delta^x := [\Delta^x_{ij}], \quad  \Delta^v := [\Delta^v_{ij}], \\
 \displaystyle  \|\Delta^x\|_F :=\sqrt{ \sum_{i,j=1}^{N} \|\Delta^x_{ij} \|^2}\,,\quad \|\Delta^v\|_F:=\sqrt{ \sum_{i,j=1}^{N} \|\Delta^v_{ij} \|^2}.
 \end{cases}
\end{equation}

The main results of this paper are three-fold. First, we show that the discrete model \eqref{eq: discrete_MT} exhibits a mono-cluster flocking under a suitable set of conditions using the nonlinear functional approach. More precisely, let $(X, V)=(X(n),V(n))_{n\in\N}$ be a global solution to \eqref{eq: discrete_MT}. Then,  the pair of scalar-valued functionals $( \|\Delta^x(n)\|_F, \|\Delta^v(n)\|^2_F)$ satisfies the following set of recursive inequalities (see \cref{prop: delta_xv_recurrence}):
 \[
    \begin{cases} 
    \displaystyle \|\Delta^x(n+1)\|_F \leq \|\Delta^x(n)\|_F + h\|\Delta^v(n)\|_F\,, \\
    \displaystyle \|\Delta^v(n+1)\|_F \leq \left[  1- h\kappa \left( 1 - \| \phi \|_{\text{Lip}}N \|\Delta^x(n)\|_F \right)   \right ] \|\Delta^v(n)\|_F\,,
    \end{cases}
 \]
Based on these relations, the mono-cluster flocking estimate follows (see \cref{thm: flocking}): If  system parameters and initial data satisfy 
\[
\|\Delta^x(0)\|_F <  \frac{1}{4N \|\phi \|_{\text{Lip}}} =: M \quad \mbox{and} \quad  \|\Delta^v(0)\|_F < \kappa \int_{\|\Delta^x(0)\|_F}^M (1 - \| \phi \|_{\text{Lip}} N s) ds,
\]
then we have asymptotic mono-cluster flocking: there exists a positive constant $C$ such that 
\[ \sup_{0 \leq n <\infty}\|\Delta^x(n)\|_F \leq M \quad \mbox{and} \quad \|\Delta^v(n)\|_F\leq\|\Delta^v(0)\|_F  e^{-C \kappa (1 - \| \phi \|_{\text{Lip}} N M)nh}, \quad \forall~n \geq 0.
\]
\noindent Second, we are concerned with the uniform-in-time transition from the discrete model to the continuous model as $h \to 0$ using a finite-time convergence result and asymptotic flocking estimate. Let $(X, V)$ and $(X^h, V^h)$ be global solutions to the continuous and discrete models, respectively, with the same initial data. Then we have  for some constant $C>0$,
   \[ \limsup_{h \rightarrow 0} \sup_{0 \leq n < +\infty}  \|\Delta^{x,h}(n) - \Delta^x(nh)\|_F \le C\,,\quad
  \limsup_{h \rightarrow 0} \sup_{0 \leq n < +\infty}   \|V^h(n) - V(nh)\|  = 0\,,
 \]
 where $\Delta^{x,h}$ and $\Delta^{x}$ refer to the shape discrepancies in positions for the discrete and continuum evolutions, respectively.
We refer to \cref{T3.10} for details. \newline

\noindent Third, we present a version of uniform-in-time stability of the discrete solution under the following sufficient conditions on initial data, coupling strength, and communication weight function:
\[
\begin{cases}
\displaystyle \max\left\{\|\Delta^x(0)\|_F,\|\Delta^{\overline{x}}(0)\|_F\right\} < M,\\
\displaystyle \|\Delta^v(0)\|_F < \kappa \int_{\|\Delta^x(0)\|_F}^M \psi(s) ds, \quad  \|\Delta^{\overline{v}}(0)\|_F < \kappa \int_{\|\Delta^{\overline{x}}(0)\|_F}^M  \psi(s) ds.
\end{cases}
\]
Under the above conditions, we show that there exists a constant $C_0>0$ only depending on initial data such that for small $h>0$,
\begin{align}\label{eq:stability}
    &\left\|\Delta^x(n)-\Delta^{\overline{x}}(n)\right\|_F + \left\|\Delta^v(n)-\Delta^{\overline{v}}(n)\right\|_F\notag
    \\    &\hspace{2cm}
    \leq \left\|\Delta^x(0)-\Delta^{\overline{x}}(0)\right\|_F + \left\|\Delta^v(0)-\Delta^{\overline{v}}(0)\right\|_F + C(n,h)\,,
    \end{align}
where $\lim_{h\to0}C(n,h) = C_0$ for all $n\ge 0$. 
Throughout this paper, we refer to $\|\Delta^x(n)-\Delta^{\overline{x}}(n)\|_F$ and $\|\Delta^v(n)-\Delta^{\overline{v}}(n)\|_F$ as a \emph{shape discrepancy functional} (in position and velocity, respectively) of $(X(n),V(n))$ and $(\overline{X}(n), \overline{V}(n))$ since both measure how far the particles are from one another as an ensemble (in terms of position or velocity) at the $n$th step. If the shape discrepancy is zero, then two configurations coincide exactly, and shape discrepancy is invariant under rigid body motion. 
Note that \eqref{eq:stability} is different from what is usually referred to as a stability estimate as we are considering shape discrepancy rather than the difference between two solutions. This is analogous to orbital stability in hyperbolic conservation laws; see \cref{sec:4} for details. \\


The rest of this paper is organized as follows. In \cref{sec:2}, we briefly discuss motivations for the MT model and recall previous results on the emergent dynamics of the continuous MT model. In \cref{sec:3}, we study asymptotic mono-cluster flocking of the discrete MT model, and using this flocking estimate and finite-time results on the transition from the discrete model to the continuous model, we verify that the continuous transition can be made uniformly in time as the time-step tends to zero. In \cref{sec:4}, we present the rigorous statement and proof of the uniform-in-time stability result \eqref{eq:stability}.
In \cref{sec:5}, we provide several numerical examples for the discrete MT model to illustrate the analytical results obtained in previous sections. \cref{sec:6} concludes with a brief summary of our main results. 

%
%
%
%
\section{Preliminaries} \label{sec:2}
\setcounter{equation}{0}
In this section, we briefly review how the MT model \eqref{eq: conti_MT} can be motivated from the Cucker-Smale model \eqref{eq: CS}, and summarize previous results on the emergent behaviors of the model \eqref{eq: conti_MT}. 

\subsection{From the CS model to the MT model} \label{sec:2.1}
In this subsection, we discuss the motivation for the normalized communication weights in the MT model. For this, we begin with the CS model:
\begin{equation} 
\begin{cases} \label{B-0}
 \displaystyle {\dot x}_i = v_i, \quad i \in [N], \\
 \vspace{0.2cm}
 \displaystyle {\dot v}_i =  \kappa \sum_{j = 1}^N \frac{a_{ij}}{N} (v_j - v_i),
\end{cases}
\end{equation}
where $a_{ij}$ is the pairwise communication weight depending only on the distance between particles $x_i$ and $x_j$:
\[ a_{ij} := a( \|x_i - x_j \|), \quad \forall~i, j \in [N]. \]
In what follows, we consider a situation in which two groups of agents are separated far apart from each other, and the number of agents occupying the region of each flocking group varies widely in size, one being significantly larger than the other. In this case, the communication weight between two agents, one in the smaller group and another in the larger group will be close to zero. Moreover, for agent $i$ in the smaller group, we can observe
 \[
{\dot v}_i = \frac{\kappa}{N} \sum_{j = 1}^N a_{ij} (v_j - v_i) \approx \frac{\kappa}{N} \sum_{j \in \text{ smaller group}} a_{ij} (v_j - v_i)\,,
\]
where $N$ is the total system size which is the sum of the sizes of the smaller group and the larger group. In the case of $N \gg 1$, the dynamics for such particles (agents) in a smaller group will be described as being almost static $\frac{dv_i}{dt}  \approx 0$. Thus, the model  \eqref{eq: CS} does not capture the actual phenomenon accurately. To overcome this limitation in  \eqref{eq: CS}, Motsch and Tadmor proposed a slightly modified variant \eqref{eq: conti_MT}, which we will refer to as the Motsch-Tadmor (MT) model throughout this paper. By normalizing the communication weight $a_{ij}$, the new resulting model can now describe the above setting more accurately compared to the CS model. 
As a trade-off, the communication weight matrix $(\phi_{ij})$ loses the symmetry. In the MT model, the symmetric communication weight $a_{ij}$ is now replaced by the normalized one
\[ \phi_{ij} := \frac{a_{ij}}{\sum_{k = 1}^N a_{ik}}, \quad \forall~ i, j \in [N]\,, \]
leading to the continuous MT model:
\begin{equation} 
\begin{cases} \label{B-0-0-0}
 \displaystyle {\dot x}_i = v_i, \quad i \in [N]\,,\\
 \vspace{0.2cm}
 \displaystyle {\dot v}_i =  \kappa \sum_{j = 1}^N \phi_{ij}(v_j - v_i)\,.
\end{cases}
\end{equation}
Now, we apply the first-order forward Euler scheme to \eqref{B-0-0-0} with $h = \Delta t$: for $t = n h$
\[ 
( {\dot x}_i(t), {\dot v}_i(t)) \quad \Longrightarrow \quad \Big( \frac{x_i(n+1) - x_i(n)}{h}, \frac{v_i(n+1) - v_i(n)}{h} \Big)
\]
to derive the discrete counterpart. Note that the continuous and discrete models  \eqref{eq: conti_MT} and \eqref{eq: discrete_MT} do not have obvious conserved quantities, except the total number of particles. This lack of conservation laws causes lots of technical difficulties in the analysis. \newline

Next, we explain why the rigorous flocking analysis for the Cucker-Smale model based on energy estimates is difficult to translate to the discrete and continuuum Motsch-Tadmor models.  To see this, we first begin with the flocking analysis for the CS model. We sum up $\eqref{B-0}_2$ over all $i \in [N]$, and use the index exchange transformation $i~\leftrightarrow~j$ to find 
\begin{equation} \label{B-0-0-1}
\frac{d}{dt} \sum_{i=1}^{N} v_i  =\frac{\kappa}{2N} \sum_{i, j = 1}^N (a_{ij} - a_{ji}) (v_j - v_i).
\end{equation} 
Then, we use the symmetry $a_{ij} = a_{ji}$ to derive the conservation of total momentum:
\begin{equation} \label{B-0-0-2}
\frac{d}{dt} \sum_{i=1}^{N} v_i = 0, \quad t >0.
\end{equation}
Secondly, one can take the inner product $\eqref{B-0}_2$ with $2v_i$ and sum up the resulting relation over $i$ using the symmetry $a_{ij} = a_{ji}$ to derive an energy dissipation estimate:
\begin{equation} \label{B-0-1}
\frac{d}{dt} \Big( \sum_{i=1}^{N} \|v_i \|^2 \Big) = -\frac{\kappa}{N} \sum_{i, j = 1}^N a_{ij} \|v_j - v_i \|^2 \leq 0. 
\end{equation}
These estimates imply exponential velocity alignment for the fluctuation around the average initial velocity under suitable conditions on $a$ and the initial configuration: there exists a positive constant $\Lambda$ such that 
\begin{equation} \label{B-0-2}
\Big \|  v_i(t) - \frac{1}{N} \sum_{j=1}^{N} v_j(0) \Big\| \lesssim e^{-\Lambda t} \Big \|  v_i(0) - \frac{1}{N} \sum_{j=1}^{N} v_j(0) \Big\|, \quad \mbox{as $t \to \infty$}, \quad \forall~i \in [N].
\end{equation}
For details on how to derive this exponential decay from \eqref{B-0-1} see
\cite{Seung-Yeal-Ha-2008-Kinetic-and-Related-Models}.
Obviously, this exponential velocity alignment also implies spatial cohesiveness:
\begin{equation} \label{B-0-3}
\sup_{0 \leq t < \infty} \Big \|x_i(t) - \frac{1}{N} \sum_{j=1}^{N} x_j(0) \Big\| < \infty, \quad \forall~i \in [N].
\end{equation}
Next, we return to the MT model \eqref{B-0-0-0}.  Unlike the CS model, the MT model does not preserve the total momentum due to the non-symmetric property of $\phi$: in general,
\[ \frac{d}{dt} \sum_{i=1}^{N} v_i(t) \neq 0, \quad t > 0.   \]
Hence, these energy estimates cannot be used in the MT model. This is one of the main technical difficulties for flocking analysis. Before we close this subsection, we list several assumptions on $a$ and the time-step $h$ used throughout the paper: 
\begin{assumption}\label{as: a}  
Suppose that the unnormalized communication weight $a(r)$ and the time-step $h$ satisfy the following structural conditions:
\begin{enumerate} 
    \item There exists positive constants $c_1, c_2$ such that
    \begin{equation}\label{eq: phi_bound}
        c_1 \leq a(r) \leq c_2, \quad \forall~r \geq 0. 
    \end{equation}
    \item There exists a positive Lipschitz constant $L_{a}$ such that
    \begin{equation}\label{eq: phi_lip}
        |a(r_1) -a(r_2) | \leq L_{a}|r_1 - r_2|, \quad r_1, r_2 \geq 0.
    \end{equation}
     \item The time-step $h$ is sufficiently small so that 
     \begin{equation*}
        0 < h < \min \Big \{1, \frac{1}{\kappa} \Big \}
        ,
    \end{equation*}
    where $\kappa$ is the coupling strength in \eqref{eq: discrete_MT}.
\end{enumerate}
\end{assumption}
\begin{remark}\label{rmk: as_a}
Below, we comment on the conditions in Assumption~\ref{as: a} and put them in context to the current state-of-the-art. Conditions (1) and (2) can be compared to those in \cite{Ha_Zhang_CS_continuous_transition} for the original Cucker-Smale model. More precisely, in \cite{Ha_Zhang_CS_continuous_transition} the authors assumed $0< a(r) \leq 1$ in order to show a uniform-in-time transition from the discrete to continuum Cucker-Smale model. In \cite{Ha2009EmergenceOT}, which is concerned with a stochastic version of the Cucker-Smale model, $a(r)$ is just assumed to be non-negative, non-increasing, and radially symmetric. 
Further, monotonicity (or being non-increasing) is a common assumption on $a(r)$ \cite{Motsch_2011, MT14, Ha2009EmergenceOT}. We can think of monotonicity as a special case of almost everywhere Lipschitz continuity. This is because a monotonic function is differentiable almost everywhere with a derivative bounded almost everywhere over a bounded domain; hence in the setting where the spatial diameter is bounded (as we will show in our flocking result), a monotone $a(r)$ can be considered Lipschitz almost everywhere. Condition (2) is also required to guarantee the well-posedness of the model and can be considered a standard assumption \cite{Amann+1990, doi:10.1137/1009057, doi:10.1137/1.9780898719222}.
In some sense, Condition (1) can be viewed as corresponding to the variable all-to-all communication weight case which is always long-ranged. Whether we can relax Condition (1) as is the case for the original CS model remains a challenging open problem.
Condition (3) is not very restrictive, as it is always possible to take small enough time steps $h$.
\end{remark}

%
%
%
\subsection{Previous results}\label{sec:2.2}
In this subsection, we review the emergent dynamics of the continuous MT model \eqref{B-0-0-0} for comparison with the new flocking estimate for the discrete MT model in the next section. Although the energy estimates for the MT model cannot be employed due to a lack of conservation laws as explained in the previous subsection, we can still use the nonlinear functional approach as applied in \cite{H-Liu}.

For a given configuration $\{(x_i, v_i) \}$, we define spatial and velocity diameters:
\[ {\mathcal D}(X):= \max_{1 \leq i, j \leq N} \|x_i - x_j\|, \quad {\mathcal D}(V):=  \max_{1 \leq i, j \leq N} \|v_i - v_j \|. \]
Next, we recall the concept of the mono-cluster flocking as follows.
\begin{definition}
Let ${\mathcal Z} := \{(x_i(t), v_i(t) \}_{i = 1}^N$ be a global solution to \eqref{B-0-0-0}. Then, configuration ${\mathcal Z}$ exhibits an asymptotic mono-cluster flocking if and only if the following two conditions hold.
\begin{equation} \label{B-0-4}
 \sup_{0 \leq t < \infty} {\mathcal D}(X(t)) < \infty \quad \mbox{and} \quad   \lim_{t \rightarrow \infty} {\mathcal D}(V(t)) =0.
 \end{equation}
\end{definition}
By following calculations in \cite{Motsch_2011, MT14},
we can show these diameters satisfy a system of dissipative differential inequalities (SDDI):
\begin{equation}
\begin{cases} \label{B-0-5}
\displaystyle \Big| \frac{d}{dt} {\mathcal D}(X(t)) \Big| \leq  {\mathcal D}(V(t)),  \quad t > 0, \\
\vspace{0.2cm}
\displaystyle \frac{d}{dt} {\mathcal D}(V(t)) \leq - \kappa a( {\mathcal D}(X(t))) {\mathcal D}(V(t)).
 \end{cases}
 \end{equation}
 Note that once we can show the uniform boundedness of ${\mathcal D}(X)$, the differential inequality $\eqref{B-0-5}_2$ yields the exponential decay of  ${\mathcal D}(V)$. To derive the estimates \eqref{B-0-3} for the continuous MT model, we introduce nonlinear functionals ${\mathcal H}^{\pm}$:
 \begin{equation} \label{eq: lyapunov_H_MT}  
 {\mathcal H}^{\pm}(t) := {\mathcal D}(V(t)) \pm \kappa \int_{ 0}^{{\mathcal D}(X(t))} a(\eta) d\eta, \quad t \geq 0,
  \end{equation}
  where we used the following notation for initial configurations:
  \[ X^0 := X(0), \quad  V^0 := V(0). \]
 Since diameter functionals are Lipschitz continuous, $\mathcal{H}^{\pm}$ is Lipschitz continuous as well. Hence, they are almost differentiable and satisfy the stability estimate:
 \[
 \mathcal{H}^{\pm}(t) \leq \mathcal{H}^{\pm}(0), \quad t \geq 0,
 \]
which implies (following a similar argument as in \cite[Lemma 3.1]{H-Liu}):
 \begin{equation*} \label{B-0-6}
 {\mathcal D}(V(t)) +  \kappa \int_0^{{\mathcal D}(X(t))} a(\eta) d\eta   \leq {\mathcal D}(V^0), \quad t \geq 0.
 \end{equation*}

 Now we recall the results in \cite{LX, Motsch_2011, MT14} on the emergence of mono-cluster flocking. 
\begin{theorem}
\emph{ \cite{LX, Motsch_2011,MT14}} \label{thm: conti_flocking}
Suppose that initial data and system parameters satisfy either one of the following conditions: 
\begin{enumerate}
\item
\mbox{either} 
\begin{equation} \label{B-1}
 {\mathcal D}(V^0) \leq \kappa \int_{{\mathcal D}(X^0) }^{\infty} a(s)ds, 
 \end{equation}
\item
\mbox{or}
\begin{equation} \label{B-2}
a(s) = (1+s)^{-\beta}, ~ \beta \geq 0 \quad \mbox{and} \quad {\mathcal D}(V^0) \leq \frac{\kappa}{N} \int_{{\mathcal D}(X^0)}^{\infty} \phi(s)ds,
\end{equation}
\end{enumerate}
and let $\{ (x_{i}, v_{i}) \}$ be a global solution of \eqref{B-0-0-0}. Then, mono-cluster flocking emerges asymptotically, i.e., there exists ${\mathcal D}^{\infty}$ such that 
\begin{equation} \label{B-3}
\sup_{0 \leq t < \infty}  {\mathcal D}(X(t)) \leq {\mathcal D}^{\infty}, \quad {\mathcal D}(V(t)) \leq {\mathcal D}(V^0) \exp \Big( -\kappa a({\mathcal D}^{\infty}) t \Big), \quad \forall~t \geq 0. 
\end{equation}
\end{theorem}
\begin{corollary}\label{cor: continuous_v_limit}
    Under the same setting as in \cref{thm: conti_flocking}, there exists a set of asymptotic velocities $\{v_i^{\infty} \}$ such that 
  \[  v_i^{\infty} := \lim_{t \rightarrow \infty} v_i(t) \quad \text{ with } \quad \|v_i^\infty\|< \infty, \quad \forall~i \in [N]\,,\]
  where \begin{align}\label{eq: conti_velocity}
      \| v_i(t)-v_i^{\infty}  \| \leq \frac{N \mathcal{D}(V(0))}{ a(\mathcal{D}^{\infty})}\, e^{- \kappa a(\mathcal{D}^{\infty}) t}\,.
  \end{align}
\end{corollary}
\begin{proof}
Note that for all $i \in [N]$, $v_i(t)$ satisfies 
\begin{align*}
    v_i(t) &= v_i(0) + \kappa \sum_{j = 1}^N \int_0^t \phi_{ij}(s) (v_j(s) - v_i(s)) ds \,,
\end{align*}
where $\phi_{ij}(s):= \frac{a_{ij}(\|x_i(s) - x_j(s) \|)}{\sum_{k = 1}^N a_{ik}(\|x_i(s) - x_j(s) \|)}.$ 
Since $\phi_{ij} \leq 1$, we have 
\begin{align}\label{eq: vlimit_conti}
    \|v_i(t)\| & \leq  \|v_i(0)\| + \kappa \sum_{j = 1}^N \int_0^{t}|\phi_{ij}(s)| \left\| v_j(s) - v_i(s) \right\| ds \notag\\
    & \leq \|v_i(0)\| + \kappa \sum_{j = 1}^N \int_0^{t}  \mathcal{D}(V(s)) ds\,.
\end{align}
Now, we set
\begin{align*}
    v_i^\infty &:= v_i(0) + \kappa \sum_{j = 1}^N \int_0^\infty \phi_{ij}(s) (v_j(s) - v_i(s)) ds\,.
\end{align*}
 By \cref{thm: conti_flocking}, $\mathcal{D}(V(s))$ decays exponentially to zero, and so $v_i^\infty$ above is well-defined. Indeed, note that
    \begin{align*}
    \begin{aligned}
      \|v_i^{\infty}\| &\leq \|v_i(0)\| + \kappa \sum_{j = 1}^N \int_0^{\infty} \mathcal{D}(V(s)) ds \\
      &\leq \|v_i(0)\| + \kappa \sum_{j = 1}^N \int_0^{\infty} \mathcal{D}(V(0)) \exp\left( -\kappa a(\mathcal{D}^{\infty}) s\right) ds \\
      & \leq \|v_i(0)\| + \kappa \sum_{j = 1}^N  \frac{\mathcal{D}(V(0))}{\kappa a (\mathcal{D}^{\infty})}  = \|v_i(0)\| + \frac{N \mathcal{D}(V(0))}{a (\mathcal{D}^{\infty})}  \,.
     \end{aligned}
    \end{align*}
Furthermore, \begin{align*}
    \left\|v_i^{\infty} - v_i(t) \right\| &= \left| \kappa \sum_{j = 1}^N \int_t^{\infty} \phi_{ij} (v_j(s) - v_i(s)) ds\right| 
    \leq \kappa \sum_{j = 1}^{N} \int_t^{\infty}\mathcal{D}(V(s)) ds \\
    &\leq \kappa \sum_{j = 1}^{N} \int_t^{\infty}\mathcal{D}(V(0)) \exp(-\kappa a(\mathcal{D}^{\infty}s)) ds  \leq  \frac{N \mathcal{D}(V(0)) \exp(- \kappa a(\mathcal{D}^{\infty}) t)}{ a(\mathcal{D}^{\infty})}\,,
\end{align*}
and so the result follows.
\end{proof}

\begin{remark}\label{rmk: enhanced_conti_flocking}
    Estimations in \cref{thm: conti_flocking}, and \cref{cor: continuous_v_limit} follow \cite{MT14}, which are improved versions from the original estimates with $a^2$ instead of $a$ from \cite{Motsch_2011}.
\end{remark}
\begin{remark}
    Thanks to the exponential decay of ${\mathcal D}(V(t))$ derived in Theorem~\ref{thm: conti_flocking}, the asymptotic velocities are well-defined, and they are the same:
\[ v_i^{\infty} = v_j^{\infty}, \quad i \neq j. \]
Note that asymptotic velocities depend on the whole particle trajectory of the initial configuration, unlike the CS model in which the asymptotic flocking velocity is completely determined by the average initial velocity. 
\end{remark}
%
%
%
%
\section{Mono-cluster flocking and uniform-in-time transition} \label{sec:3}
\setcounter{equation}{0}
In this section, we study the emergence of mono-cluster flocking for \eqref{eq: discrete_MT} and use this flocking estimate to derive the uniform-in-time transition from the discrete model to the continuous one as the time-step $h$ tends to zero.
\subsection{Preparatory lemmas} \label{sec:3.1} 
In this subsection, we derive several lemmas to be used in the proof of mono-cluster flocking.
\begin{lemma} \label{lem:a_lip}
Let ${\phi}_{ij}$ be the communication weight defined in \eqref{New-1-1}. Then
\begin{equation}\label{eq: a_lip}
    |{\phi}_{il} - {\phi}_{jl}| \leq \|{\phi}\|_{\text{Lip}} \|x_i - x_j \|, \quad \forall~i, j, l \in [N],
\end{equation}
where $\| {\phi} \|_{\text{Lip}}:= \frac{L_{a}}{N c_1}\left( 1 + \frac{c_2}{c_1}\right)$ depends on the constants from \cref{as: a}.
\end{lemma}
\begin{proof} We use \eqref{New-1-1} and \eqref{eq: phi_lip} to find 
\begin{align*} 
\begin{aligned}
 |{\phi}_{il} - {\phi}_{jl}| &= \left| \frac{a_{il} \sum_{k=1}^N a_{jk} - a_{jl} \sum_{k=1}^N a_{ik}}{\sum_{k=1}^{N} a_{ik} \sum_{k=1}^N a_{jk}} \right|  \\ 
 &= \left| 
 \frac{a_{il} \sum_{k=1}^N a_{jk} - a_{jl}\sum_{k=1}^N a_{jk} + a_{jl}(\sum_{k=1}^N a_{jk} - \sum_{k=1}^N a_{ik})}{\sum_{k=1}^{N} a_{ik} \sum_{k=1}^N a_{jk}}\right| \\ 
 &\leq \frac{L_{a} \|x_i - x_j \|}{N c_1} + \frac{N|\phi_{jl}| \|x_i - x_j \|L_{a}}{N^2 c_1^2} \\ &\leq \left( \frac{L_{a}}{N c_1} + \frac{L_{a} c_2}{N c_1^2} \right) \|x_i - x_j \| 
 := \| \phi \|_{\text{Lip}} \|x_i - x_j \|.
\end{aligned}
\end{align*}
\end{proof}
\begin{lemma}\label{lem: triple_delta} 
For $X = \{ x_i \} $ and $ V = \{ v_i \}$, we have 
\[
 \sum_{1 \leq i, j, l  \leq N} \|\Delta_{li}^v \| \cdot \|\Delta_{ij}^v \| \cdot \|\Delta_{ij}^x \| \leq N  \|\Delta^x\|_F \|\Delta^v\|_F^2\,.
\]
\end{lemma}
\begin{proof} We use \eqref{New-1} to get 
  \begin{align*}
  \begin{aligned}
  & \sum_{1\leq  i, j, l \leq N} \|\Delta_{li}^v \| \|\Delta_{ij}^v\| \| \Delta_{ij}^x \| \\
  & \hspace{1cm} \leq \left[ \sum_{i=1}^{N} \left( \sum_{l=1}^{N} |\Delta_{li}^v|\right)^2 \left( \sum_{j =1}^{N} |\Delta_{ij}^v||\Delta_{ij}^x| \right)^2 \right]^{\frac{1}{2}} \leq \left[ \sum_{i=1}^{N} \left( N\sum_{l=1}^{N} |\Delta_{li}^v|^2 \right) \left( N\sum_{j=1}^{N} |\Delta_{ij}^v|^2 |\Delta_{ij}^x|^2 \right)  \right]^{\frac{1}{2}} \\
  & \hspace{1cm}  \leq N \left[ \left( \sum_{1 \leq i, l \leq N}  |\Delta_{li}^v|^2 \right) \left( \sum_{1 \leq i, j \leq N}  |\Delta_{ij}^v|^2 \sum_{1 \leq r, s \leq N} |\Delta_{rs}^x|^2 \right)  \right]^{\frac{1}{2}} = N  \|\Delta^x\|_F \|\Delta^v\|_F^2\,.
    \end{aligned}
    \end{align*} 
    Here, we used the H\"older inequality and Jensen's inequality for the first and second inequalities. In the third inequality, we use the following inequality:
     \begin{equation*}
        \sum_{1 \leq i, j \leq N} a_{ij}b_{ij} \leq \sum_{1 \leq i, j \leq N} a_{ij} \sum_{1 \leq r, s \leq N} b_{rs}.
    \end{equation*}
\end{proof}
\begin{remark} By the same argument as in the proof of \cref{lem: triple_delta},  we also have
    \begin{equation*}
        \sum_{1 \leq i, j \leq N} \|\Delta^v_{ij}\|^2 \|\Delta^x_{ij}\|^2\leq \|\Delta^v\|_F^2 \|\Delta^x\|_\infty^2\leq \|\Delta^v\|_F^2 \|\Delta^x\|_F^2\,.
    \end{equation*}
\end{remark}
Next, we provide recursive relations for $\| \Delta^x(n) \|_F$ and $\| \Delta^v(n) \|_F$.
\begin{proposition}\label{prop: delta_xv_recurrence}
    Let $(X, V)$ be a global solution to \eqref{eq: discrete_MT}. Then, we have 
    \begin{align}
    \begin{aligned} \label{New-2}
        &\|\Delta^x(n+1)\|_F \leq \|\Delta^x(n)\|_F + h\|\Delta^v(n)\|_F\,, \quad n \geq 0, \\
        &\|\Delta^v(n+1)\|_F \leq \left[  1- h\kappa \left( 1 - \| \phi \|_{\text{Lip}}N \|\Delta^x(n)\|_F \right)   \right ] \|\Delta^v(n)\|_F.
        \end{aligned}
    \end{align}
\end{proposition}
\begin{proof}
\noindent (i)~It follows from \eqref{eq: discrete_MT} that 
\[  x_i(n+1) = x_i(n)+hv_i(n)\,, \quad  x_j(n+1) = x_j(n)+hv_j(n)\,. \]
These imply
\[
\Delta_{ij}^x(n+1) = \Delta_{ij}^x(n) + h \Delta_{ij}^v(n)\,,
\]
and so
\begin{align}
\begin{aligned} \label{New-3}
\|\Delta_{ij}^x(n+1) \|^2 &= \|\Delta_{ij}^x(n) \|^2 + h^2 \|\Delta_{ij}^v(n) \|^2 + 2 h \Delta_{ij}^x(n) \cdot  \Delta_{ij}^v(n) \\
&\leq  \|\Delta_{ij}^x(n) \|^2 + h^2 \|\Delta_{ij}^v(n) \|^2 + 2 h  \|\Delta_{ij}^x(n) \| \cdot \|\Delta_{ij}^v(n) \|.
\end{aligned}
\end{align}
We sum up \eqref{New-3} over all $i, j \in [N]$ to derive 
 \begin{align*}
 \begin{aligned}
        \|\Delta^x(n+1)\|^2_F &= \sum_{i, j =1}^{N}  \Big( \|\Delta^x_{ij}(n) \|^2 + h^2 \|\Delta^v_{ij}(n) \|^2 + 2h \|\Delta^x_{ij}(n) \| \cdot \|\Delta^v_{ij}(n) \|  \Big) \\ 
        & \leq \sum_{i, j =1}^{N} \left( \|\Delta^x_{ij}(n) \|^2 + h^2 \|\Delta^v_{ij}(n) \|^2 + 2h  \sqrt{\sum_{i, j = 1}^{N} \|\Delta^x_{ij}(n) \|^2 \sum_{r, s = 1}^{N} \|\Delta^v_{rs}(n) \|^2 }  \right) \\ 
        & = \left( \sqrt{\sum_{i, j =1}^{N} \| \Delta_{ij}^x(n) \|^2} + h  \sqrt{\sum_{r, s = 1}^{N} \| \Delta_{rs}^v(n) \|^2} \right)^2 
        = \Big(\|\Delta^x(n)\|_F + h \|\Delta^v(n)\|_F \Big)^2.
   \end{aligned}
    \end{align*}
 This yields the desired estimate for $\eqref{New-2}_1$. \newline
     
\noindent (ii)~It follows from \eqref{eq: discrete_MT} that 
\begin{align*}
\begin{aligned}
v_i(n+1) &= v_i(n)+h\kappa \sum_{l=1}^{N} {\phi}_{il} (n) \big(v_l(n)-v_i(n)\big), \\
v_j(n+1) &= v_j(n)+h\kappa\sum_{l=1}^{N} {\phi}_{jl}(n)  \big(v_l(n)-v_j(n)\big).
\end{aligned}
\end{align*}
These yields
\begin{align*}
\begin{aligned}
       & \Delta_{ij}^v(n+1) 
       = v_i(n) - v_j(n) + h\kappa\sum_{l =1}^N \Big( \phi_{il}(n)(v_l(n) - v_i(n)) - \phi_{jl}(n)(v_l(n) - v_j(n))\Big) \\ 
        & \hspace{0.7cm} = v_i(n) - v_j(n) + h\kappa\sum_{l =1}^N \Big[ 
         \phi_{il}(n)(v_l(n) - v_i(n)) - \phi_{jl}(n) \Big(v_l(n) - v_i(n) + v_i(n) - v_j(n) \Big) \Big] \\ 
        & \hspace{0.7cm} = (1 - h\kappa) (v_i(n) - v_j(n)) + h\kappa\sum_{l = 1}^N (\phi_{il}(n) - \phi_{jl}(n))(v_l(n) - v_i(n)) \\
        & \hspace{0.7cm} =   (1 - h\kappa) \Delta_{ij}^v(n) + h\kappa\sum_{l = 1}^N  (\phi_{il}(n) - \phi_{jl}(n) ) \Delta^v_{li}(n).
    \end{aligned}
    \end{align*}
    Here in the last equality, we used the unit sum relation $\sum_{l=1}^N \phi_{jl} = 1$. 
  Then we have 
  \begin{align} 
  &\|\Delta_{ij}^v(n+1) \|^2 \notag\\
  &= (1 - h\kappa)^2 \|\Delta_{ij}^v(n) \|^2 + 2h\kappa(1 - h\kappa)\sum_{l = 1}^N(\phi_{il}(n) - \phi_{jl}(n))\cdot \Delta_{li}^v(n) \cdot \Delta_{ij}^v(n) \notag\\
  &\qquad + h^2\kappa^2\left \| \sum_{l = 1}^N ( \phi_{il}(n) - \phi_{jl}(n)) \Delta_{li}^v(n) \right \|^2 \notag\\ 
  &\leq(1 - h\kappa)^2 \|\Delta_{ij}^v(n) \|^2 + 2h\kappa(1 - h\kappa)\sum_{l = 1}^N \|\phi \|_{\text{Lip}}\cdot \|\Delta_{ij}^x(n)\| \cdot \|\Delta_{li}^v(n) \| \cdot  \|\Delta_{ij}^v(n) \|  \notag\\
  &\qquad+ h^2\kappa^2\left[ \sum_{l = 1}^N \|\phi\|_{\text{Lip}} \cdot\|\Delta_{ij}^x(n) \| \cdot \|\Delta_{li}^v(n) \| \right]^2\,,\label{New-4}
    \end{align}
    where we used the Lipschitz property of $\phi$ from \cref{lem:a_lip}. \newline
   
  \noindent  Next, we sum up \eqref{New-4}  over $i, j$, and use \cref{lem: triple_delta} for the second term to find 
  \begin{align*}
        &\|\Delta^v(n+1)\|_F^2\\
        &\leq ( 1- h\kappa)^2\|\Delta^v(n)\|_F^2 + 2h\kappa N(1 - h\kappa) \| \phi \|_{\text{Lip}} \|\Delta^x(n)\|_F \|\Delta^v(n)\|_F^2 
        \\ & \qquad 
        + h^2 \kappa^2 \sum_{i,  j = 1}^n \|\phi \|_{\text{Lip}}^2 \| \Delta^x_{i, j}(n) \|^2 \Big \| \sum_{l = 1}^N  \Delta^v_{li}(n)\Big \|^2  \\ & \leq \|\Delta^v(n)\|_F^2 \Big( ( 1- h\kappa)^2 +   2h\kappa N(1 - h\kappa)  \|\phi \|_{\text{Lip}} \|\Delta^x(n)\|_F 
        + h^2 \kappa^2 N^2 \|\phi \|_{\text{Lip}}^2 \|\Delta^x(n)\|_F^2 \Big) \\ &= \Big[ 1- h\kappa \left( 1 - \|\phi\|_{\text{Lip}}N \|\Delta^x(n)\|_F \right) \Big]^2 \|\Delta^v(n)\|_F^2 \,,
    \end{align*}
    where we used the following relation by the Jensen inequality: 
    \begin{equation*}
         \sum_{i,  j = 1}^N \| \Delta^x_{ij}(n) \|^2 \Big \| \sum_{l = 1}^N  \Delta^v_{li}(n) \Big \|^2  \leq N \|\Delta^x(n)\|_F^2 \|\Delta^v(n)\|_F^2\,.
    \end{equation*}
\end{proof}

\subsection{Asymptotic flocking dynamics} \label{sec:3.2}
In this subsection, we develop a suitable framework leading to mono-cluster flocking. First, we set 
\[ M:= \frac{1}{4N \|\phi \|_{\text{Lip}}} \quad \mbox{and} \quad   \psi(s):= 1 - \| \phi \|_{\text{Lip}} N s.  \]

\begin{theorem} \label{thm: flocking}
Suppose that initial data and coupling strength satisfy 
\begin{equation} \label{eq: initial}
\|\Delta^x(0)\|_F < M \quad \mbox{and} \quad  \|\Delta^v(0)\|_F < \kappa \int_{\|\Delta^x(0)\|_F}^M \psi(s) ds,
\end{equation}
and let $(X, V) \in \R^{2dN}$ be a global solution to \eqref{eq: discrete_MT}. Then
\begin{enumerate}
    \item[(i)] it holds
    \begin{equation} \label{eq: flocking_x}
\sup_{0 \leq n < \infty} \|\Delta^x(n)\|_F \leq M\,,
\end{equation}
    \item[(ii)] and for any constant $0<C<1$ there exists a sufficiently small $h>0$ such that 
    \begin{equation} \label{eq: discrete_v_flocking} 
\|\Delta^v(n)\|_F\leq\|\Delta^v(0)\|_F  e^{-C \kappa \psi(M)nh}, \quad \forall~n \geq 0.
\end{equation}
\end{enumerate}
\end{theorem}
\begin{remark}
       Our estimates for the discrete MT model only provide a conditional flocking guarantee, and it cannot be extended to an unconditional flocking theorem with our current proof techniques. This is different from the results for the continuous MT model. In \cite[Proposition 2.9]{MT14}, an unconditional flocking result is presented. 
       Numerically, we will indeed observe unconditional flocking, see \cref{sec:5}. Generalizing our results to obtain unconditional flocking is an interesting open question.
\end{remark}
\begin{proof}[Proof of \cref{thm: flocking}]
If $\|\Delta^v(0) \|_F=0$ then (i) and (ii) are trivially true since \eqref{New-2} guarantees $\|\Delta^v(n) \|_F=0$ and $\|\Delta^x(n+1) \|_F\le \|\Delta^n(n) \|_F$ for all $n\ge 0$. Therefore from now on we assume  $\|\Delta^v(0) \|_F>0$. Showing (i) is the most involved step; once established, (ii) follows as a direct consequence. \\

\noindent (i)~(Uniform bound in position shape discrepancy):
Suppose the contrary holds, i.e., there exists a positive constant $n^\infty = n^{\infty}(M) > 0$ such that
    \begin{equation}\label{eq: assump}
        \|\Delta^x(n) \|_F \leq M \quad \forall~n < n ^{\infty} \quad \text{ and } \quad \|\Delta^x(n ^{\infty}) \|_F > M.
    \end{equation} 
 By assumption $\|\Delta^x(0)\|_F < M$ in \eqref{eq: initial}, and so
 $n^{\infty} > 0$. From $\eqref{New-2}_2$ we have
    \begin{equation}\label{eq: v_recurrence_1}
        \|\Delta^v(n+1)\|_F - \|\Delta^v(n)\|_F \leq - h\kappa \psi(\|\Delta^x(n)\|_F) \|\Delta^v(n)\| \le 0 \,.
    \end{equation}
Summing \eqref{eq: v_recurrence_1}  over all $n = 0, \cdots, n^{\infty}-1$ and using $\eqref{New-2}_1$ yields 
     \begin{align} 
     \begin{aligned} \label{eq: n_inf}
      \|\Delta^v(n^{\infty})\|_F - \|\Delta^v(0)\|_F 
      &\leq - \sum_{n = 0}^{n^{\infty} - 1} h\kappa \|\Delta^v(n)\|_F  \psi(\|\Delta^x(n)\|_F)  \\ 
      &\leq - \sum_{n = 0}^{n^{\infty }- 1} \kappa \left( \|\Delta^x(n+1)\|_F - \|\Delta^x(n)\|_F\right)  \psi(\|\Delta^x(n)\|_F) \,.
    \end{aligned} 
    \end{align} 
Here, we used that \eqref{eq: assump} guarantees $\psi(\|\Delta^x(n)\|_F) > 0$ for all $n = 0, \cdots, n^{\infty} - 1$.
Since $\|\Delta^x(n)\|_F$ is not monotonic with respect to $n$, we reindex the set 
\[ \{\|\Delta^x(n)\|_F:\|\Delta^x(n)\|_F\geq \|\Delta^x(0)\|_F\}_{n=0}^{n^\infty} \]
into $\{ \|\Delta^x_q \|_F\}_{q = 0}^{K}$ so that it satisfies the following indexing rule: 
\[ \|\Delta^x(0)\|_F = \|\Delta^x_0\|_F, \quad \|\Delta^x(n^{\infty})\|_F = \|\Delta^x_K\|_F, \quad \|\Delta^x_q\|_F < \|\Delta^x_{q +1}\|_F. \]
If there exist  $n \neq m$ such that 
\[ \|\Delta^x(n)\|_F = \|\Delta^x(m)\|_F, \]
we set 
\[ \|\Delta^x_q\|_F = \|\Delta^x(n)\|_F = \|\Delta^x(m)\|_F \quad \mbox{for some}~~q \in [K]. \]
Hence, $\|\Delta^x(n)\|_F$ with the same value is relabelled and merged as sharing the same index. 

Note that the left-hand side of \eqref{eq: n_inf} is strictly negative. We would like to interpret the right-hand side of \eqref{eq: n_inf} an approximation of the integral\begin{equation*} \int_{\|\Delta^x(0)\|_F}^{\|\Delta^x(n^{\infty})\|_F} \kappa \psi(s) ds\,.
    \end{equation*}
    By \cref{lem:reindex} in \cref{sec:App-A}, we have 
    \begin{align}
    \begin{aligned} \label{eq: sum_ineq}
    & \sum_{q = 0}^{K- 1} \left( \|\Delta^x_{q +1} \|_F - \|\Delta^x_q\|_F\right) \kappa \psi(\|\Delta^x_q\|_F) \\
    & \hspace{1.5cm} \leq \sum_{n = 0}^{n^{\infty }- 1} \left( \|\Delta^x(n+1)\|_F - \|\Delta^x(n)\|_F\right) \kappa \psi(\|\Delta^x(n)\|_F)\,.
    \end{aligned}
    \end{align}
    where the right-hand side of \eqref{eq: sum_ineq} is strictly positive. Combining with \eqref{eq: n_inf}, we get 
    \begin{align}\label{eq: v_recurrence_2}
        \|\Delta^v(n^{\infty})\|_F - \|\Delta^v(0)\|_F & \leq - \sum_{q = 0}^{K - 1} (\|\Delta^x_{q + 1}\|_F - \|\Delta_q^x\|_F) \kappa \psi(\|\Delta_q^x\|_F)\,.
    \end{align}
    Moreover, it follows from \eqref{eq: v_recurrence_1} that
 \[ \|\Delta^v(n)\|_F \leq \|\Delta^v(0)\|_F \quad \mbox{for all}~ 0 \leq n < n^{\infty}\,, \]
 and so \eqref{New-2} implies
 \begin{equation}\label{eq: Dx_n_recurrence}
        \|\Delta^x(n+1)\|_F - \|\Delta^x(n)\|_F \leq \|\Delta^v(n)\|_F h \leq \|\Delta^v(0)\|_F h  \quad \mbox{for $0 \leq n < n ^{\infty}$}.
    \end{equation}
By condition~\eqref{eq: initial}, we have $\|\Delta^x(0)\|_F < M$
and 
    \begin{align*}
        \|\Delta^v(0)\|_F &< \kappa \int_{\|\Delta^x(0)\|_F}^M \psi(s) ds  = \kappa   \int_{\|\Delta^x(0)\|_F}^M (1 - N  \|\phi \|_{\text{Lip}} s ) ds \\ & = \kappa \left(M - \|\Delta^x(0)\|_F - \frac{N  \|\phi \|_{\text{Lip}} }{2} (M^2 - \|\Delta^x(0)\|_F^2) \right) < \kappa M  \,.
    \end{align*}    
Then applying \eqref{eq: Dx_n_recurrence} to $n = n^{\infty -1}$ we see
  \begin{align*}
        \|\Delta^x(n^{\infty})\|_F  & < \|\Delta^v(0)\|_F h + M < \kappa h M + M < 2M
    \end{align*} as long as $ h< \frac{1}{\kappa}$, which holds from \cref{as: a}. It follows that $\psi(\|\Delta^x(n^{\infty})\|_F) > 0$. 
Since $\psi$ is decreasing, we can estimate the integral from $\|\Delta^x(0)\|_F$ to $\|\Delta^x(n^{\infty})\|_F$ from above by a Riemann sum: \begin{align*}
        & \sum_{q = 0}^{K - 1} \left( \|\Delta^x_{q + 1}\|_F - \|\Delta^x_q\|_F \right) \kappa \psi(\|\Delta^x_q\|_F) > \int_{\|\Delta^x(0)\|_F}^{\|\Delta^x(n^{\infty})\|_F}\kappa \psi(s) ds  \\ 
        & \hspace{1cm}  > \int_{\|\Delta^x(0)\|_F}^{\|\Delta^x(n^{\infty})\|_F}\kappa \psi(s) ds  - \left( \int_{\|\Delta^x(0)\|_F}^M \kappa \psi(s) ds - \|\Delta^v(0)\|_F \right) \\ 
        &  \hspace{1cm}  = \int_{M}^{\|\Delta^x(n^{\infty})\|_F}\kappa \psi(s) ds + \|\Delta^v(0)\|_F > \|\Delta^v(0)\|_F\,.
    \end{align*}
    We substitute this bound into \eqref{eq: v_recurrence_2} to get 
    \begin{equation*}
        \|\Delta^v(n^{\infty})\|_F - \|\Delta^v(0)\|_F < - \|\Delta^v(0)\|_F\,.
    \end{equation*}
    We arrived at a contradiction since obviously, it must hold that $\|\Delta^v(n^{\infty})\|_F\ge 0$. Therefore, we have
    \[ n^{\infty} = + \infty \quad \mbox{and} \quad  \sup_{0 \leq n < \infty} \|\Delta^x(n)\|_F \leq M. \]
    
\noindent (ii) (Exponential decay of velocity shape discrepancy): Since $\psi$ is decreasing and $\|\Delta^x(n)\|_F$ is uniformly bounded by $M$ from (i),  we have $\psi(\|\Delta^x(n)\|_F) \geq \psi(M)$.
   Therefore, it follows from $\eqref{New-2}_2$ that 
    \[
        \|\Delta^v(n+1)\|_F - \|\Delta^v(n)\|_F \leq - h\kappa \|\Delta^v(n)\|_F \psi(\|\Delta^x(n)\|_F) \leq - h\kappa \|\Delta^v(n)\|_F\psi(M)\,.
    \]
    This yields
     \begin{align}
     \begin{aligned} \label{eq: NN-0}
        \|\Delta^v(n+1)\|_F &\leq \|\Delta^v(n)\|_F ( 1- h\kappa \psi(M)) \\ & \leq \|\Delta^v(0)\|_F(1 - h\kappa \psi(M))^{n+1}  \\ & = \|\Delta^v(0)\|_F\exp\left[ (n+1) \log( 1 - h\kappa \psi(M)) \right]  \\ & =  \|\Delta^v(0)\|_F\exp\left[ (n+1) h \left(\frac{\log( 1 - h\kappa \psi(M))}{h} \right) \right].
   \end{aligned}
    \end{align}
    On the other hand, note that 
    \begin{equation} \label{eq: NN-1}
        \lim_{h \rightarrow 0} \frac{\log (1 - h\kappa \psi(M))}{h} = - \kappa \psi(M).
    \end{equation}
 Finally, we combine \eqref{eq: NN-0} and \eqref{eq: NN-1} to obtain 
    \begin{equation*}
        \|\Delta^v(n)\|_F \leq \|\Delta^v(0)\|_F e^{-C \kappa \psi(M)nh},
    \end{equation*}
        for any $0 < C < 1$ and $h \ll 1$ small enough. 
\end{proof}

In \cref{sec:2}, we mention that if solutions to the continuous MT model follow flocking dynamics, then the particles converge to exhibit a uniform constant velocity as they are evolved by the dynamical system. Although such velocity is not an invariant quantity as it is determined by the initial setup, it is notable that there indeed exists a limiting velocity. In \cite{MT14}, this is referred to as an \emph{emergent} flocking velocity. We can conclude the analogous result for the discrete MT model based on \cref{thm: flocking}.
\begin{corollary}\label{cor: discrete_v_limit}
    Under the same setting as in \cref{thm: flocking}, there exists a set of asymptotic velocities $\{v_i^{\infty} \}$ such that 
  \[  v_i^{\infty} := \lim_{n \rightarrow \infty} v_i(n) \quad \text{ with } \quad \|v_i^\infty\|< \infty, \quad \forall~i \in [N]\,,\]
  and 
  \begin{align}\label{eq: discrete_velocity}
          \|v_i^{\infty} - v_i(n)\| \le \frac{ h\kappa \sqrt{N} \|\Delta^v(0)\|_F  e^{ - C\kappa \psi(M)h n}}{1- e^{ - C\kappa \psi(M)h}  } \,.
  \end{align}
\end{corollary}
\begin{proof}
Note that $v_i(n)$ satisfies \begin{align*} 
    v_i(n) = v_i(0) + h\kappa \sum_{m = 0}^{n-1}\sum_{j = 1}^N \phi_{ij}(m)\left(v_j(m) - v_i(m)\right).
\end{align*}
Now, we set 
\begin{equation}  \label{eq: discrete_v_inf}
v_i^{\infty} := v_i(0) + h\kappa \sum_{m = 0}^{\infty}\sum_{j = 1}^N \phi_{ij}(m) \left(v_j(m) - v_i(m)\right), \quad \phi_{ij}(m):= \frac{a_{ij}(m)}{\sum_{k = 1}^N a_{ik}(m)}.
\end{equation}
Since $\phi_{ij} \leq 1$ and the summand in the summation decays exponentially to zero by \cref{thm: flocking}, the right-hand side of \eqref{eq: discrete_v_inf} is well-defined. More specifically, we denote the summand by $b_m\in\R^d$ and note that
    \begin{align*}
    \begin{aligned}
      \|b_m\|^2&= \left\|\sum_{j = 1}^N \phi_{ij}(m)\left(v_j(m) - v_i(m)\right)\right\|^2 
      \le N \sum_{j=1}^N \| \Delta^v_{ij}(m)  \|^2
          \le N  \| \Delta^v(m)  \|_F^2\\
      &\le  N  \| \Delta^v(0)  \|_F^2 e^{-2C\kappa \psi(M) mh}\,.
     \end{aligned}
    \end{align*}
    Hence, $\sum_{m=0}^\infty b_m$ is absolutely convergent and so $v_i^\infty$ in \eqref{eq: discrete_v_inf} is well-defined. 
    We can then compute directly
 \begin{align*}
    \|v_i^{\infty} - v_i(n)\| &= \left\| h\kappa\sum_{m = n}^{\infty}\sum_{j=1}^N \phi_{ij}(m) \left(v_j(m) - v_i(m)\right) \right\| 
     \leq h\kappa \sum_{m = n}^{\infty} \sum_{j = 1}^N \left\| v_j(m) - v_i(m) \right\| \\ 
    & \leq h\kappa \sqrt{N} \sum_{m = n}^{\infty} \|\Delta^v(m)\|_F 
     \leq h\kappa \sqrt{N} \sum_{m = n}^{\infty}\|\Delta^v(0)\|_F e^{-C\kappa \psi(M) mh}\\
    & =\frac{ h\kappa \sqrt{N} \|\Delta^v(0)\|_F  e^{ - C\kappa \psi(M)h n}}{1- e^{ - C\kappa \psi(M)h}  } \,.
\end{align*}
Hence, for fixed $h>0$,
\begin{align*}
    \lim_{n\to\infty}\|v_i^{\infty} - v_i(n)\| =0\,.
\end{align*}
\end{proof}
\begin{remark}
    The asymptotic flocking velocity exists, although it cannot be expressed explicitly in terms of the initial configuration. This is similar to the continuous MT model as seen in \cref{cor: continuous_v_limit}.
\end{remark}
\subsection{Uniform-in-time continuous limit} \label{sec:3.3}
In this subsection, we study the uniform-in-time transition from the discrete model:
\begin{equation}\label{C-1}
	\begin{cases}
		\displaystyle x^h_i(n+1) = x^h_i(n)+hv^h_i(n)\,, \quad n \geq 0, \quad i \in [N], \\
		\displaystyle v^h_i(n+1) = v^h_i(n)+h\kappa \sum_{j=1}^{N} {\phi}^h_{ij}(n) \big(v^h_j(n)-v^h_i(n)\big),
	\end{cases}
\end{equation}
to the continuous model:
\begin{equation}\label{C-2}
\begin{cases}
 \displaystyle {\dot x}_i = v_i\,, \quad t > 0, \quad i \in [N], \\
 \displaystyle  {\dot v}_i =  \kappa \sum_{j = 1}^N \phi_{ij}(v_j - v_i),
 \end{cases}
\end{equation}
as the time-step $h$ vanishes. For notational simplicity, we set 
\[ X^h := (x_1^h, \ldots, x_N^h), \quad V^h := (v_1^h, \ldots, v_N^h). \]
We first establish the continuous transition for $h\to 0$ in any finite time interval (see \cref{cor:finitetime-limit} below) : for any fixed $T>0$,
\begin{equation*} 
\lim_{h \to 0+} \sup_{0 \leq n < \floor*{\frac{T}{h}}}  \Big(  \|X^h(n) - X(nh) \| +  \|V^h(n) - V(nh)\| \Big) = 0, 
\end{equation*}
where $\lfloor x\rfloor $ denotes the largest integer less than equal to $x$. 

In order to extend this result to a uniform-in-time statement, we cannot hope to achieve decay to zero for $\|X^h(n) - X(nh) \|$ uniformly in time due to the fact that our uniform-in-time flocking result only holds in shape discrepancy and not for positions and velocities directly. Therefore, we introduce the following definition.
\begin{definition}\label{def: continuous_transition}
    The discrete system \eqref{eq: discrete_MT} converges to \eqref{eq: conti_MT} \emph{uniformly in time} if for some class of solutions $(X^h, V^h)$ and $(X, V)$ to \eqref{C-1} and \eqref{C-2}, respectively, there exists a constant $c\ge 0$ only depending on the model parameters, choice of communication function $a$ and initial condition $(X(0), V(0))=(X^h(0), V^h(0))$ such that 
 \[ \limsup_{h \rightarrow 0} \sup_{0 \leq n < +\infty}  \|\Delta^{x,h}(n) - \Delta^x(nh)\|_F \le C\,,\quad
  \limsup_{h \rightarrow 0} \sup_{0 \leq n < +\infty}   \|V^h(n) - V(nh)\|  = 0\,.
 \]
\end{definition}
In order to derive such a result for an infinite time interval, we can use the flocking theorem (\cref{thm: flocking}) for the discrete MT model, and its continuous counterpart \cite[Proposition 2.9]{MT14}. This was also the approach that the authors in \cite{Ha_Zhang_CS_continuous_transition} used to show the discrete to continuous uniform-tin-time transition for the Cucker-Smale model.

We start by recalling a general finite-time convergence result. Consider the following Cauchy problems for a general first-order ODE system and its corresponding discrete system:
 \begin{equation}  \label{eq: discrete_general}
        \begin{cases}
         \displaystyle   \frac{dy}{dt} = f(y), \quad  0 < t < T, \\ 
        \displaystyle   y(0) = y_0,
        \end{cases} 
        \quad \begin{cases}
            y_{n+1} = y_n + hf(y_n). \quad 0 <  n  < \lfloor \frac{T}{h} \rfloor, \\  
            y(0) = y_0.
        \end{cases}
    \end{equation}
For $R > 0$, we set 
\[ \Omega:= \{y: \|y - y_0\|< R\}. \]
\vspace{0.2cm}
Suppose that the vector field $f$ and discrete solution satisfy the following conditions: 
\begin{itemize}
        \item
        ${(\mathcal A}1)$:~The vector field $f$ is Lipschitz continuous on the open set $\Omega$:
         \[
            \mathcal{L}_f := \sup_{y_1 \neq y_2, y_1, y_2 \in \Omega} \frac{|f(y_2) - f(y_1)|}{|y_2 - y_1|} < \infty\,.
        \]
        \item 
         ${(\mathcal A}2)$:~The discrete solution $(y_n)$ satisfies
         \[
            \|y_n - y_0\| \leq R, \quad \forall~n = 0, 1, \cdots, \floor*{\frac{T}{h}}.
        \]
    \end{itemize} 
In the following proposition, we recall finite-time convergence from the discrete model to its continuous counterpart stated in \eqref{eq: discrete_general}. 
\begin{proposition} \label{prop: finite_transition}
Suppose that the conditions  ${(\mathcal A}1)$ and  ${(\mathcal A}2)$ hold for some $T>0$. Then the following assertions hold.
    \begin{enumerate}
        \item For the truncation error $\mathcal{E}_1^h(n)$ defined by
        \begin{align*}
            \mathcal{E}_1^h(n):= \left\|\frac{dy}{dt}\bigg\rvert_{t = nh} - \frac{y((n+1)h - y(nh)}{h}\right\|\,.
        \end{align*}
        we have \begin{align}\label{eq: finite_e1}
            \lim_{h \rightarrow 0 }\max_{0 \leq n \leq \floor*{T/h}}\mathcal{E}_1^h(n) = 0\,. 
        \end{align}
        \item The error $\mathcal{E}_2^h(n)$ defined by
         \begin{align*}
            \mathcal{E}_2^h(n):= \|y(nh) - y_n\|
        \end{align*} 
        can be controlled by the truncation error $\mathcal{E}_1^h(n):$
        \begin{align} \label{eq: finite_e2}
            \mathcal{E}_2^h(n) \leq \frac{\max_{0 \leq n \leq \floor*{T/h}}\mathcal{E}_1^h(n)}{\mathcal{L}_f}\left(  e^{\mathcal{L}_f nh} - 1\right)\,, \quad 0 \leq n \leq \floor*{\frac{T}{h}}\,.
        \end{align}
    \end{enumerate}
\end{proposition}
\begin{proof}
    For the proof, refer to \cite[Chapter 12,Theorem 12.2]{Suli_Mayers_2003}.
\end{proof}
In fact, in some situations, the condition ${(\mathcal A}2)$ follows from the condition ${(\mathcal A}1)$ directly, as can be seen in the following lemma.
\begin{lemma}\label{lem:A1-A2}
    If ${(\mathcal A}1)$ is satisfied for 
$
R> \frac{1}{\mathcal{L}_f}\left((1+\mathcal{L}_f h)^{\floor*{\frac{T}{h}}}-1\right) \|f(y_0)\|$,
then ${(\mathcal A}2)$ holds. 
\end{lemma}
\begin{proof}
    Indeed, the relation
\[ \|f(y_k)\|\le \mathcal{L}_f\|y_k-y_0\|+\|f(y_0)\| \]
implies
\begin{align*}
    \|y_n-y_0\|&\le (1+\mathcal{L}_f h)^k\|y_{n-k}-y_0\|+h\|f(y_0)\|\sum_{l=0}^{k-1} (1+\mathcal{L}_f h)^l\\
    &= h\|f(y_0)\|\sum_{l=0}^{n-1} (1+\mathcal{L}_f h)^l
    \le \frac{1}{\mathcal{L}_f}\left((1+\mathcal{L}_f h)^{\floor*{\frac{T}{h}}}-1\right) \|f(y_0)\|\,.
\end{align*}
\end{proof}
To apply the above results to the MT model, we consider \eqref{eq: discrete_general} with
\begin{equation*}
y=
    \begin{pmatrix}
x_1\\ \vdots \\ x_N \\ v_1 \\ \vdots \\ v_N
\end{pmatrix} 
 \in \R^{2Nd}\,,\quad
f(y)=
    \begin{pmatrix}
0_N & \Id_N\\
0_N & \kappa (A(X)-\Id_N)
\end{pmatrix} \cdot y
\,,\quad 
A_{ij}(X)= \frac{a(\|x_i - x_j \|)}{\sum_{k=1}^{N} a( \|x_i - x_k \|)}\,.
\end{equation*}

\begin{corollary}[Finite-time transition]\label{cor:finitetime-limit} 
Let $(X(t),V(t))$ and $(X^h(n),V^h(n))$ be the global solutions to the continuous MT model \eqref{C-2} and the discrete MT model \eqref{C-1}, respectively, such that the initial condition $\|y_0\|$ for the discrete model is sufficiently small (depending on $\kappa, h, c_1, c_2, L_a$). Then, for all $0\le T < \infty$, we have the following finite-time transition from discrete to continuum:
    \begin{equation} \label{eq: finite_transition} 
\lim_{h \to 0+} \sup_{0 \leq n \le \floor*{\frac{T}{h}}}  \Big(  \|X^h(n) - X(nh) \| +  \|V^h(n) - V(nh)\| \Big) = 0\,. 
\end{equation}
\end{corollary}
\begin{proof}
To show that the assumptions for \cref{prop: finite_transition} are satisfied, note that for any two-particle ensembles $y,\bar y \in \Omega$, it holds
\begin{align*}
    \|f(y)-f(\bar y)\|^2 
    &\le \|V-\bar V\|^2 + \kappa^2 \|A(X)(V-\bar V)\|^2 + \kappa^2 \|(A(X)-A(\bar X)) \bar V\|^2 + \kappa^2\|V-\bar V\|^2\,.
\end{align*}
Using condition (1) in \cref{as: a}, we have \[
\|A(X)(V-\bar V)\|^2\le (c_2/c_1)^2\|V-\bar V\|^2\,. \] Thanks to condition (2), the relation
$$
|A_{ij}(X)-A_{ij}(\bar X)|\le
\frac{L_a}{Nc_1}\left[
\left(1+\frac{c_2}{c_1}\right) \|x_i-\bar x_i\| + \|x_j-\bar x_j\| + \frac{c_2}{Nc_1}\sum_k \|x_k-\bar x_k\|
\right]
$$
implies
\begin{align*}
    \|(A(X)-A(\bar X)) \bar V\|^2 &\le N\sum_{ij}|A_{ij}(X)-A_{ij}(\bar X)|^2 \|\bar v_j\|^2 \\
    &\le  4\left(1+\frac{c_2}{c_1}\right)^2\left(\frac{L_a }{c_1}\right)^2 (R+\|y_0\|)^2\|X-\bar X\|^2\,.
\end{align*}
Hence, the hypothesis ${(\mathcal A}1)$ is satisfied with 
\begin{equation*}
    \mathcal{L}_f= \left(\max\left\{\left(1+\kappa^2+\left(\frac{c_2}{c_1}\right)^2\right), 4 \kappa^2 \left(1+\frac{c_2}{c_1}\right)^2\left(\frac{L_a }{c_1}\right)^2 (R+\|y_0\|)^2\right\}\right)^{1/2} <\infty\,.
\end{equation*}
Next, we use \cref{lem:A1-A2}, and denote $R^*$ by
\begin{align*}
    R^*:= \frac{\sqrt{1+\kappa^2+\left(\frac{c_2}{c_1}\right)^2}}{2\kappa\left(1+\frac{c_2}{c_1}\right)\frac{L_a}{c_1}}-\|y_0\|\,,
\end{align*}
which is positive for sufficiently small $\|y_0\|$. Then for any $0<R\le R^*$, we have 
\begin{equation*}
    \mathcal{L}_f= \sqrt{1+\kappa^2+\left(\frac{c_2}{c_1}\right)^2}\,.
\end{equation*}
For the condition in \cref{lem:A1-A2} to hold and choosing $R=R^*$,  we require
\begin{align*}
R^*=\frac{\mathcal{L}_f}{2\kappa\left(1+\frac{c_2}{c_1}\right)\frac{L_a}{c_1}}-\|y_0\| > \frac{1}{\mathcal{L}_f}\left((1+\mathcal{L}_f h)^{\lfloor  \frac{T}{h} \rfloor}-1\right) \|f(y_0)\|\,,
\end{align*}
which is indeed satisfied for sufficiently small $\|y_0\|$ and $\|f(y_0)\|$. We conclude that ${(\mathcal A}2)$ holds, noting that 
$ \|f(y_0)\|^2\le \left(1+\kappa^2\left(\frac{c_2}{c_1}\right)^2\right)\|V_0\|^2$ thanks to condition (1) in \cref{as: a}.
Hence, it follows from \cref{prop: finite_transition} that for all $0 \leq n \leq \lfloor  \frac{T}{h} \rfloor,$ 
    \begin{align*}
        \mathcal{E}_2^h(n)&:=\left( \|X^h(n) - X(nh)\|^2 + \|V^h(n) - V(nh)\|^2\right)^{1/2} \leq  \frac{\max_{0 \leq n \leq \floor*{T/h}}\mathcal{E}_1^h(n)}{\mathcal{L}_f}\left(  e^{\mathcal{L}_f nh} - 1\right)
    \end{align*}
    by \eqref{eq: finite_e2}. Thus, we have
    \begin{align*}
        \limsup_{h \rightarrow 0} \sup_{0 \leq n \le \floor*{\frac{T}{h}}} \mathcal{E}_2^h(n) &\leq \limsup_{h \rightarrow 0} \sup_{0 \leq n \le \floor*{\frac{T}{h}}}\frac{\max_{0 \leq n \le \floor*{\frac{T}{h}}}\mathcal{E}_1^h(n)}{\mathcal{L}_f}\left(  e^{\mathcal{L}_f nh} - 1\right) \\ &\leq  \lim_{h \rightarrow 0 } \frac{\max_{0 \leq n \le \floor*{\frac{T}{h}}}\mathcal{E}_1^h(n)}{\mathcal{L}_f} \left(  e^{\mathcal{L}_f T} - 1\right)   = 0
    \end{align*} 
    by \eqref{eq: finite_e1}.
\end{proof}
Note that in the above estimation, there is an exponential dependence on the upper bound of the error term with respect to the time horizon $T$. Hence, such analysis does not hold for an infinite time interval. On the other hand, if we consider the shape discrepancy for positions instead of positions themselves, we can achieve the uniform-in-time transition from the discrete MT model to the continuous one as $h \to 0$ for an infinite time interval as stated in \cref{def: continuous_transition}. We present this result in the following theorem.
\begin{theorem}[uniform-in-time transition] \label{T3.10}
    Let $(X^h, V^h)$ and $(X, V)$ be global solutions to \eqref{C-1} and \eqref{C-2}, respectively. Assume the condition from \cref{thm: conti_flocking} hold for $(X(0), V(0))$ and the conditions from \cref{thm: flocking} hold for $(X^h(0), V^h(0)):=(X(0), V(0))$. Then, we have
 \begin{gather*}
 \limsup_{h \rightarrow 0} \sup_{0 \leq n < +\infty}  \|\Delta^{x,h}(n) - \Delta^x(nh)\|_F \le c_0 \mathcal{D}(V(0))\,,\\
  \text{ and } \quad \limsup_{h \rightarrow 0} \sup_{0 \leq n < +\infty}   \|V^h(n) - V(nh)\|  = 0
  \end{gather*}
  for some constant $c_0=c_0(\kappa, N, a)$ independent of $n$ and $h$.
\end{theorem}
\begin{proof}
    By \cref{cor: continuous_v_limit} and \cref{cor: discrete_v_limit}, we have
 \[ \exists~V^{\infty} := \lim_{t \to \infty} V(t), \quad \exists~V^{h,\infty} := \lim_{n \to \infty} V^h(n). \]   
Fix $\epsilon>0$. Then by the flocking estimates for \eqref{C-1} and \eqref{C-2}, there exist times $T_1$ and $T_2$ such that for any $h>0$,
\begin{equation}  \label{eq: est1}
    \|V^h(n) - V^{h, \infty} \| < \frac{\epsilon}{9} \quad  
    \forall~n \geq \floor*{\frac{T_1}{h}} \quad \mbox{and} \quad 
    \|V(nh) - V^{\infty} \| < \frac{\epsilon}{9}  
    \quad  \forall~n \geq \floor*{\frac{T_2}{h}}.
\end{equation} 
The second statement above follows directly as the solution itself does not depend on $h$. For the first statement, however, we need to check carefully the interplay between $n$ and $h$. Indeed, from \eqref{eq: discrete_velocity} we have \begin{align*}
    \|V^h(n) - V^{h, \infty}\|^2 &= \sum_{i = 1}^N \|v_i^h(n) - v_i^{h, \infty} \|^2 
    \le\sum_{i = 1}^N \frac{ h^2\kappa^2 N \|\Delta^v(0)\|^2_F  e^{ - 2C\kappa \psi(M)h n}}{(1- e^{ - C\kappa \psi(M)h} )^2 } \\
    &= \frac{ h^2\kappa^2 N^2 \|\Delta^v(0)\|_F^2  e^{ - 2C\kappa \psi(M)h n}}{(1- e^{ - C\kappa \psi(M)h})^2 }. 
\end{align*}
For this expression to be controlled by $(\epsilon/9)^2$, we require
\begin{align*}
    n&> \frac{1}{C\kappa \psi(M) h}\, \left[ 
    \log\frac{9}{\epsilon} + \log\left(\frac{h\kappa N \|\Delta^v(0)\|_F}{1- e^{ - C\kappa \psi(M)h}} \right)
    \right]\\
   & = \frac{1}{2C\kappa \psi(M) h}\, \left[ 
    \log\frac{9}{\epsilon} + \log\left( \kappa N \|\Delta^v(0)\|_F\right) 
    + \log\left(\frac{h}{ C\kappa \psi(M)h + O(h^2)} \right)
    \right] = O\left(\frac{1}{h}\right)\,,
\end{align*}
and therefore such a $T_1$ exists for small enough $h$ as claimed. 
Now, we set 
\[ \bar T := \max\{{T_1, T_2 }\}. \]
Then we use the finite-time result in \cref{cor:finitetime-limit} to see that there exists a small enough $h_0=h_0(\epsilon)>0$ such that 
\begin{equation}\label{eq: est3}
	\|V^h(n) - V(nh)\| <  \frac{\epsilon}{9}
\end{equation} 
for all $0 < h < h_0$ and $0 \leq n \leq \floor*{\bar T/h}$. Choosing $\tilde n =  \floor*{\bar T/h},$ 
we can estimate the error between the discrete and continuum limit points via
\begin{equation}\label{eq: est4}
    \| V^{h,\infty} - V^\infty\| \leq \|V^{h, \infty} - V^h(\tilde n) \| 
    + \|V^h(\tilde n) - V(\tilde n h )\|  + \|V(\tilde n h) - V^\infty \|  
    < \frac{\epsilon}{3}
\end{equation} for all $0 < h < h_0$. 
Finally, we combine \eqref{eq: est1}, and \eqref{eq: est4} to find the uniform-in-time transition in velocities: 
\begin{equation}\label{eq: v_transition_est}
	\|V^h(n) - V(nh)\| \leq \|V^h(n) - V^{h,\infty} \| + \|V^{h, \infty} - V^\infty\|  + \|V^\infty - V(nh)\|  < \epsilon
\end{equation} 
for all $n \geq \floor*{\bar T/h}$, and $0 < h < h_0$. Combining with \eqref{eq: est3}, we get \begin{align*}
    \|V^h(n) - V(nh)\| < \epsilon
\end{align*}
for all $n \geq 0$, and $0 < h < h_0$.\\
Now, consider the shape discrepancy for positions:
\begin{align} \label{C-28}
\begin{aligned}
\|\Delta^{x,h}_{ij}(n) - \Delta^x_{ij}(nh)\| &=  \|((x_i^h(n) - x_j^h(n)) - (x_i(nh) - x_j(nh))\|\\
&= \left\|\sum_{k=0}^{n-1} h [(v_i(kh) - v_i^h(k)) - (v_j(kh) - v_j^h(k))] \right\|+ O(h^2)\\
&\leq h\sum_{k=0}^{n-1} \|\Delta_{ij}^{v,h}(k) - \Delta_{ij}^v(kh)\| + O(h^2)\\
&\leq h\sum_{k=0}^{\infty} \|\Delta_{ij}^{v,h}(k) - \Delta_{ij}^v(kh)\| + O(h^2).
\end{aligned}
\end{align}
Note we can use \cref{thm: flocking} and \cref{thm: conti_flocking} to obtain that for any $K \in \mathbb{N}$,
\begin{align*}
    \sum_{k=K}^\infty &\|\Delta^{v,h}_{ij}(k) -\Delta^v_{ij}(hk)\| \leq \sum_{k=K}^\infty \|\Delta^{v,h}_{ij}(k)\| + \sum_{k=K}^\infty\|\Delta^v_{ij}(hk)\|\\
    &\leq \sum_{k=K}^\infty\|\Delta^{v,h}(k)\|_F + \sum_{k=K}^\infty \mathcal{D}(V(hk))\\
    &\leq \sum_{k=K}^\infty\|\Delta^{v,h}(0)\|_F e^{-C\kappa \psi(M)kh} + \sum_{k=K}^\infty \mathcal{D}(V^0)e^{-\kappa a(\mathcal{D}^\infty)hk}\\
    &\leq \max \left\{\|\Delta^{v}(0)\|_F,\mathcal{D}(V(0))\right\}\frac{e^{-\tilde{c}hK}}{1-e^{-\tilde{c}h}}
    \le \frac{1}{h}\left(\frac{N\mathcal{D}(V(0))}{\tilde c } + O(h)\right)\,,
\end{align*}
where 
$\tilde{c} = \min\left\{C\kappa \psi(M), \kappa a(\mathcal{D}^\infty)\right\}$.
Together with \cref{cor:finitetime-limit} to bound the case $k\in\{0,...,K-1\}$, we conclude that for fixed $K \in \mathbb{N}$ and $h$ small enough 
\begin{align*}
\|\Delta^{x,h}_{ij}(n) - \Delta^x_{ij}(nh)\| \le \frac{2N\mathcal{D}(V(0))}{\tilde c }=:c_0 \mathcal{D}(V(0))\,,
\end{align*}
and so $\|\Delta^{x,h}(n) - \Delta^x(nh)\|_F$ 
is uniformly bounded for all $n\ge 0$ and small enough $h>0$.
\end{proof}
%
%
%
%
\section{Uniform-in-time stability estimates} \label{sec:4}
\setcounter{equation}{0}
In this section, we study a version of uniform $\|\cdot\|_F$-stability for the discrete MT model with respect to initial data. As briefly discussed in the introduction, our estimate is not the typical type of stability estimate. Typical stability estimates bound the difference between two pairs of solutions over all time in terms of the difference in the initial data. Due to the loss of conservation of momentum in \eqref{eq: discrete_MT}, comparing two pairs of solutions directly is technically challenging. Hence, we consider instead the difference in shape discrepancies $\Delta^x, \Delta^v$ between solution pairs. In \cref{sec:3}, we derived the flocking estimate for these quantities and concluded uniform-in-time discrete to continuous transition based on this estimate. The flocking estimate is also a key ingredient to obtain uniform-in-time stability for the discrete MT model. Our stability estimate will guarantee that the structure of the solution is preserved if the initial shape discrepancies are close enough. Such structural preservation of solutions is analogous to orbital stability in hyperbolic conservation laws, also see \cref{rmk:orbital-stability} for more details. 
\newline

We start by introducing two assumptions needed for uniform stability.
\begin{itemize}
    \item 
    ($\mathcal{F}1$):~Initial data $(X^0,V^0)$ and $(\overline{X}^0, \overline{V}^0)$ satisfy
    \[\max\left\{\|\Delta^x(0)\|_F,\|\Delta^{\overline{x}}(0)\|_F \right\} < M =  \frac{1}{4N \| \phi \|_{\text{Lip}}}.\]
    \item 
    ($\mathcal{F}2$):~Coupling strength $\kappa$ is sufficiently large such that
    \begin{gather*}
\|\Delta^v(0)\|_F < \kappa \int_{\|\Delta^x(0)\|_F}^M \psi(s) ds\,, \qquad  \|\Delta^{\overline{v}}(0)\|_F < \kappa \int_{\|\Delta^{\overline{x}}(0)\|_F}^M  \psi(s) ds\,.
    \end{gather*}
\end{itemize}
By \cref{thm: flocking}, the framework $(\mathcal{F}1) - (\mathcal{F}2)$ implies the emergence of asymptotic flocking for $(X,V)$, $(\overline{X}, \overline{V})$ and their shape discrepancies.
\begin{theorem} 
\emph{(Uniform-in-time stability)} \label{thm: stability} 
    Suppose that the initial data $(X^0,V^0)$ and $(\overline{X}^0, \overline{V}^0)$ and coupling strength satisfy the framework $(\mathcal{F}1) -  (\mathcal{F}2)$, and    
    let $(X,V)=(X(n),V(n))_{n\in\N}$ and $(\overline{X}, \overline{V})=(\overline{X}(n), \overline{V}(n))_{n\in\N}$ be solutions to \eqref{eq: discrete_MT} with initial data $(X^0,V^0)$ and $(\overline{X}^0, \overline{V}^0)$, respectively. 
    For sufficiently small $h>0$, there exist $C(n,h)>0$ such that
    \begin{align*}
    &\left\|\Delta^x(n)-\Delta^{\overline{x}}(n)\right\|_F + \left\|\Delta^v(n)-\Delta^{\overline{v}}(n)\right\|_F\\
    &\hspace{2cm}\leq\left\|\Delta^x(0)-\Delta^{\overline{x}}(0)\right\|_F + \left\|\Delta^v(0)-\Delta^{\overline{v}}(0)\right\|_F+C(n,h), \quad n \geq 1,
    \end{align*}
    where $\lim_{h \to 0} C(n,h) = C_0$ for some constant $C_0>0$, only depending on the initial conditions $V(0)$ and $\overline{V}(0)$.
\end{theorem}
\begin{proof}  Since the proof is rather lengthy, we give here an overview of the main ingredients and provide rigorous proofs of all necessary auxiliary results in the next subsections. We split the proof in three steps:\\

\noindent $\bullet$~Step A: we claim that for any $0<\epsilon<1$ and sufficiently small $h>0$ there exists positive constants $b_1(\epsilon, h),\,b_2(h)$ such that
\begin{align}
\begin{aligned} \label{D-0-0}
&\left\|\Delta^x(n+1)-\Delta^{\overline{x}}(n+1)\right\|_F\leq\left\|\Delta^x(n)-\Delta^{\overline{x}}(n)\right\|_F+h\left\|\Delta^v(n)-\Delta^{\overline{v}}(n)\right\|_F, \\
& \|\Delta^v(n+1) - \Delta^{\overline{v}}(n+1)\|_F \leq (1-\epsilon)\|\Delta^v(n) - \Delta^{\overline{v}}(n)\|_F + b_1(\epsilon,h)e^{-b_2(h)n},
\end{aligned}
\end{align}
where
\begin{align*}
b_1(\epsilon,h)&:=\left(\bar c_0 \epsilon+\bar c_1 h + \bar c_2 h^2\right)^{1/2}\,,\qquad b_2(h):=C\kappa \psi(M)h\,.
\end{align*}
\vspace{0.2cm}

\noindent $\diamond$~(Derivation of $\eqref{D-0-0}_1$):~It follows from \eqref{eq: discrete_MT} that 
 \[
\Delta_{ij}^x(n+1)-\Delta_{ij}^{\overline{x}}(n+1)=\Delta_{ij}^x(n)-\Delta_{ij}^{\overline{x}}(n)+h\left(\Delta_{ij}^v(n)-\Delta_{ij}^{\overline{v}}(n)\right).
\]
This yields
\[
\|\Delta_{ij}^x(n+1)-\Delta_{ij}^{\overline{x}}(n+1) \| \leq \|\Delta_{ij}^x(n)-\Delta_{ij}^{\overline{x}}(n) \| + h \|\Delta_{ij}^v(n)-\Delta_{ij}^{\overline{v}}(n) \|.
\]
Now, we sum up the above relation over all $i, j \in [N]$ to find the desired estimate $\eqref{D-0-0}_1$.

\vspace{0.2cm}

\noindent $\diamond$~(Derivation of $\eqref{D-0-0}_2$):~The proof of the second estimate $\eqref{D-0-0}_2$ requires several auxiliary results; we postpone it to \cref{lem:D-0-0-b}. \\

\vspace{0.2cm}

\noindent $\bullet$~Step B:~We estimate $\|\Delta^v(n) - \Delta^{\overline{v}}(n)\|_F $ and $\|\Delta^x(n) - \Delta^{\overline{x}}(n)\|_F$ in terms of quantities depending only on the initial data.\\

\noindent $\diamond$~(Bound on $\|\Delta^v(n) - \Delta^{\overline{v}}(n)\|_F$):~We claim that for any $0<\epsilon<1$ and $h$ small enough, we have
 \begin{equation} \label{D-New-0}
 \|\Delta^v(n) - \Delta^{\overline{v}}(n)\|_F\leq (1-\epsilon)^n\|\Delta^v(0) - \Delta^{\overline{v}}(0)\|_F + \frac{b_1e^{b_2}}{1-(1-\epsilon)e^{b_2}}e^{-b_2n}. 
 \end{equation}
 See \cref{cor: v_fluctuation_bound} for details.\\
 
 \noindent $\diamond$~(Bound on $\|\Delta^x(n) - \Delta^{\overline{x}}(n)\|_F$):~By  \eqref{D-0-0} and \eqref{D-New-0}, we have
\begin{align}
  & \|\Delta^x(n) - \Delta^{\overline{x}}(n)\|_F \notag\\
  & \hspace{0.5cm} \leq \|\Delta^x(0) - \Delta^{\overline{x}}(0)\|_F + h\sum_{k=0}^{n-1}\|\Delta^v(k) - \Delta^{\overline{v}}(k)\|_F\notag\\
  & \hspace{0.5cm} \leq \|\Delta^x(0) - \Delta^{\overline{x}}(0)\|_F + h\sum_{k=0}^{n-1}\left((1-\epsilon)^k\|\Delta^v(0) - \Delta^{\overline{v}}(0)\|_F + \frac{b_1e^{b_2}}{1-(1-\epsilon)e^{b_2}}e^{-b_2k}\right)\notag\\
  & \hspace{0.5cm} = \|\Delta^x(0) - \Delta^{\overline{x}}(0)\|_F + h\frac{1-(1-\epsilon)^n}{\epsilon}\|\Delta^v(0) - \Delta^{\overline{v}}(0)\|_F + h\frac{b_1e^{b_2}}{1-(1-\epsilon)e^{b_2}}\frac{1-e^{-b_2n}}{1-e^{-b_2}}\,.
    \label{D-0-1}
\end{align}


\noindent $\bullet$~Step C:~Choosing $0<h<\epsilon<1$, we combine \eqref{D-New-0} and \eqref{D-0-1} to find 
\begin{align*}
\begin{aligned}
& \|\Delta^x(n) - \Delta^{\overline{x}}(n)\|_F +   \|\Delta^v(n) - \Delta^{\overline{v}}(n)\|_F \\
&  \leq  \max \Big \{  1\,;\, \underbrace{\frac{h}{\epsilon}\left[1-(1-\epsilon)^{n}\right] +  (1-\epsilon)^{n}}_{<\,1}\Big \} \Big(   \|\Delta^x(0) - \Delta^{\overline{x}}(0)\|_F +   \|\Delta^v(0) - \Delta^{\overline{v}}(0)\|_F          \Big) \\
&\qquad+h\frac{b_1e^{b_2}}{1-(1-\epsilon)e^{b_2}}\frac{1-e^{-b_2n}}{1-e^{-b_2}} + \frac{b_1e^{b_2}}{1-(1-\epsilon)e^{b_2}}e^{-b_2n}\\
&\le \|\Delta^x(0) - \Delta^{\overline{x}}(0)\|_F +   \|\Delta^v(0) - \Delta^{\overline{v}}(0)\|_F  + C(n,h)
\end{aligned}
\end{align*}
with 
\begin{align*}
    C(n,h):=h\frac{b_1(\epsilon,h)e^{b_2(h)}}{1-(1-\epsilon)e^{b_2(h)}}\frac{1-e^{-b_2(h)n}}{1-e^{-b_2(h)}} + \frac{b_1(\epsilon,h)e^{b_2(h)}}{1-(1-\epsilon)e^{b_2(h)}}e^{-b_2(h)n}\,.
\end{align*}
It follows that for fixed $n\ge 0$,
\begin{align*}
   \lim_{h\to 0} \frac{1-e^{-b_2(h)n}}{1-e^{-b_2(h)}} = \lim_{h\to 0} \frac{C\kappa \psi(M)hn + O(h^2)}{C\kappa \psi(M)h + O(h^2)} = n\,.
\end{align*}
Hence for any $n\ge 0$,
\begin{align*}
    \lim_{h\to 0} C(n,h) = \sqrt{\frac{\bar c_0}{ \epsilon}} = \frac{2}{\sqrt{\epsilon}} \sqrt{\|\Delta^v(0)\|_F^2 + \|\Delta^{\overline{v}}(0)\|_F^2}\,.
\end{align*}
\end{proof}
\subsection{Propagation of velocity shape discrepancy} \label{sec:4.1}

Our main goal in this subsection is to control the shape discrepancy in velocities between pairs $(X, V)$ and $(\overline{X}, \overline{V})$ of solutions by bounding how their difference propagates via the discrete MT dynamics. For this, we define $\lambda = \lambda(n)$ as
\begin{equation} \label{D-lambda}
	\lambda:= 
 \min_{1\leq i,j\leq N}(\overline{\phi}_{ij}+\overline{\phi}_{ji})>0\,,\qquad 
 \overline{\phi}_{ij} = \frac{a\left(\|\overline{x_i}-\overline{x_j}\|\right)}{\sum_{k=1}^N a\left(\|\overline{x_i}-\overline{x_k}\|\right)}\,.
\end{equation}
and $\alpha = \alpha(n)$ as
\begin{equation} \label{D-alpha}
	\alpha:= \max_{1\leq i,j\leq N}\alpha_{ij} = \max_{1\leq i,j\leq N}\left(1-\overline{\phi}_{ij} - \overline{\phi}_{ii}\right)^2.
\end{equation}
\begin{proposition} \label{P4.5}
    Suppose the framework $(\mathcal{F}1) - (\mathcal{F}2)$ holds, and let $(X,V)$ and $(\overline{X}, \overline{V})$ be global solutions to \eqref{eq: discrete_MT}. Then, we have
    \begin{align*} 
		&\|\Delta^v(n+1)-\Delta^{\overline{v}}(n+1)\|_F^2\notag\\
  &\qquad \le \cC_1(n,h)\|\Delta^v(n) - \Delta^{\overline{v}}(n)\|_F^2 + \cC_2(n,h)\|\Delta^v(n) - \Delta^{\overline{v}}(n)\|_F\|\Delta^v(n)\|_F + \cC_3(n,h)\|\Delta^v(n)\|_F^2\,,
  \end{align*}
  where
  \begin{align*}
  \cC_1(n,h)&:=(1-h\kappa\lambda(n))^2+ 4h^2\kappa^2\left(\alpha(n)+\frac{L_a^2M^2}{Nc_1^2}\left(1+\frac{c_2}{c_1}\right)^2\right)\\
   &\qquad+ 2\sqrt{2}h\kappa (1-h\kappa \lambda(n)) \left(\alpha(n)+\frac{L_a^2M^2}{c_1^2}\left(1+\frac{c_2}{c_1}\right)^2\right)^{1/2}\,,\\
   \cC_2(n,h)&:=\frac{8h\kappa L_a M}{c_1}(1-h\kappa\lambda(n))\left(1+\frac{c_2}{c_1}\right)\,,\\
    \cC_3(n,h)&:=\frac{32h^2\kappa^2L_a^2M^2}{c_1^2}\left(1+\frac{c_2}{c_1}\right)^2\,.
\end{align*}
\end{proposition}
We will prepare this result with several a priori estimates. 
\begin{lemma} \label{L4.2}
Suppose the framework $(\mathcal{F}1) -  (\mathcal{F}2)$ holds, and let $(X,V)$ and $(\overline{X}, \overline{V})$ be global solutions to \eqref{eq: discrete_MT}. Then, we have
\begin{align*}
\begin{aligned}
&\Delta_{ij}^v(n+1) -\Delta_{ij}^{\overline{v}}(n+1)
=(1-h\kappa\lambda_{ij}(n))(\Delta_{ij}^v(n)-\Delta_{ij}^{\overline{v}}(n)) + \cI^{ij}_1 + \cI^{ij}_2\,,
\end{aligned}
\end{align*}
where $\lambda_{ij} = \overline{\phi}_{ij}+\overline{\phi}_{ji}>0$
and
\begin{align*}
   \cI^{ij}_1&:= h\kappa\sum_{l=1}^{N}( \phi_{il}-\overline{\phi}_{il})\Delta_{li}^v(n)-h\kappa\sum_{l=1}^{N}(\phi_{jl}-\overline{\phi}_{jl})\Delta_{lj}^{v}(n)\,,\\
   \cI^{ij}_2&:=h\kappa\sum_{l\neq i,j}\Big[\overline{\phi}_{il}\left(\Delta_{li}^v(n)-\Delta_{li}^{\overline{v}}(n)\right)-\overline{\phi}_{jl}\left(\Delta_{lj}^v(n)-\Delta_{lj}^{\overline{v}}(n)\right)\Big]\,.
\end{align*}
\end{lemma}
\begin{proof}
It follows from \eqref{eq: discrete_MT} that 
    \begin{align} 
	\begin{aligned} \label{D-1-1}
	&\Delta_{ij}^v(n+1) -\Delta_{ij}^{\overline{v}}(n+1)\\ 
	&\hspace{0.5cm}=\big(v_i(n+1)-v_j(n+1)\big)-\big(\overline{v}_i(n+1)-\overline{v}_j(n+1)\big)\\
	&\hspace{0.5cm}=\left[v_i(n)-v_j(n)+h\kappa\sum_{l=1}^{N}\big(\phi_{il}(v_l(n)-v_i(n))- \phi_{jl}(v_l(n)-v_j(n))\big)\right]\\
	&\hspace{0.5cm}-\left[\overline{v}_i(n)-\overline{v}_j(n)+h\kappa\sum_{l=1}^{N}\big(\overline{\phi}_{il}(\overline{v}_l(n)-\overline{v}_i(n))-\overline{\phi}_{jl}(\overline{v}_l(n)-\overline{v}_j(n))\big)\right]\\
	&\hspace{0.5cm}=\Delta_{ij}^v(n)-\Delta_{ij}^{\overline{v}}(n)+h\kappa\sum_{l=1}^{N}\Big[\phi_{il}\Delta_{li}^v(n)- \phi_{jl}\Delta_{lj}^v(n)-\overline{\phi}_{il}\Delta_{li}^{\overline{v}}(n)+\overline{\phi}_{jl}\Delta_{lj}^{\overline{v}}(n)\Big]\,.
\end{aligned}
\end{align}
Now, we replace $\phi_{il}$ with $\phi_{il}-\overline{\phi}_{il}+\overline{\phi}_{il}$ to obtain
\begin{equation} \label{D-1}
	\sum_{l=1}^{N}\Big[ \phi_{il}\Delta_{li}^v(n)-\overline{\phi}_{il}\Delta_{li}^{\overline{v}}(n)\Big]=\sum_{l=1}^{N}\left(\phi_{il}-\overline{\phi}_{il}\right)\Delta_{li}^{v}(n)+\sum_{l=1}^{N}\overline{\phi}_{il}\left(\Delta_{li}^v(n)-\Delta_{li}^{\overline{v}}(n)\right).
\end{equation}
Note that the second summation on the right-hand side of \eqref{D-1} can be rewritten as
\begin{align*}
\begin{aligned}
	&\sum_{l=1}^{N}\overline{\phi}_{il}\left(\Delta_{li}^v(n)-\Delta_{li}^{\overline{v}}(n)\right)\\
	&\hspace{1cm}=\overline{\phi}_{ij}\left(\Delta_{ji}^v(n)-\Delta_{ji}^{\overline{v}}(n)\right)+\sum_{l\neq j}\overline{\phi}_{il}\left(\Delta_{li}^v(n)-\Delta_{li}^{\overline{v}}(n)\right)\\
	&\hspace{1cm}=\left(1-\sum_{l\neq j}\overline{\phi}_{il}\right)\left(\Delta_{ji}^v(n)-\Delta_{ji}^{\overline{v}}(n)\right)+\sum_{l\neq j}\overline{\phi}_{il}\left(\Delta_{li}^v(n)-\Delta_{li}^{\overline{v}}(n)\right)\\
	&\hspace{1cm}=-\left(\Delta_{ij}^v(n)-\Delta_{ij}^{\overline{v}}(n)\right)+\sum_{l\neq j}\overline{\phi}_{il}\left(\Delta_{li}^v(n)-\Delta_{li}^{\overline{v}}(n)-\Delta_{ji}^v(n)+\Delta_{ji}^{\overline{v}}(n)\right).
\end{aligned}
\end{align*}
Thus, we obtain
\begin{align}\label{D-2}
\begin{aligned}
&\sum_{l=1}^{N}\Big[ \phi_{il}\Delta_{li}^v(n)-\overline{\phi}_{il}\Delta_{li}^{\overline{v}}(n)\Big] = -\left(\Delta_{ij}^v(n)-\Delta_{ij}^{\overline{v}}(n)\right)+\sum_{l=1}^{N}\left(\phi_{il}-\overline{\phi}_{il}\right)\Delta_{li}^{v}(n) \\
& \hspace{2cm} +\sum_{l\neq j}\overline{\phi}_{il}\left(\Delta_{li}^v(n)-\Delta_{li}^{\overline{v}}(n)-\Delta_{ji}^v(n)+\Delta_{ji}^{\overline{v}}(n)\right).
\end{aligned}
\end{align}
Similarly, we have
\begin{align}\label{D-3}
\begin{aligned}
& \sum_{l=1}^{N}\Big[\phi_{jl}\Delta_{lj}^v(n)- \overline{\phi}_{jl}\Delta_{lj}^{\overline{v}}(n)\Big] = -\left(\Delta_{ji}^v(n)-\Delta_{ji}^{\overline{v}}(n)\right)+\sum_{l=1}^{N}\left(\phi_{jl}-\overline{\phi}_{jl}\right)\Delta_{lj}^{v}(n)\\
&\hspace{2cm}+\sum_{l\neq i}\overline{\phi}_{jl}\left(\Delta_{lj}^v(n)-\Delta_{lj}^{\overline{v}}(n)-\Delta_{ij}^v(n)+\Delta_{ij}^{\overline{v}}(n)\right).
	\end{aligned}
\end{align}
Subtracting \eqref{D-3} from \eqref{D-2} we get 
\begin{align}
\begin{aligned} \label{D-4}
	&\sum_{l=1}^{N}\Big[\phi_{il}\Delta_{li}^v-\overline{\phi}_{il}\Delta_{li}^{\overline{v}}-\phi_{jl}\Delta_{lj}^v+\overline{\phi}_{jl}\Delta_{lj}^{\overline{v}}\Big]\\
	&\hspace{1.5cm}=-2\left(\Delta_{ij}^v-\Delta_{ij}^{\overline{v}}\right)+\sum_{l=1}^{N}(\phi_{il}-\overline{\phi}_{il})\Delta_{li}^v-\sum_{l=1}^{N}(\phi_{jl}-\overline{\phi}_{jl})\Delta_{lj}^v\\
	&\hspace{1.7cm}+ \underbrace{\sum_{l\neq j}\overline{\phi}_{il}\left(\Delta_{li}^v-\Delta_{li}^{\overline{v}}-\Delta_{ji}^v + \Delta_{ji}^{\overline{v}}\right)-\sum_{l\neq i}\overline{\phi}_{jl}\left(\Delta_{lj}^v-\Delta_{lj}^{\overline{v}}-\Delta_{ij}^v+\Delta_{ij}^{\overline{v}}\right)}_{=:{\mathcal I}_{0}}.
\end{aligned}
\end{align}
The term ${\mathcal I}_{0}$ can be treated as follows:
\begin{align}
\begin{aligned} \label{D-5}
{\mathcal I}_0 
	 &=\sum_{l\neq i,j}\Big[\overline{\phi}_{il}\left(\Delta_{li}^v-\Delta_{li}^{\overline{v}}\right)-\overline{\phi}_{il}\left(\Delta_{ji}^v-\Delta_{ji}^{\overline{v}}\right)-\overline{\phi}_{jl}\left(\Delta_{lj}^v-\Delta_{lj}^{\overline{v}}\right)+\overline{\phi}_{jl}\left(\Delta_{ij}^v-\Delta_{ij}^{\overline{v}}\right)\Big]\\
	&~~+\overline{\phi}_{ii}\left(-\Delta_{ji}^v+\Delta_{ji}^{\overline{v}}\right)-\overline{\phi}_{jj}\left(-\Delta_{ij}^v+\Delta_{ij}^{\overline{v}}\right)\\
	&=\sum_{l\neq i,j}\Big[\overline{\phi}_{il}\left(\Delta_{li}^v-\Delta_{li}^{\overline{v}}\right)-\overline{\phi}_{jl}\left(\Delta_{lj}^v-\Delta_{lj}^{\overline{v}}\right)\Big]+\sum_{l\neq i,j}\left(\overline{\phi}_{il}+\overline{\phi}_{jl}\right)\left(\Delta_{ij}^v-\Delta_{ij}^{\overline{v}}\right)\\
	&~~+\overline{\phi}_{ii}\left(\Delta_{ij}^v-\Delta_{ij}^{\overline{v}}\right)+\overline{\phi}_{jj}\left(\Delta_{ij}^v-\Delta_{ij}^{\overline{v}}\right).
\end{aligned}
\end{align}
Thus, we combine \eqref{D-4} and \eqref{D-5} to obtain
\begin{align}
\begin{aligned} \label{D-6}
& \sum_{l=1}^{N}\Big[\phi_{il}\Delta_{li}^v-\overline{\phi}_{il}\Delta_{li}^{\overline{v}}-\phi_{jl}\Delta_{lj}^v+\overline{\phi}_{jl}\Delta_{lj}^{\overline{v}}\Big]  \\
& \hspace{0.5cm} =-\underbrace{\left[2-\overline{\phi}_{ii}-\overline{\phi}_{jj}-\sum_{l\neq i,j}\left(\overline{\phi}_{il}+\overline{\phi}_{jl}\right)\right]}_{=:\lambda_{ij}>0}\left(\Delta_{ij}^v-\Delta_{ij}^{\overline{v}}\right) +\sum_{l=1}^{N}( \phi_{il}-\overline{\phi}_{il})\Delta_{li}^v \\
& \hspace{0.5cm} -\sum_{l=1}^{N}(\phi_{jl}-\overline{\phi}_{jl})\Delta_{lj}^v +\sum_{l\neq i,j}\Big[\overline{\phi}_{il}\left(\Delta_{li}^v-\Delta_{li}^{\overline{v}}\right)-\overline{\phi}_{jl}\left(\Delta_{lj}^v-\Delta_{lj}^{\overline{v}}\right)\Big].
\end{aligned}
\end{align}
Note that
\[
\lambda_{ij}(n)=2-\sum_{l\neq j}\overline{\phi}_{il}-\sum_{l\neq i}\overline{\phi}_{jl}=2-\left(1-\overline{\phi}_{ij}\right)-\left(1-\overline{\phi}_{ji}\right)=\overline{\phi}_{ij}+\overline{\phi}_{ji}>0\,.
\]
Finally, we combine \eqref{D-1-1}, \eqref{D-6} to obtain
\begin{align*} 
\begin{aligned}
& \Delta_{ij}^v(n+1) -\Delta_{ij}^{\overline{v}}(n+1) =(1-h\kappa\lambda_{ij}(n))(\Delta_{ij}^v(n)-\Delta_{ij}^{\overline{v}}(n)) \\
& \hspace{1cm} +\underbrace{h\kappa\sum_{l=1}^{N}(\phi_{il}-\overline{\phi}_{il})\Delta_{li}^v(n)-h\kappa\sum_{l=1}^{N}(\phi_{jl}-\overline{\phi}_{jl})\Delta_{lj}^{v}(n)}_{:=\mathcal{I}_1^{ij}(n)}\\
&\hspace{1cm}+\underbrace{h\kappa\sum_{l\neq i,j}\Big[\overline{\phi}_{il}\left(\Delta_{li}^v(n)-\Delta_{li}^{\overline{v}}(n)\right)-\overline{\phi}_{jl}\left(\Delta_{lj}^v(n)-\Delta_{lj}^{\overline{v}}(n)\right)\Big]}_{:=\mathcal{I}_2^{ij}(n)}\,.
	\end{aligned}
\end{align*}
\end{proof}
In what follows, we use \cref{thm: flocking} to estimate the above expressions in terms of the $\|\cdot\|_F$-norm. 
Taking the square on both sides of the equality in \cref{L4.2} and sum up  the resulting relation over all $i,j$ to obtain
\begin{align*}
\begin{aligned}
& \left\|\Delta^v(n+1)-\Delta^{\overline{v}}(n+1)\right\|_F^2 \\
& \hspace{1cm} =\sum_{i,j=1}^{N}\left(1-h\kappa\lambda_{ij}\right)^2\left \|\Delta_{ij}^v(n)-\Delta_{ij}^{\overline{v}}(n)\right \|^2+\sum_{i,j=1}^{N} \|\mathcal{I}_1^{ij}+\mathcal{I}_2^{ij} \|^2\\
&\hspace{1.2cm} + 2 \sum_{i,j=1}^{N} \left(1-h\kappa\lambda_{ij}\right)\left(\Delta_{ij}^v(n)-\Delta_{ij}^{\overline{v}}(n)\right)\cdot(\mathcal{I}_1^{ij} + \mathcal{I}_2^{ij})\\
& \hspace{1cm} \leq \sum_{i,j=1}^{N}\left(1-h\kappa\lambda_{ij}\right)^2 \left \|\Delta_{ij}^v(n)-\Delta_{ij}^{\overline{v}}(n)\right \|^2 + 2 \sum_{i,j=1}^{N}\left( \|\mathcal{I}_1^{ij} \|^2 + \|\mathcal{I}_2^{ij} \|^2\right)\\
&\hspace{1.2cm} + 2 \sum_{i,j=1}^{N} \left(1-h\kappa\lambda_{ij}\right)\left \|\Delta_{ij}^v(n)-\Delta_{ij}^{\overline{v}}(n)\right \|\cdot\left( \|\mathcal{I}_1^{ij} \| + \|\mathcal{I}_2^{ij} \|\right).
\end{aligned}
\end{align*}
We use \eqref{D-lambda} to obtain
\begin{align} \label{D-5-0}
\begin{aligned}
& \left\|\Delta^v(n+1)-\Delta^{\overline{v}}(n+1)\right\|_F^2  \\
& \hspace{1cm} \leq (1-h\kappa\lambda)^2\left\|\Delta^v(n)-\Delta^{\overline{v}}(n)\right\|_F^2 + 2\sum_{i,j=1}^{N} \|\mathcal{I}_1^{ij} \|^2 + 2\sum_{i,j=1}^{N} \|\mathcal{I}_2^{ij} \|^2\\
&\hspace{1.2cm} + 2 \left(1-h\kappa\lambda\right)\sum_{i,j=1}^{N} \left\|\Delta_{ij}^v(n)-\Delta_{ij}^{\overline{v}}(n)\right\|\,\left( \|\mathcal{I}_1^{ij} \| + \|\mathcal{I}_2^{ij} \|\right).
\end{aligned}
\end{align}
In the following two lemmas, we provide estimates for $\sum_{i,j=1}^{N} \|\mathcal{I}_1^{ij} \|^2$ and $\sum_{i,j=1}^N \|\mathcal{I}_{2}^{ij} \|^2$.
\begin{lemma}  \label{lem: I1_bound}  
For $i, j \in [N]$,  let ${\mathcal I}_1^{ij}$ be the quantity defined in \cref{L4.2}. Then, we have
\[
\sum_{i,j=1}^{N} \|\mathcal{I}_1^{ij} \|^2  \leq \frac{16h^2\kappa^2L_a^2M^2}{c_1^2}\left(1+\frac{c_2}{c_1}\right)^2\|\Delta^v(n)\|_F^2\,. \]
\end{lemma}
\begin{proof}
Similar to the proof of \cref{lem:a_lip}, we have
\begin{align} 
\begin{aligned} \label{D-5-0-0}
	\left| \phi_{il}-\overline{\phi}_{il}\right|&=\left|\frac{a_{il}}{\sum_{k=1}^{N}a_{ik}}-\frac{\overline{a}_{il}}{\sum_{k=1}^{N}\overline{a}_{ik}}\right|=\left|\frac{a_{il}\sum_k\overline{a}_{ik}-\overline{a}_{il}\sum_ka_{ik}}{\sum_ka_{ik}\sum_k\overline{a}_{ik}}\right|\\\\
	&=\left|\frac{\left(a_{il}-\overline{a}_{il}\right)\sum_k\overline{a}_{ik}+\overline{a}_{il}\left(\sum_k\overline{a}_{ik}-\sum_ka_{ik}\right)}{\sum_ka_{ik}\sum_k\overline{a}_{ik}}\right|\\\\
	&\leq \frac{L_{a}}{Nc_1}\left \|\Delta_{il}^x-\Delta_{il}^{\overline{x}}\right \|+\frac{c_2L_{a}}{N^2c_1^2}\sum_{k=1}^N \left \|\Delta_{ik}^x-\Delta_{ik}^{\overline{x}}\right \|.
	\end{aligned}
\end{align}
We sum the square of \eqref{D-5-0-0} over $l$ to find 
\begin{align}\label{D-5-1}
	\begin{aligned}
	 &\sum_{l=1}^{N}| \phi_{il} - \overline{\phi}_{il}|^2 \\
	& \hspace{1cm} \leq \sum_{l=1}^{N} \left( \frac{L_{a}}{Nc_1}\left \|\Delta_{il}^x-\Delta_{il}^{\overline{x}}\right \|+\frac{c_2L_{a}}{N^2c_1^2}\sum_{k=1}^N \left \|\Delta_{ik}^x-\Delta_{ik}^{\overline{x}}\right \| \right)^2\\
		& \hspace{1cm} \leq \frac{L_a^2}{N^2c_1^2}\sum_{l=1}^{N} \|\Delta_{il}^x - \Delta_{il}^{\overline{x}} \|^2 + \frac{2c_2L_a^2}{N^3c_1^3}\left(\sum_{l=1}^N \|\Delta_{il}^x-\Delta_{il}^{\overline{x}} \|\right)\left(\sum_{k=1}^N \|\Delta_{ik}^x-\Delta_{ik}^{\overline{x}} \|\right) \\
		& \hspace{1cm} + \frac{c_2^2 L_a^2}{N^3c_1^4}\left(\sum_{k=1}^{N} \|\Delta_{ik}^x-\Delta_{ik}^x \|\right)^2\\
		& \hspace{1cm} \leq \frac{L_a^2}{N^2c_1^2} \left(1+\frac{c_2}{c_1}\right)^2 \sum_{l=1}^{N} \|\Delta_{il}^x - \Delta_{il}^{\overline{x}} \|^2\,,
	\end{aligned}
\end{align}
where we used Jensen's inequality for the last bound. 
Next, we use $\Delta_{li}^v(n) = \Delta_{ji}^v(n) + \Delta_{lj}^v(n)$ to find 
\begin{equation*}
	\mathcal{I}_1^{ij} = h\kappa\sum_{l=1}^{N}( \phi_{il} - \overline{\phi}_{il})\Delta^v_{ji}(n) + h\kappa\sum_{l=1}^{N}\left( \phi_{il} - \phi_{jl} -(\overline{\phi}_{il} - \overline{\phi}_{jl})\right)\Delta_{lj}^v(n)\,.
\end{equation*}
This leads to 
\begin{equation} \label{D-6-a}
	\|\mathcal{I}_1^{ij} \| \leq h\kappa\sum_{l=1}^{N}| \phi_{il} - \overline{\phi}_{il}| \|\Delta^v_{ji}(n) \| + h\kappa \sum_{l=1}^{N}\left(|\phi_{il} - \phi_{jl}| + |\overline{\phi}_{il} - \overline{\phi}_{jl}|\right) \|\Delta_{lj}^v(n) \|.
\end{equation}
We take the square on both sides of \eqref{D-6-a} and sum up the resulting relation over all $i, j$ to obtain
\begin{align} \label{eq: I_1bound}
\begin{aligned}
\sum_{i,j=1}^{N} \|\mathcal{I}_1^{ij} \|^2 \leq 2h^2\kappa^2 \sum_{i,j=1}^{N}&\left(\sum_{l=1}^N|\phi_{il} - \overline{\phi}_{il}|^2\sum_{l=1}^{N} \|\Delta_{ji}^v(n) \|^2  \right.
\\& \quad
+ \left. \sum_{l=1}^N\left(|\phi_{il} - \phi_{jl}| + |\overline{\phi}_{il} - \overline{\phi}_{jl}|\right)^2\sum_{l=1}^{N} \|\Delta_{lj}^v(n) \|^2\right),
\end{aligned}
\end{align}
where we used the Cauchy-Schwarz inequality. 

\noindent Now, we use \eqref{D-5-1} to see that the first part in the right-hand side of \eqref{eq: I_1bound} can be bounded by
\begin{align} \label{D-8}
	\begin{aligned}
		&\sum_{i,j=1}^N\left(\sum_{l=1}^{N}| \phi_{il} - \overline{ \phi}_{il}|^2\sum_{l=1}^{N} \|\Delta_{ji}^v(n) \|^2\right)\\
		&\hspace{0.5cm} \leq \frac{L_a^2}{N^2c_1^2} \left(1+\frac{c_2}{c_1}\right)^2 \sum_{i,j=1}^{N} \left(\sum_{l=1}^{N} \|\Delta_{il}^x(n)-\Delta_{il}^{\overline{x}}(n) \|^2 \sum_{l=1}^{N} \|\Delta_{ji}^v(n) \|^2 \right)\\
		&\hspace{0.5cm} \leq \frac{L_a^2}{N^2c_1^2} \left(1+\frac{c_2}{c_1}\right)^2 \left(\sum_{j=1}^{N}\sum_{i,l=1}^{N} \|\Delta_{il}^x(n) - \Delta_{il}^{\overline{x}}(n) \|^2\right)\times \left(\sum_{l=1}^{N}\sum_{j,i=1}^{N} \|\Delta_{ji}^v(n)\|^2\right)\\
		&\hspace{0.5cm} = \frac{L_a^2}{c_1^2} \left(1+\frac{c_2}{c_1}\right)^2\|\Delta^x(n) - \Delta^{\overline{x}}(n)\|_F^2\|\Delta^v(n)\|_F^2\,.
	\end{aligned}
\end{align}
Next, the second part in the right-hand side of \eqref{eq: I_1bound} can be bounded by
\begin{align} \label{D-9}
	\begin{aligned}
		&\sum_{i,j=1}^{N}\sum_{l=1}^{N}\left(| \phi_{il} -  \phi_{jl}| + |\overline{\phi}_{il} - \overline{ \phi}_{jl}|\right)^2\sum_{l=1}^{N} \|\Delta_{lj}^v(n) \|^2\\
		& \hspace{0.5cm}\leq \sum_{i,j=1}^{N} \| \phi\|_{\text{Lip}}^2\sum_{l=1}^N \Big( \|\Delta_{ij}^x(n) \| + \|\Delta_{ij}^{\overline{x}}(n) \| \Big)^2\sum_{l=1}^{N} \|\Delta_{lj}^v(n) \|^2\\
		& \hspace{0.5cm} \leq \sum_{i,j=1}^{N} \| \phi\|_{\text{Lip}}^2N\cdot 2 \left( \|\Delta_{ij}^x(n) \|^2+ \|\Delta_{ij}^{\overline{x}}(n) \|^2\right)\sum_{l=1}^{N} \|\Delta_{lj}^v(n) \|^2\\
		& \hspace{0.5cm} \leq 2\| \phi\|_{\text{Lip}}^2N \left(\sum_{i,j=1}^{N}\left( \|\Delta_{ij}^x(n) \|^2+ \|\Delta_{ij}^{\overline{x}}(n) \|^2\right)\right)\times \left(\sum_{l,j=1}^{N} \|\Delta_{lj}^v(n) \|^2\right)\\
		&\hspace{0.5cm} \leq \frac{2L_a^2}{Nc_1^2}\ \Big( 1+\frac{c_2}{c_1} \Big)^2\left( \|\Delta^x(n)\|_F^2+\|\Delta^{\overline{x}}(n)\|_F^2\right)\|\Delta^v(n)\|_F^2\,,
	\end{aligned}
\end{align}
where we used again \cref{lem:a_lip}.
Now we combine \eqref{eq: I_1bound}, \eqref{D-8}, \eqref{D-9} and use flocking estimate to derive the desired estimate:
\begin{align} \label{D-10}
	\begin{aligned}
		\sum_{i,j=1}^{N} \|\mathcal{I}_1^{ij} \|^2 &\leq 2h^2\kappa^2 \left[\frac{L_a^2}{c_1^2}\left(1+\frac{c_2}{c_1}\right)^2\|\Delta^x(n) - \Delta^{\overline{x}}(n)\|_F^2\right.\\ 
		&\hspace{2cm} + \left.\frac{2L_a^2}{Nc_1^2}\left(1+\frac{c_2}{c_1}\right)^2\left(\|\Delta^x(n)\|_F^2+\|\Delta^{\overline{x}}(n)\|_F^2\right)\right]\|\Delta^v(n)\|_F^2\\
		&\leq \frac{2h^2\kappa^2L_a^2}{c_1^2}\left(1+\frac{c_2}{c_1}\right)^2\Big[2\left(\|\Delta^x(n)\|_F^2 +  \|\Delta^{\overline{x}}(n)\|_F^2\right)\\ 
		&\hspace{4cm}  + 2\left(\|\Delta^x(n)\|_F^2+\|\Delta^{\overline{x}}(n)\|_F^2\right)\Big]\|\Delta^v(n)\|_F^2\\
		&\leq \frac{16h^2\kappa^2L_a^2M^2}{c_1^2}\left(1+\frac{c_2}{c_1}\right)^2\|\Delta^v(n)\|_F^2\,.
	\end{aligned}
\end{align}
\end{proof}

\begin{lemma}  \label{lem: I2_bound}  
For $i, j \in [N]$,  let ${\mathcal I}_2^{ij}$ be the quantity defined in \cref{L4.2}. Then, we have
\[
\sum_{i,j=1}^{N} \|\mathcal{I}_2^{ij} \|^2  \leq 2h^2\kappa^2 \left(\alpha + \frac{L_a^2M^2}{Nc_1^2}\left(1+\frac{c_2}{c_1}\right)^2 \right)\|\Delta^v(n) - \Delta^{\overline{v}}(n)\|_F^2\,, \]
where $\alpha=\alpha(n)$ is given in \eqref{D-alpha}.
\end{lemma}
\begin{proof}
In \eqref{eq: I_1bound}, we use 
\[ \Delta_{li}^v(n) = \Delta_{ji}^v(n) + \Delta_{lj}^v(n) \]
to get 
\begin{equation} \label{D-11}
	\mathcal{I}_2^{ij} = h\kappa\sum_{l\neq i,j}\Big[\overline{ \phi}_{il}\left(\Delta_{ji}^v(n)-\Delta_{ji}^{\overline{v}}(n)\right) + (\overline{ \phi}_{il} - \overline{ \phi}_{jl})\left(\Delta_{lj}^v(n)-\Delta_{lj}^{\overline{v}}(n)\right)\Big].
\end{equation}
We take the square of both sides of \eqref{D-11} and sum up the resulting relation over all $i,j$ to obtain
\begin{align} \label{D-12}
 \begin{aligned}
&\sum_{i,j=1}^{N} \|\mathcal{I}_2^{ij} \|^2 \\
&\hspace{0.3cm}\leq 2h^2\kappa^2 \sum_{i,j=1}^{N} \left(\left \|\sum_{l\neq i,j}\overline{ \phi}_{il}(\Delta_{ji}^v(n) - \Delta_{ji}^{\overline{v}}(n))\right \|^2 + \left \|\sum_{l\neq i,j}(\overline{ \phi}_{il} - \overline{ \phi}_{jl})\left(\Delta_{lj}^v(n) - \Delta_{lj}^{\overline{v}}(n)\right) \right \|^2 \right) \\
&\hspace{0.3cm} =: 2h^2 \kappa^2 ( {\mathcal I}_{21} +   {\mathcal I}_{22} ).
\end{aligned}
\end{align}
By direct estimate, one has 
\begin{align} \label{D-13}
\begin{aligned}
{\mathcal I}_{21} = \sum_{i,j=1}^{N} \left \|\sum_{l\neq i,j}\overline{ \phi}_{il}\left(\Delta^v_{ji}(n) - \Delta^{\overline{v}}_{ji}(n)\right)\right \|^2 &= \sum_{i,j=1}^N \underbrace{\left(1-\overline{ \phi}_{ii} - \overline{\phi}_{ij}\right)^2}_{:=\alpha_{ij}}\left \|\Delta_{ji}^v(n) - \Delta_{ji}^{\overline{v}}(n)\right \|^2,
	\end{aligned}
\end{align}
and
\begin{align} \label{D-14}
\begin{aligned}
{\mathcal I}_{22} &=	 \sum_{i,j=1}^{N}\left \|\sum_{l\neq i,j}(\overline{\phi}_{il} - \overline{\phi}_{jl})\left(\Delta_{lj}^v(n) - \Delta_{lj}^{\overline{v}}(n)\right) \right \|^2 \\
&\leq \sum_{i,j=1}^{N} \left(\sum_{l\neq i,j}\left(\overline{\phi}_{il}-\overline{\phi}_{jl}\right)^2\sum_{l\neq i,j}\left \|\Delta_{lj}^v(n) - \Delta_{lj}^{\overline{v}}(n)\right \|^2\right)\\
&\leq \sum_{i,j=1}^{N}(N-2)\|\phi\|_{\text{Lip}}^2 \|\Delta_{ij}^{\overline{x}} \|^2\sum_{l\neq i,j}\left \|\Delta_{lj}^v(n) - \Delta_{lj}^{\overline{v}}(n)\right \|^2
\\&\leq (N-2)\|\phi \|_{\text{Lip}}^2\|\Delta^{\overline{x}}(n)\|_F^2\|\Delta^v(n) - \Delta^{\overline{v}}(n)\|_F^2\,.\\
	\end{aligned}
\end{align}
Substituting \eqref{D-13}, \eqref{D-14}, \eqref{D-alpha} into  \eqref{D-12}, we obtain 
\begin{align} \label{D-16}
	\begin{aligned}
		\sum_{i,j=1}^N\|\mathcal{I}_{2}^{ij}\|^2 &\leq 2h^2\kappa^2 \left(\alpha + \frac{(N-2)L_a^2}{N^2c_1^2}\left(1+\frac{c_2}{c_1}\right)^2\|\Delta^x(n)\|_F^2 \right)\|\Delta^v(n) - \Delta^{\overline{v}}(n)\|_F^2\\
		&\leq 2h^2\kappa^2 \left(\alpha + \frac{L_a^2M^2}{Nc_1^2}\left(1+\frac{c_2}{c_1}\right)^2 \right)\|\Delta^v(n) - \Delta^{\overline{v}}(n)\|_F^2\,.
	\end{aligned}
\end{align}
\end{proof}

Now we are ready to conclude the proof of \cref{P4.5}.
\begin{proof}[Proof of \cref{P4.5}]
We estimate the last term in right-hand side of \eqref{D-5-0} using the Cauchy-Schwarz inequality and \cref{lem: I1_bound} and \cref{lem: I2_bound} to obtain
\begin{align}
		&\sum_{i,j=1}^{N} \left \|\Delta_{ij}^v(n)-\Delta_{ij}^{\overline{v}}(n)\right \|\cdot\left( \|\mathcal{I}_1^{ij} \| + \|\mathcal{I}_2^{ij} \|\right) \notag\\
		&\hspace{1cm}\leq \left(\sum_{i,j=1}^{N} \|\Delta_{ij}^v(n) - \Delta_{ij}^{\overline{v}}(n) \|^2\right)^{1/2} \left[\left(\sum_{i,j=1}^{N} \|\mathcal{I}_1^{ij} \|^2\right)^{1/2} + \left(\sum_{i,j=1}^{N} \|\mathcal{I}_2^{ij} \|^2\right)^{1/2}\right]\notag\\
		&\hspace{1cm} \leq \sqrt{2}h\kappa\left(\alpha + \frac{L^2_\phi M^2}{Nc_1^2}\left(1+\frac{c_2}{c_1}\right)^2\right)^{1/2}\|\Delta^v(n) - \Delta^{\overline{v}}(n)\|_F^2\notag\\
		&\hspace{1.2cm} +\frac{4h\kappa L_a M}{c_1}\left(1+\frac{c_2}{c_1}\right)\|\Delta^v(n) - \Delta^{\overline{v}}(n)\|_F \cdot \|\Delta^v(n)\|_F\,. \label{D-17}
\end{align}
Finally, we estimate \eqref{D-5-0} using \eqref{D-10}, \eqref{D-16}, and \eqref{D-17} to obtain the desired estimate in \cref{P4.5}:
\begin{align*}
	\begin{aligned}
		&\|\Delta^v(n+1)-\Delta^{\overline{v}}(n+1)\|_F^2\\
		&\hspace{1cm}\leq (1-h\kappa\lambda)^2 \|\Delta^v(n) - \Delta^{\overline{v}}(n)\|_F^2 + \frac{32h^2\kappa^2L_a^2M^2}{c_1^2}\left(1+\frac{c_2}{c_1}\right)^2\|\Delta^v(n)\|_F^2\\
		&\hspace{1.3cm} + 4h^2\kappa^2\left(\alpha+\frac{L_a^2M^2}{Nc_1^2}\left(1+\frac{c_2}{c_1}\right)^2\right)\|\Delta^v(n)-\Delta^{\overline{v}}(n)\|_F^2\\
		&\hspace{1.3cm} + 2\sqrt{2}h\kappa (1-h\kappa \lambda) \left(\alpha+\frac{L_a^2M^2}{c_1^2}\left(1+\frac{c_2}{c_1}\right)^2\right)^{1/2}\|\Delta^v(n) - \Delta^{\overline{v}}(n)\|_F^2\\
		&\hspace{1.3cm} + \frac{8h\kappa L_a M}{c_1}(1-h\kappa\lambda)\left(1+\frac{c_2}{c_1}\right)\|\Delta^v(n) - \Delta^{\overline{v}}(n)\|_F\|\Delta^v(n)\|_F.
	\end{aligned}
\end{align*}
\end{proof}

\subsection{Decay of velocity shape discrepancy}\label{sec:4.2}
The main goal of this section is to prove the following restatement of $\eqref{D-0-0}_2$. 
\begin{lemma}\label{lem:D-0-0-b}
For $h>0$ sufficiently small, suppose that framework $(\mathcal{F}1) - (\mathcal{F}2)$ holds and let $(X,V)$ and $(\overline{X}, \overline{V})$ be solutions to \eqref{eq: discrete_MT}, respectively. For any $0\le\epsilon\le 1$ there exists $b_1=b_1(\epsilon,h), b_2=b_2(h)>0$ given by
\begin{align*}
b_1(\epsilon,h)&:=\left(\bar c_0 \epsilon+\bar c_1 h + \bar c_2 h^2\right)^{1/2}\,,\qquad b_2(h):=C\kappa \psi(M)h
\end{align*}
for constants $\bar c_0, \bar c_1, \bar c_2$ depending on the initial condition $V(0)$, $\overline{V}(0)$ and all model parameters,
such that
\begin{equation}\label{D-19}
		\|\Delta^v(n+1) - \Delta^{\overline{v}}(n+1)\|_F \leq (1-\epsilon)\|\Delta^v(n) - \Delta^{\overline{v}}(n)\|_F + b_1e^{-b_2n}\,.
\end{equation}  
\end{lemma}
\begin{proof}
We have from \cref{P4.5} that
    \[\|\Delta^v(n+1)-\Delta^{\overline{v}}(n+1)\|_F^2\leq (1-\epsilon)^2 \|\Delta^v(n) - \Delta^{\overline{v}}(n)\|_F^2 +\mathcal{I}_{3},\]
    where the term $ \mathcal{I}_{3}$ is defined as follows: 
    \begin{align*}
    \mathcal{I}_{3} &:= (\cC_1-1+2\epsilon-\epsilon^2) \|\Delta^v(n) - \Delta^{\overline{v}}(n)\|_F^2 + \cC_2 \|\Delta^v(n) - \Delta^{\overline{v}}(n)\|_F\|\Delta^v(n)\|_F + \cC_3 \|\Delta^v(n)\|_F^2\\
    &\le (\cC_1-1+2\epsilon-\epsilon^2 + \frac{\cC_2}{2}) \|\Delta^v(n) - \Delta^{\overline{v}}(n)\|_F^2 
    +\left(\cC_3+\frac{\cC_2}{2}\right) \|\Delta^v(n)\|_F^2\,.\\
    \end{align*}
    Note that 
\[\overline{\phi}_{ij}(n) = \frac{\overline{a}_{ij}}{\sum_{k=1}^{N}\overline{a}_{ik}}\geq \frac{c_1}{Nc_2} \quad \mbox{and} \quad \frac{2c_1}{Nc_2} \leq \lambda(n) \leq 2, \]
Thus, we have
\[\alpha(n)\leq 1.\] 
Choosing a sufficiently small $h>0$ such that
\begin{equation*} 
	0 < h < \frac{c_1}{\kappa N c_2\left\{2
  +\frac{L_a^2M^2}{Nc_1^2}\left(1+\frac{c_2}{c_1}\right)^2\right\}}\,,
\end{equation*}
we estimate from above independent of $n$,
      \begin{align*}
  &\cC_1(n,h)-1+2\epsilon-\epsilon^2\\
  &\quad\le\epsilon(2-\epsilon)-2h\kappa\frac{2c_1}{Nc_2}+ 4h^2\kappa^2
  + 4h^2\kappa^2\left(1+\frac{L_a^2M^2}{Nc_1^2}\left(1+\frac{c_2}{c_1}\right)^2\right)\\
   &\qquad+ 2\sqrt{2}h\kappa \left(1-h\kappa \frac{2c_1}{Nc_2}\right) \left(1+\frac{L_a^2M^2}{c_1^2}\left(1+\frac{c_2}{c_1}\right)^2\right)^{1/2}\\
    &\quad\le\epsilon(2-\epsilon)-4h\kappa\left[\frac{c_1}{Nc_2}- h\kappa\left\{2
  +\frac{L_a^2M^2}{Nc_1^2}\left(1+\frac{c_2}{c_1}\right)^2\right\}\right]\\
   &\qquad+ 2\sqrt{2}h\kappa  \left(1+\frac{L_a^2M^2}{c_1^2}\left(1+\frac{c_2}{c_1}\right)^2\right)^{1/2}\\
    &\quad\le\epsilon(2-\epsilon)+ 2\sqrt{2}h\kappa  \left(1+\frac{L_a M}{c_1}\left(1+\frac{c_2}{c_1}\right)\right)\,.
\end{align*}
In addition, for small enough $h>0$,
\begin{align*}
      \cC_2(n,h)&\le\frac{8h\kappa L_a M}{c_1}\left(1-h\kappa\frac{2c_1}{Nc_2}\right)\left(1+\frac{c_2}{c_1}\right)
      \le\frac{8h\kappa L_a M}{c_1}\left(1+\frac{c_2}{c_1}\right)
      \,,\\
    \cC_3(n,h)&=\frac{32h^2\kappa^2L_a^2M^2}{c_1^2}\left(1+\frac{c_2}{c_1}\right)^2\,.
\end{align*}
Now we use the exponential decay of $\|\Delta^{v}(n)\|_F$ and $\|\Delta^{\overline{v}}(n)\|_F$ by \cref{thm: flocking} to see
\begin{align*}
    \mathcal{I}_{3} 
    &\le \left(\epsilon(2-\epsilon)+ 2\sqrt{2}h\kappa  \left(1+\frac{L_aM}{c_1}\left(1+\frac{c_2}{c_1}\right)\right) +\frac{4h\kappa L_a M}{c_1}\left(1+\frac{c_2}{c_1}\right) \right) \|\Delta^v(n) - \Delta^{\overline{v}}(n)\|_F^2 \\
    &\quad +\left(\frac{32h^2\kappa^2L_a^2M^2}{c_1^2}\left(1+\frac{c_2}{c_1}\right)^2
    +\frac{4h\kappa L_a M}{c_1}\left(1+\frac{c_2}{c_1}\right)\right) \|\Delta^v(n)\|_F^2\\
    &\le \left(\epsilon(2-\epsilon)+ 2\sqrt{2}h\kappa  \left(1+\frac{L_aM}{c_1}\left(1+\frac{c_2}{c_1}\right)\right) +\frac{4h\kappa L_a M}{c_1}\left(1+\frac{c_2}{c_1}\right) \right) \times 2(\|\Delta^v(n)\|_F^2\\
    &\quad  + \|\Delta^{\overline{v}}(n)\|_F^2)+\left(\frac{32h^2\kappa^2L_a^2M^2}{c_1^2}\left(1+\frac{c_2}{c_1}\right)^2
    +\frac{4h\kappa L_a M}{c_1}\left(1+\frac{c_2}{c_1}\right)\right) \|\Delta^v(n)\|_F^2\\
    &\le \left(2\epsilon+ 8h\kappa  \left(1+\frac{L_aM}{c_1}\left(1+\frac{c_2}{c_1}\right)\right)\right) \cdot 2(\|\Delta^v(0)\|_F^2 + \|\Delta^{\overline{v}}(0)\|_F^2) e^{-2C\kappa \psi(M)nh}  \\
    &\quad +\left(\frac{32h^2\kappa^2L_a^2M^2}{c_1^2}\left(1+\frac{c_2}{c_1}\right)^2
    +\frac{4h\kappa L_a M}{c_1}\left(1+\frac{c_2}{c_1}\right)\right)\|\Delta^v(0)\|_F^2 e^{-2C\kappa \psi(M)nh} \\
    &= b_1^2 e^{-2b_2 n},
    \end{align*}
where
\begin{align*}
b_1(\epsilon,h)&:=\left(\bar c_0 \epsilon+\bar c_1 h + \bar c_2 h^2\right)^{1/2}\,,\qquad b_2(h):=C\kappa \psi(M)h,
\end{align*}
for constants $\bar c_0, \bar c_1, \bar c_2$ depending on the initial condition $V(0)$, $\overline{V}(0)$ and all model parameters,
\begin{align*}
    \bar c_0&:=4(\|\Delta^v(0)\|_F^2 + \|\Delta^{\overline{v}}(0)\|_F^2)\,,\\
    \bar c_1&:= 16\kappa  \left(1+\frac{L_aM}{c_1}\left(1+\frac{c_2}{c_1}\right)\right)  (\|\Delta^v(0)\|_F^2 + \|\Delta^{\overline{v}}(0)\|_F^2)
    + \frac{4\kappa L_a M}{c_1}\left(1+\frac{c_2}{c_1}\right)\|\Delta^v(0)\|_F^2\,,\\
    \bar c_2&:=\frac{32\kappa^2L_a^2M^2}{c_1^2}\left(1+\frac{c_2}{c_1}\right)^2
    \|\Delta^v(0)\|_F^2\,.
\end{align*}
\end{proof}

\begin{corollary}\label{cor: v_fluctuation_bound}
    For $h>0$ sufficiently small, suppose that framework $(\mathcal{F}1) - (\mathcal{F}2)$ holds and let $(X,V)$ and $(\overline{X}, \overline{V})$ be solutions to \eqref{eq: discrete_MT}, respectively. Then for any $0<\epsilon < 1$,
    \[\|\Delta^v(n) - \Delta^{\overline{v}}(n)\|_F\leq (1-\epsilon)^n\|\Delta^v(0) - \Delta^{\overline{v}}(0)\|_F + \frac{b_1e^{b_2}}{1-(1-\epsilon)e^{b_2}}e^{-b_2n}\,.\]
\end{corollary}
\begin{proof}
We use the mathematical induction on \eqref{D-19} to get
\begin{align*}
	\|\Delta^v(n) - \Delta^{\overline{v}}(n)\|_F&\leq (1-\epsilon)^{n}\|\Delta^v(0) - \Delta^{\overline{v}}(0)\|_F + b_1\sum_{k=0}^{n-1}(1-\epsilon)^ke^{-b_2(n-1-k)}\\
	& = (1-\epsilon)^n\|\Delta^v(0) - \Delta^{\overline{v}}(0)\|_F + b_1e^{-b_2(n-1)}\frac{\left((1-\epsilon)e^{b_2}\right)^{n}-1}{(1-\epsilon)e^{b_2}-1}\\
        & = (1-\epsilon)^n\|\Delta^v(0) - \Delta^{\overline{v}}(0)\|_F + \frac{b_1e^{b_2}}{(1-\epsilon)e^{b_2}-1}\left((1-\epsilon)^n-e^{-b_2n}\right)\\
        &\leq (1-\epsilon)^n\|\Delta^v(0) - \Delta^{\overline{v}}(0)\|_F + \frac{b_1e^{b_2}}{1-(1-\epsilon)e^{b_2}}e^{-b_2n}\,.
        \end{align*}
Here, we used
\begin{equation*}
    \frac{b_1e^{b_2}}{(1-\epsilon)e^{b_2}-1} < 0\,
\end{equation*}
for small enough $h>0$ since
\begin{equation*}
    \lim_{h\to0}\frac{b_1(\epsilon,h)e^{b_2(h)}}{(1-\epsilon)e^{b_2(h)}-1} = -\sqrt{\frac{\bar c_0}{\epsilon}} < 0\,.
\end{equation*}

\end{proof}

\begin{remark}\label{rmk:orbital-stability}
    In \cref{thm: stability}, we showed that the shape discrepancy of two pairs of solutions remains uniformly bounded by the shape discrepancy between two sets of initial data. The conditions $\mathcal{F}1$ and $\mathcal{F}2$ for this result are imposed to guarantee the formation of mono-cluster flocking. We can view our stability result as a kind of orbital stability of solution structure. To see this, we recall the analogous result of viscous conservation laws. In viscous conservation laws, the speed of the shock structure (traveling wave connecting two constant far-field states) is determined by the Rankine-Hugoniot condition. Moreover, the perturbed shock structure converges to the shifted original shock structure with the same shape. This phenomenon is called the orbital stability of the shock structure and the shift is also determined by the mass of a perturbation via the conservation law. In our case, the MT model does not have momentum conservation, so we cannot know the asymptotic velocity a priori. Nevertheless, our uniform stability result illustrates that the differences in the shape of the configurations can be uniformly controlled by that of the initial data. This is why we can interpret our uniform stability results as a kind of orbital stability of solution structure. 
\end{remark}

\section{Numerical Simulations} \label{sec:5}
\setcounter{equation}{0}
In this section, we numerically illustrate and substantiate the theoretical results in \cref{thm: flocking} and \cref{cor: v_fluctuation_bound}.  For the communication function $a$ we consider a typical choice from the Cucker-Smale literature that satisfies all the conditions in \cref{as: a}: For some positive constant $\beta>0$, and $c_2 > c_1 > 0$, 
\begin{equation}\label{eq:a-Sec5}
    a(s) = c_1 + \frac{c_2 - c_1}{ (1 +s^2)^{\beta}}\,.
\end{equation}
Unlike the results for the continuous-time model in \cite{MT14}, where unconditional flocking is guaranteed, our flocking result imposes certain assumptions on the initial configuration and achieves therefore only \emph{conditional flocking}. We recall that these conditions are given by \eqref{eq: initial}:
\begin{align*}
    &\|\Delta^x(0)\|_F < M:=\frac{1}{4N\phiLip}\,, \qquad
    \|\Delta^v(0)\|_F < \int_{\|\Delta^x(0)\|_F}^M \kappa \psi(s) ds\,.
\end{align*} 
Since $N\phiLip = \frac{L_{a} (1 + c_2/c_1)}{c_1}$ is independent of $N$, the value is affected by $L_{a}, c_1, c_2 > 0$. It is easy to observe that $L_{a}$ decreases (increases) as the exponent $\beta$ decreases (increases). This will make $M$ larger (smaller). Hence, $M$ cannot be made arbitrarily large, but choosing small $\beta > 0$ allows us to consider less restricted initial particle configurations. An explicit calculation leads to 
\begin{align*}
    &a'(s) = -2s\beta(c_2 - c_1) (1 + s^2)^{-\beta - 1}\,,\\
    &a''(s) = 2(c_2 - c_1) \beta (1 + s^2)^{-\beta -1}\{2(\beta + 1) (1+ s^2)^{-1}s^2 - 1 \}\,.
\end{align*}
Since $a''(s) = 0$ at $ s^*= \sqrt{\frac{1}{2\beta + 1}}$ and $a''(s)<0$ for $0\le s<s^*$ and $a''(s)<0$ for $s>s^*$,  the function $a':[0,\infty)\to (-\infty,0]$ has a minimum at $s^*$. Hence, $|a'|$ has a maximum at $s^*$, and so for all $s\ge 0$,
\begin{align*}
    \left|a'(s)\right|\leq \left|a'(s^*)\right| = \frac{2\beta(c_2 - c_1) }{\sqrt{2\beta + 1}}\left(1 + \frac{1}{2\beta + 1}\right)^{-\beta - 1} =: L_{a}\,.
\end{align*}
Hence, we get \begin{align*}
    M = \frac{1}{4N\phiLip} = \frac{c_1}{4L_{a} (1 + c_2/c_1)} = \frac{c_1\sqrt{2\beta + 1}}{8\beta(c_2 - c_1) (1 + c_2/c_1)} \left(1 + \frac{1}{2\beta + 1}\right)^{\beta +1}\,.
\end{align*}
\begin{figure}[ht!]
    \centering
\includegraphics[width = 0.7\linewidth]{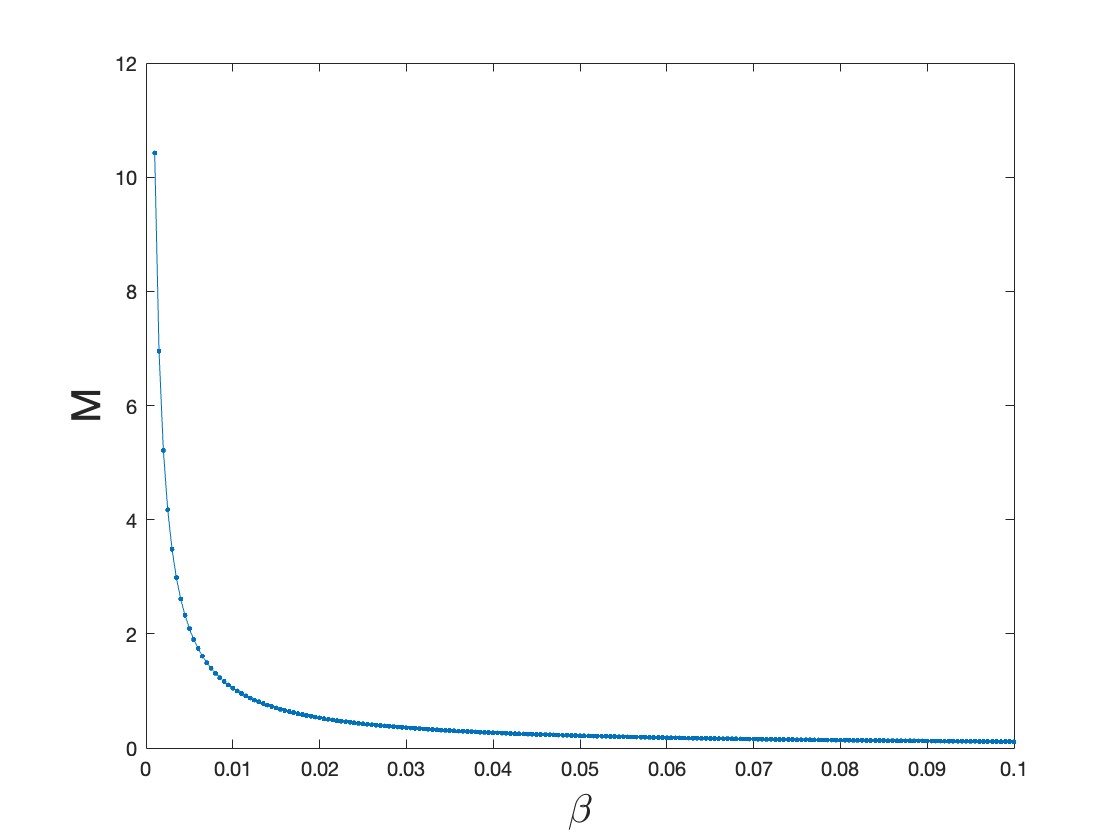}
    \caption{Bound $M$ as a function of $\beta$ for the communication weight $a$ as defined in \eqref{eq:a-Sec5} with $c_1=0.1$ and $c_2=0.5$.}
    \label{fig:M_beta}
\end{figure}
Figure~\ref{fig:M_beta} shows $M$ as a function of $\beta$ with $c_1 = 0.1, c_2 = 0.5$.
To test the robustness of our analysis, we consider a wide range of exponents $0 < \beta < 1$.
For the Cucker-Smale model,  $\beta = 1/2$ is a critical exponent for the transition between conditional and unconditional flocking \cite{H-Liu}. Although in the case of the MT model, \cref{thm: conti_flocking} and \cref{thm: flocking} are conditional flocking results for any choice of $\beta \geq 0$, setting $\beta$ closer to 1 imposes a more restrictive initial condition to achieve flocking. This is consistent with the original Cucker-Smale model, where setting $ \beta > 1/2$ leads to conditional flocking, whereas $\beta \in [0, 1/2]$ corresponds to unconditional flocking.
To set the initial values of $(X, V)$, we randomly sample either from a uniform distribution or from a normal distribution truncated to be supported on a fixed interval. Throughout this section, we fixed $\kappa = 1, h = 0.01, c_1 =0.1, c_2 = 0.5$ and $d=2$ or $d=4$ as changing the communication strength, step size, $c_1, c_2$ values, and dimension of the space do not affect the trend of the dynamics much. 
As long as particles are initialized conforming to the conditions \eqref{eq: initial}, our numerical simulations exhibit the decay and boundedness behavior predicted by \cref{thm: flocking}.
%
Figures~\ref{fig:deviation_allbetas} and \ref{fig:rate_b} show the deviation of positions and velocities of particles $\|\Delta^x\|_F, \|\Delta^v\|_F$, respectively, for different values of $\beta$ ($10^{-1}, 10^{-2}$ and $10^{-3}$). 
The numerical solution to \eqref{eq: discrete_MT} is shown as a fat line, and the theoretical upper bound $M$ from \cref{thm: flocking} is shown as a dotted line. 
\begin{figure}[ht!]
    \hspace*{-0.5cm}
    \includegraphics[width=0.5\linewidth]{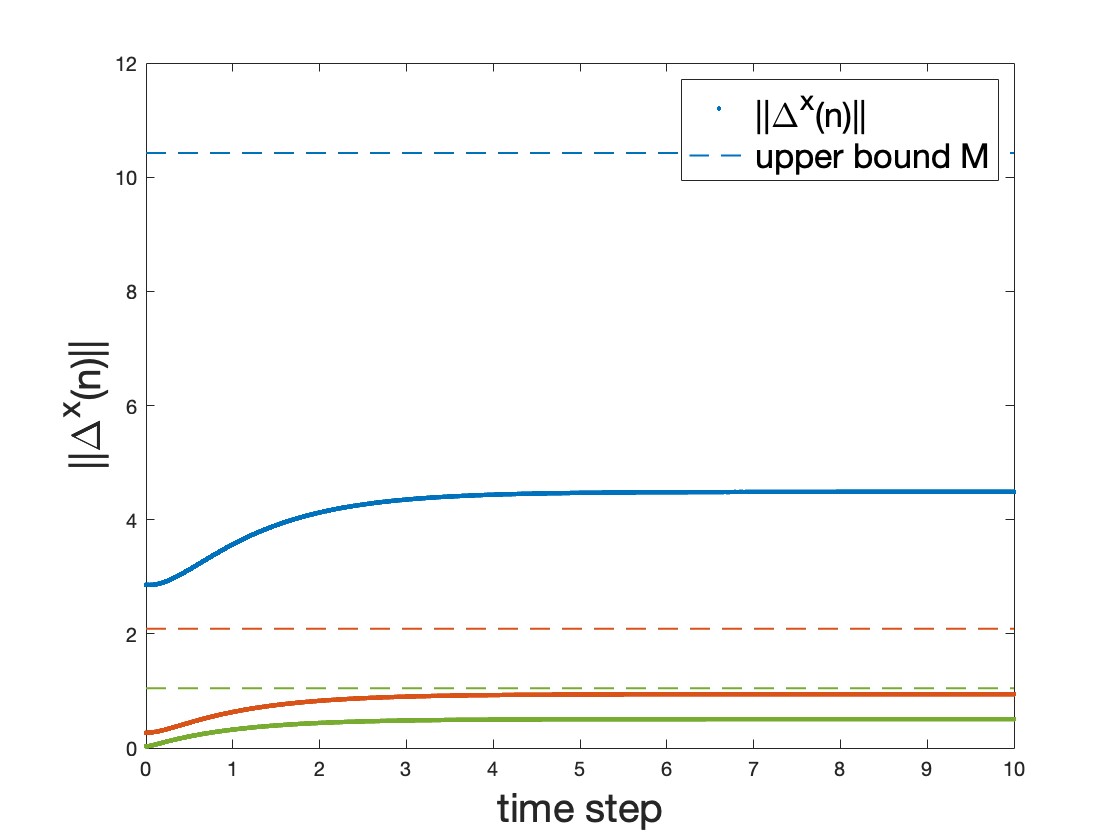}
    \includegraphics[width=0.5\linewidth]{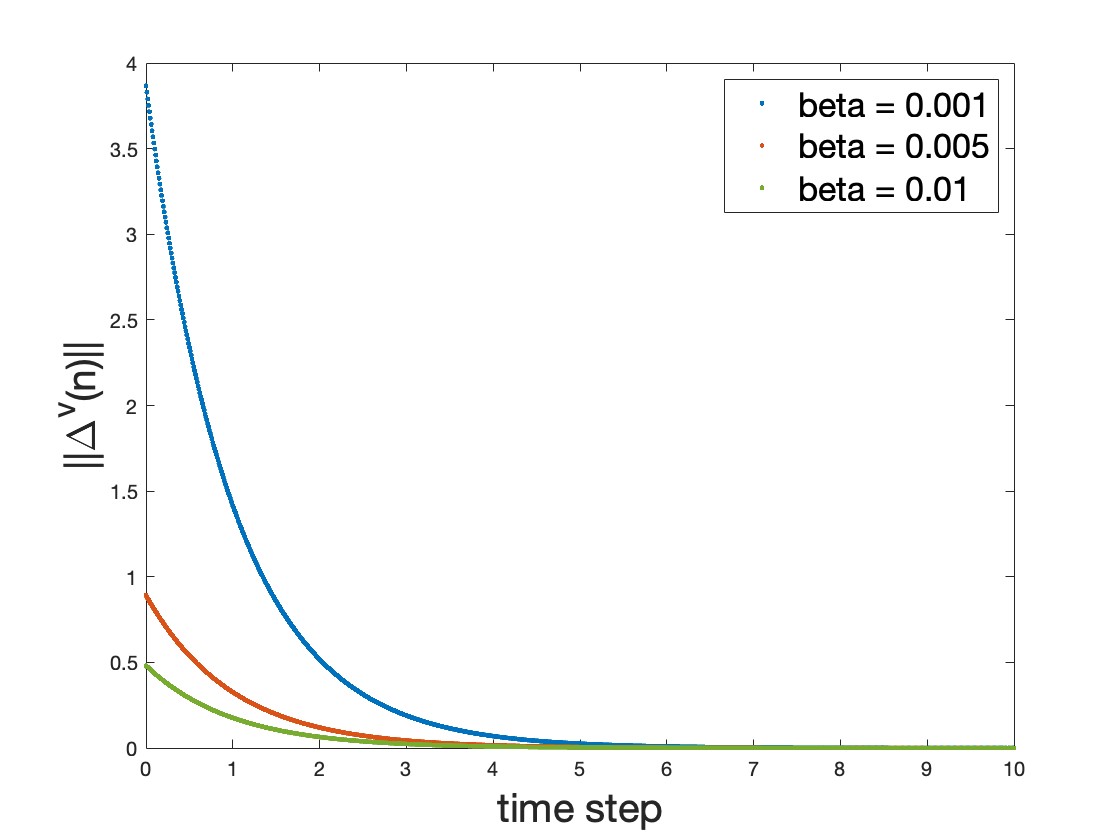}
    \captionof{figure}{Deviation of position (left) and velocity (right) for  $\beta = 0.001, 0.005, 0.01$ and $d=4$. The behavior is as predicted in \cref{thm: flocking}.}
    \label{fig:deviation_allbetas}
\end{figure}
\begin{figure}[ht!]
    \hspace*{-0.5cm}
    \includegraphics[width=0.6\linewidth]{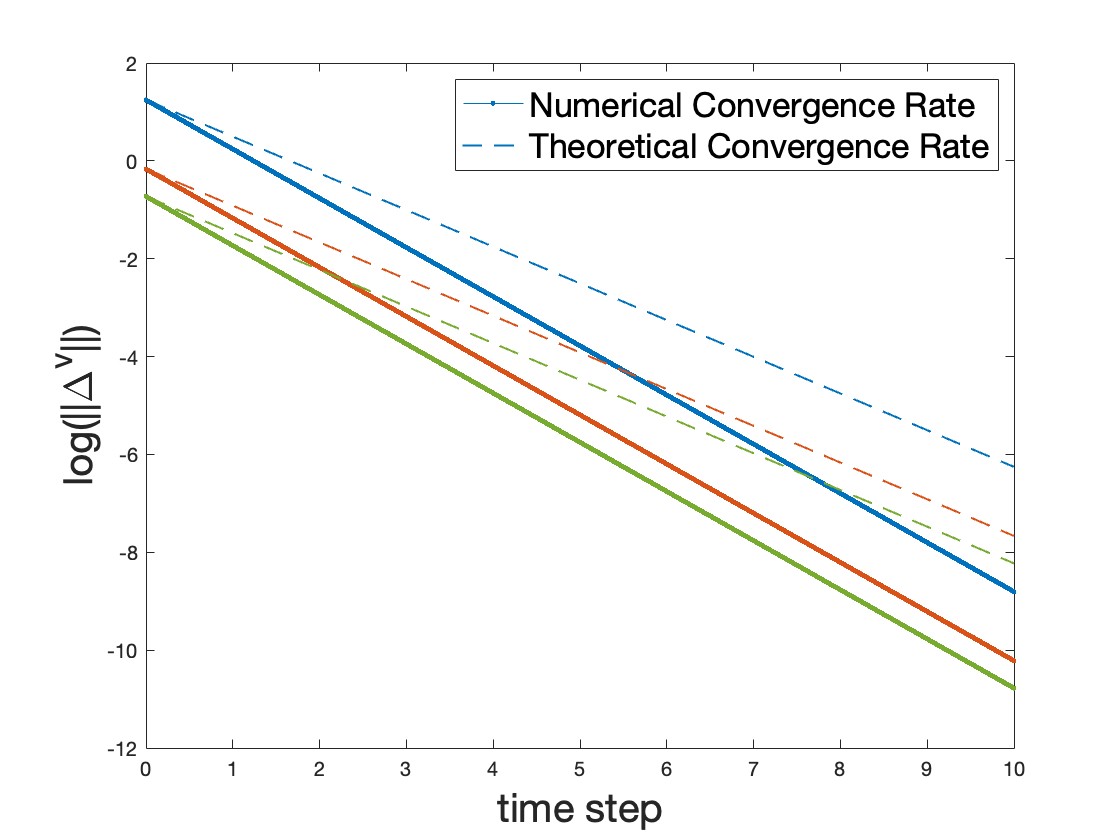}
    \captionof{figure}{Convergence of $\log(\|\Delta^v (n)\|_F)$ for $\beta = 0.001, 0.005, 0.01$ and $d=4$.}
    \label{fig:rate_b}
\end{figure}

From \eqref{eq: discrete_v_flocking}, the decay for $\|\Delta^v\|_F$ is exponential with rate close to $-\kappa \psi(M)$. As can be seen in Figure~\ref{fig:rate_b}, for $(X(n), V(n))$ solving \eqref{eq: discrete_MT} the slope of $\|\Delta^v(n)\|_F$ may be steeper in practice than that of the theoretical prediction. A similar conclusion can be made regarding the analysis for $\|\Delta^x(n)\|_F$, and it is, in fact, the non-optimality of the bound in position fluctuations via $M$ that is the source of information loss for $\|\Delta^v(n)\|_F$. Setting $M$ smaller than the predicted value in \cref{thm: flocking} choosing for example $\frac{1}{\tilde{C}N\aLip}$ with $\tilde{C} > 4$ can enhance the decay rate, but then also imposes stronger conditions on the initial configurations. This illustrates a trade-off between faster theoretical convergence rates and more restrictive initialization.

However, flocking may still occur even if the conditions of \cref{thm: flocking} are not satisfied. For instance, Figure~\ref{fig:dynamics_b = 0.001} demonstrates the emergence of flocking for $\beta=0.001$ for randomly sampled initial positions and velocities. The initial positions were sampled from a normal distribution with mean and variance tuned to violate the condition \eqref{eq: initial}. For such choice of initial data, even though the solution $\left(X(n), V(n)\right)_{n\in \N}$ iteratively solving \eqref{eq: discrete_MT} 
does not satisfy the uniform upper bound $M$ in position fluctuations $\|\Delta^x(n)\|_F$, the fluctuations in velocities $\|\Delta^v(n)\|_F$ are still seen to exhibit exponential decay; see Figure~\ref{fig:flocking_b = 0.001}. In the left figure of Figure~\ref{fig:flocking_b = 0.001}, we can check that the position deviation $\|\Delta^x(n)\|_F$ is not uniformly bounded by $M$ necessarily, but by some larger constant. This shows that our theoretical estimates to guarantee flocking are rather coarse. 
\begin{figure}[ht!]
\includegraphics[width = 0.3\linewidth]{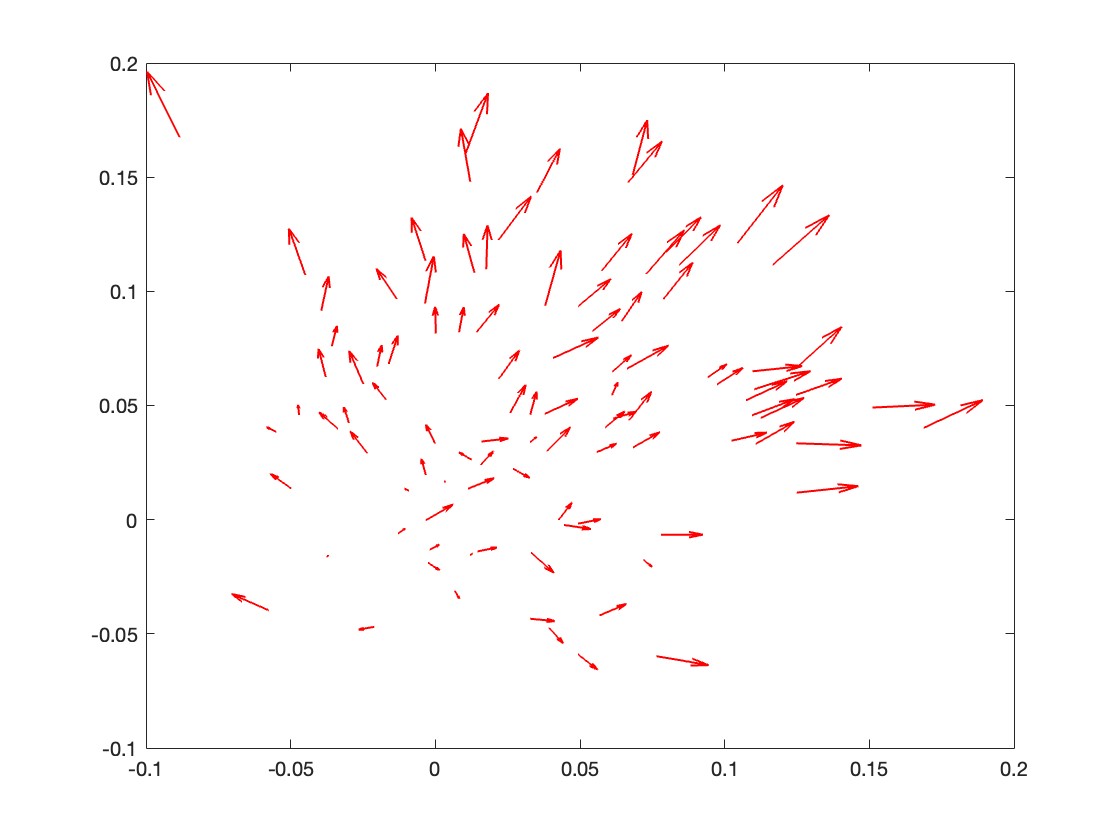}
\includegraphics[width = 0.3\linewidth]{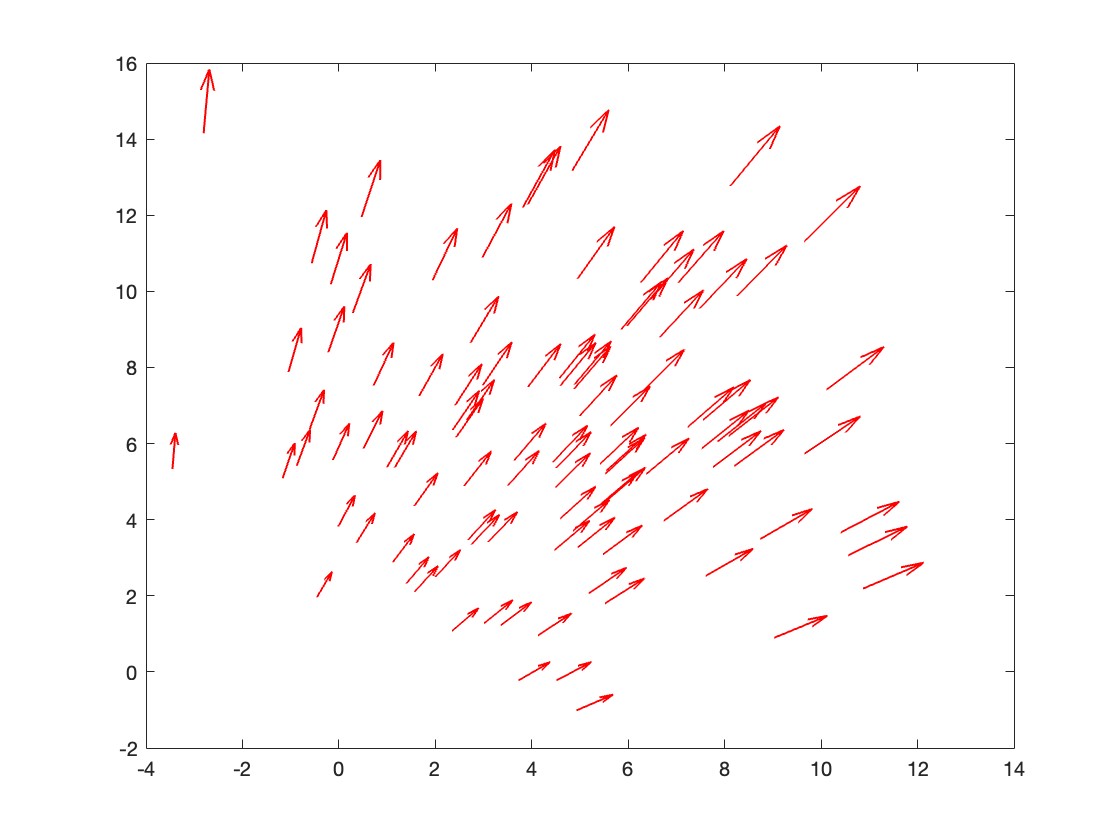}
\includegraphics[width = 0.3\linewidth]{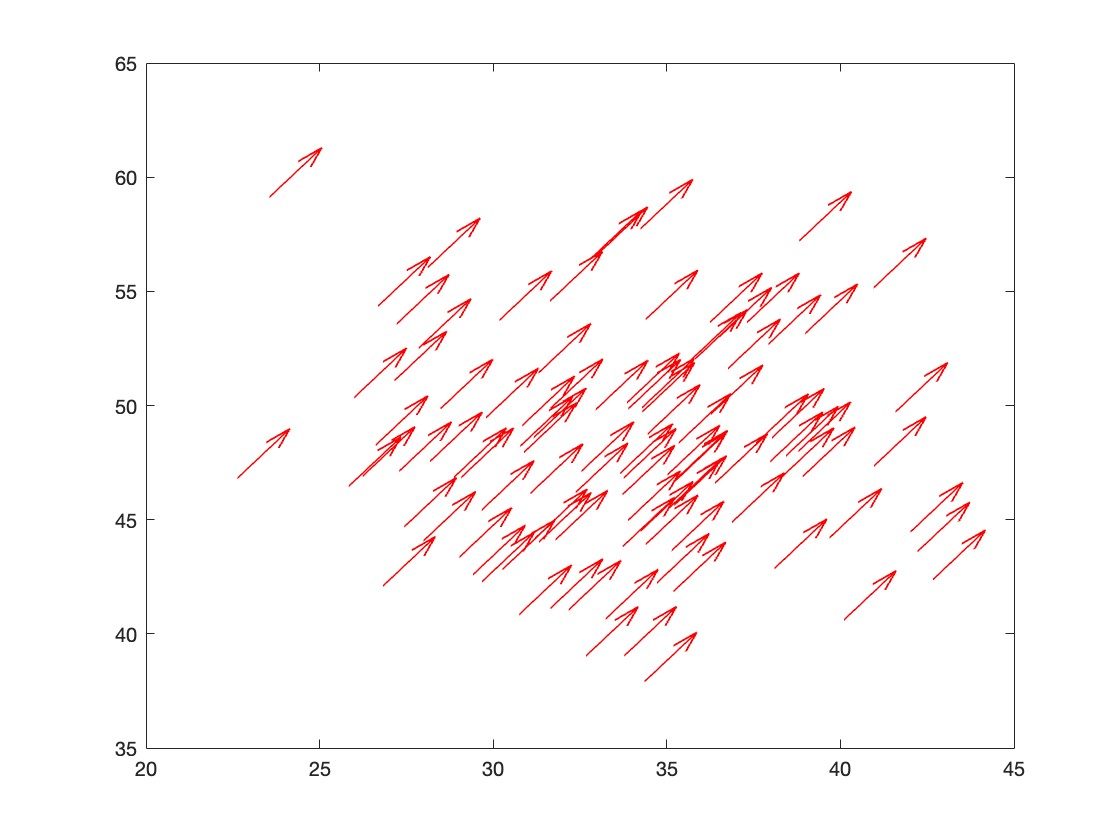}
    \caption{ Distributions of particles and their velocity vectors at time $t= 0, 1.25, 10$ for $\beta = 0.001$ and $d=2$ for a setting not satisfying \eqref{eq: initial}.}
    \label{fig:dynamics_b = 0.001}
\end{figure}
\begin{figure}[ht!]
    \hspace*{-0.5cm}
\includegraphics[width = 0.5\linewidth]{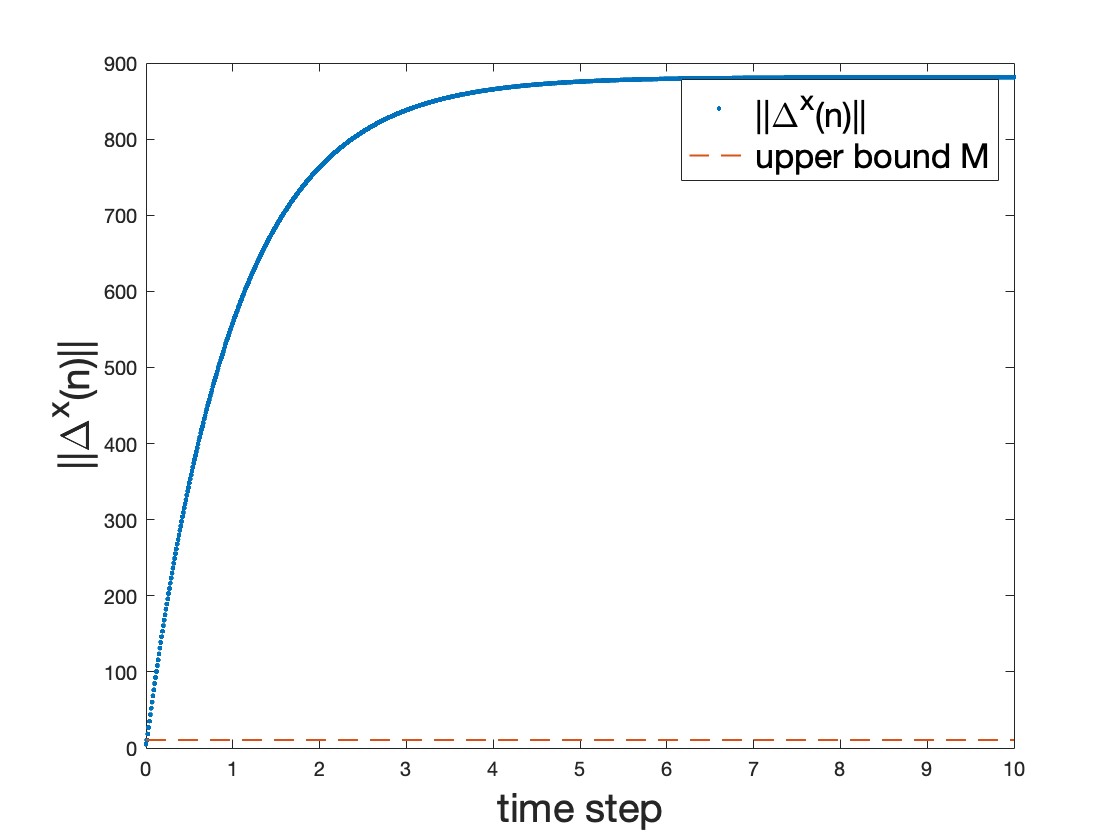}
\includegraphics[width = 0.5\linewidth]{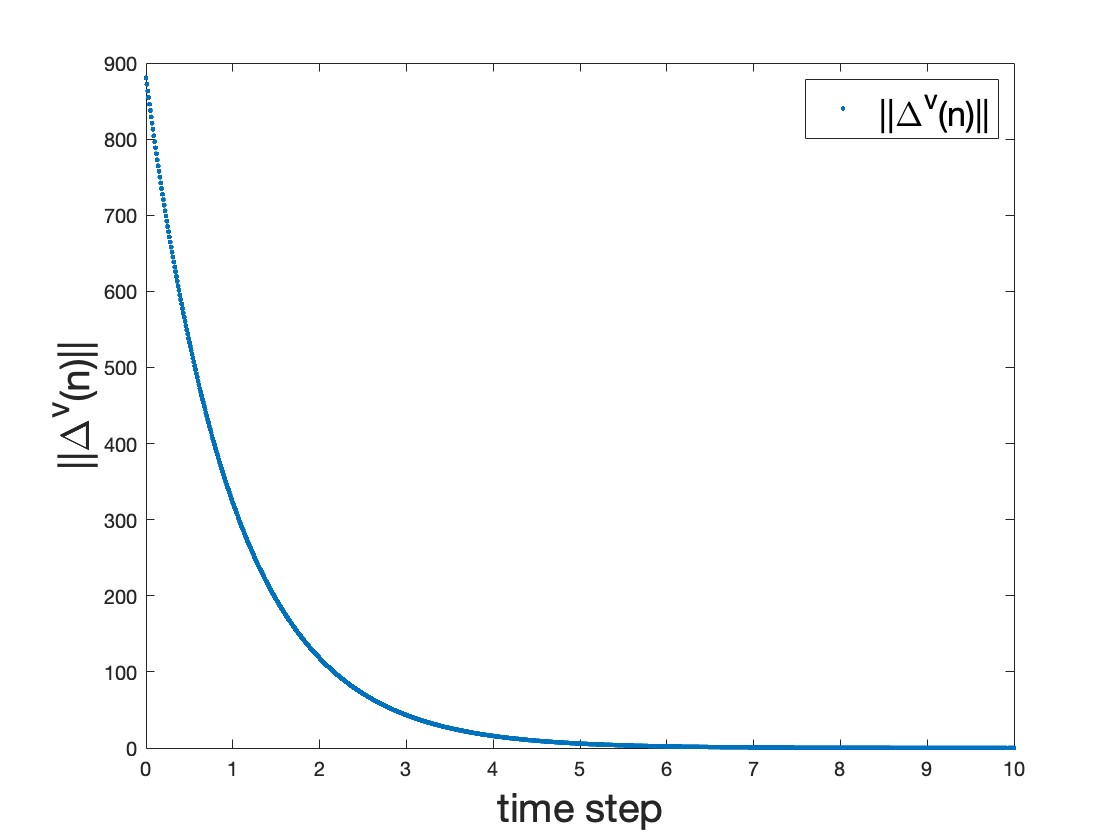}
    \caption{$\|\Delta^x(n)\|_F$ (left) and $\|\Delta^v(n)\|_F$ (right) corresponding to the particle evolutions from Figure~\ref{fig:dynamics_b = 0.001} for $\beta = 0.001$ and $d=2$. Flocking emerges even though the conditions of \cref{thm: flocking} are not satisfied.}
    \label{fig:flocking_b = 0.001}
\end{figure}
See Figure~\ref{fig:dynamics_b= 0.1} and \ref{fig:flocking_b = 0.1} for a similar behavior in the case $\beta=0.1$.

\begin{figure}[ht!]
\includegraphics[width = 0.3\linewidth]{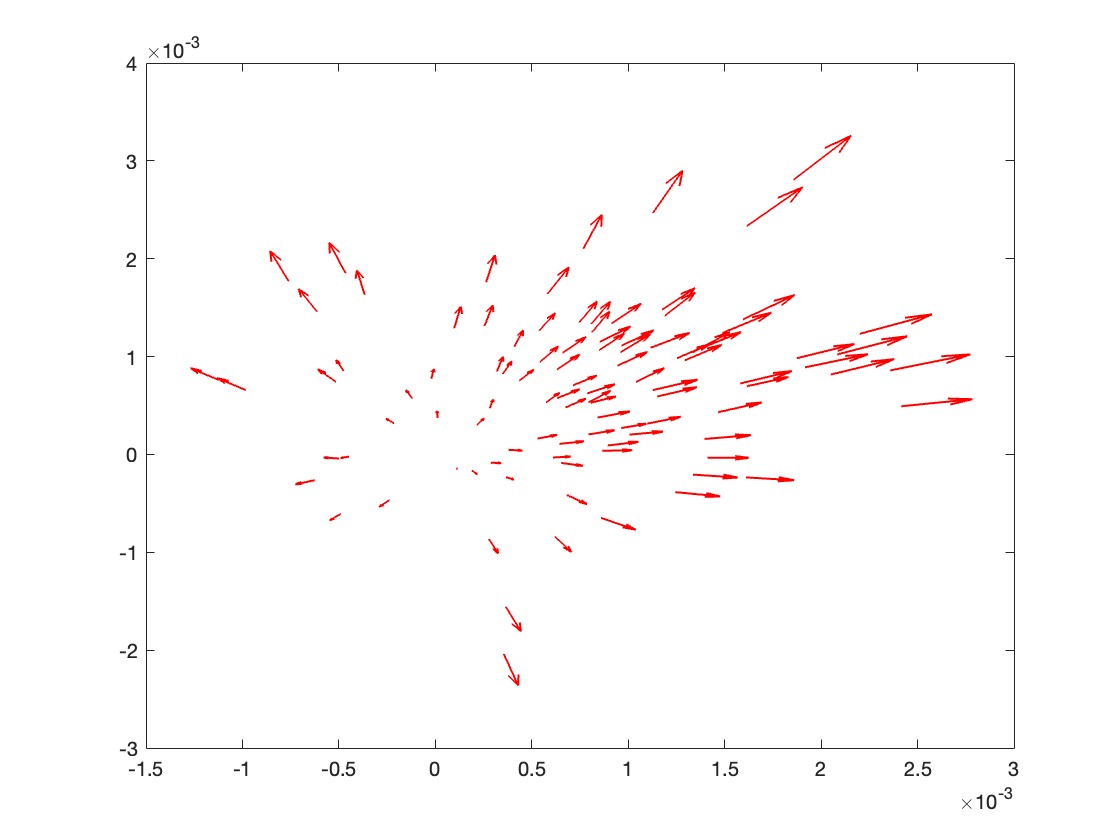}
\includegraphics[width = 0.3\linewidth]{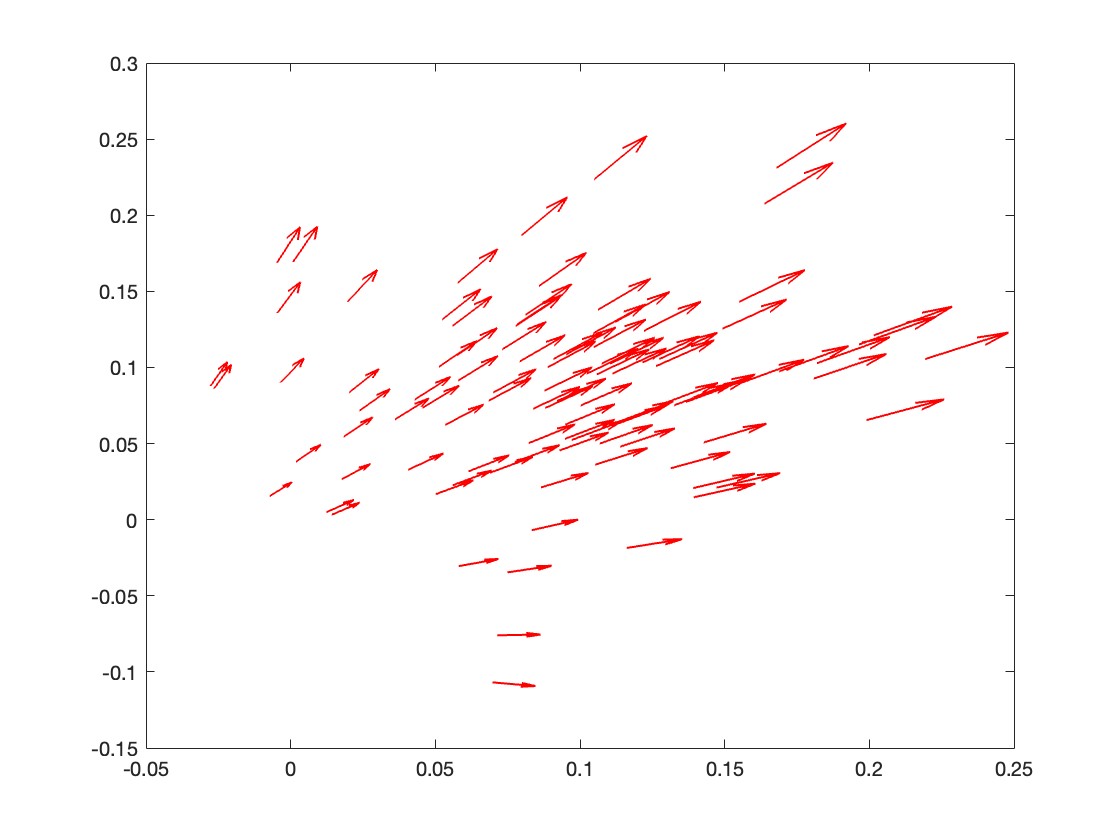}
\includegraphics[width = 0.3\linewidth]{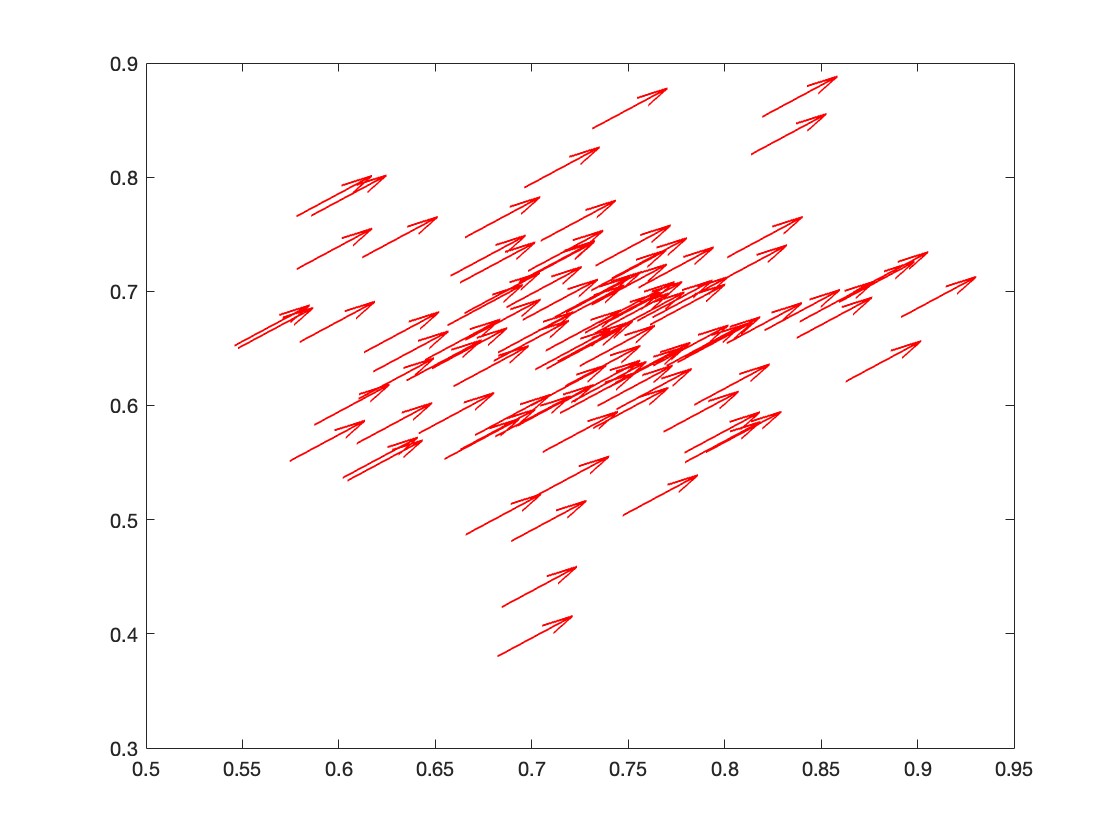}
    \caption{ Distributions of particles and their vectors at time $t= 0, 1.25, 10$ for $\beta = 0.1$}
    \label{fig:dynamics_b= 0.1}
\end{figure}
\begin{figure}[ht!]
    \hspace*{-0.5cm}
\includegraphics[width = 0.5\linewidth]{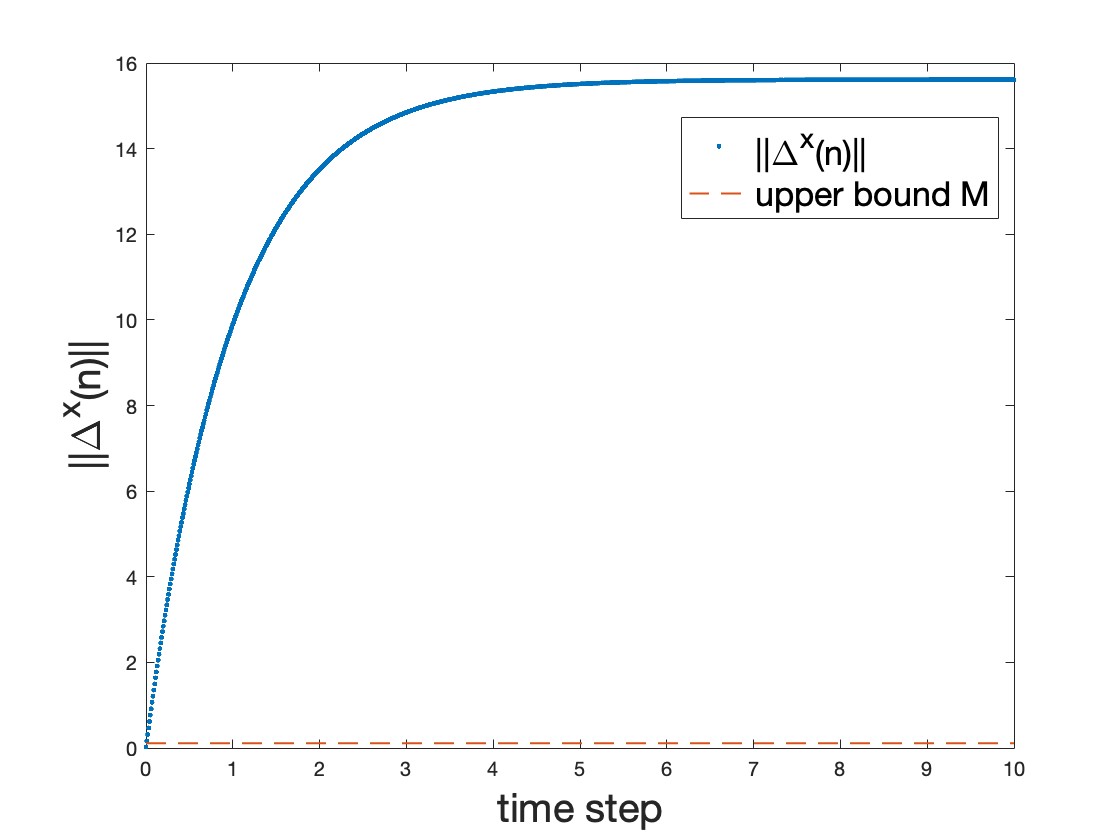}
\includegraphics[width = 0.5\linewidth]{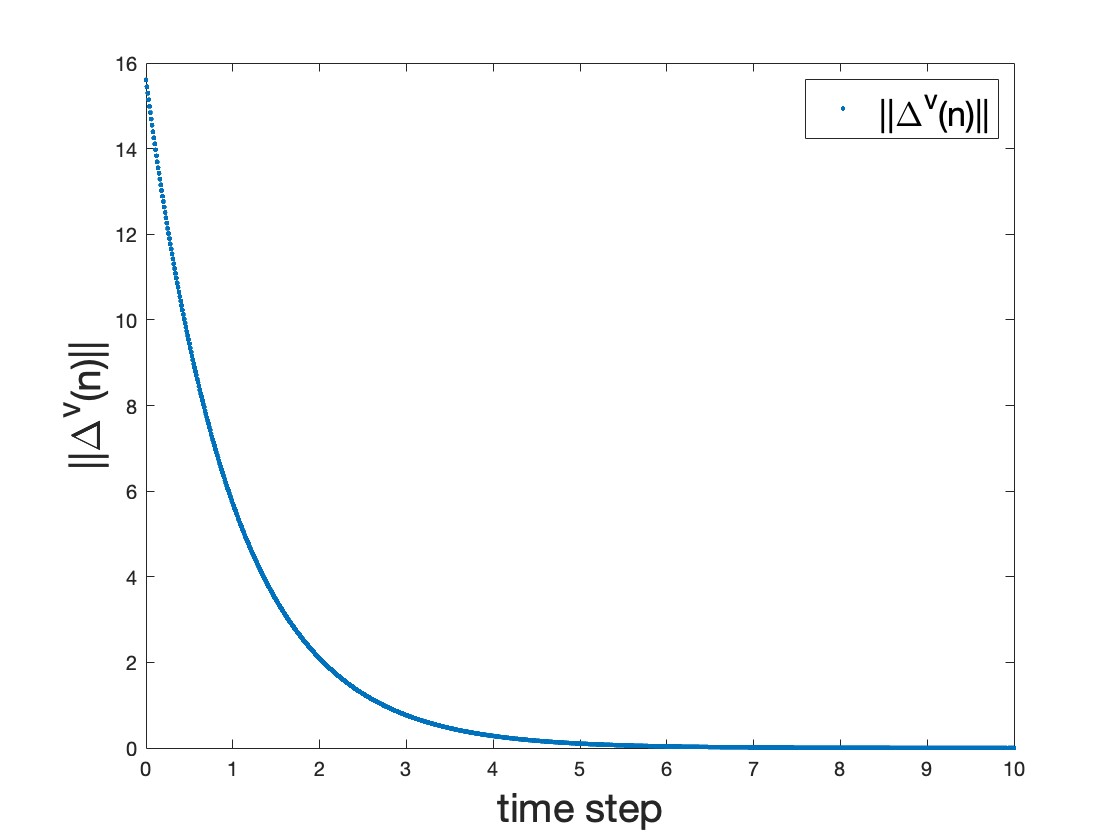}
    \caption{$\|\Delta^x(n)\|_F$ (left) and $\|\Delta^v(n)\|_F$ (right) corresponding to the particle evolutions from Figure~\ref{fig:dynamics_b= 0.1} for $\beta = 0.1$. Flocking emerges even though the conditions of \cref{thm: flocking} are not satisfied.}
    \label{fig:flocking_b = 0.1}
\end{figure}
To see if our theoretical threshold $M$ for the deviations in positions corresponds to a physical phase transition, we tested for emergence of flocking with initial deviations in position closely below and above $M$; see Figure~\ref{fig:deviation_close} for $0.9 \cdot M \le \|\Delta^x(0)\|_F \le 0.95\cdot M$, and Figure~\ref{fig:deviation_close_violated} for $1.05 \cdot M \le \|\Delta^x(0)\|_F \le 1.1 \cdot M$. No sharp transition occurs, and in all cases, we observe flocking eventually.
This indicates that our estimates in \cref{thm: flocking} are likely not sharp. 
\begin{figure}[ht!]
    \hspace*{-0.5cm}
    \includegraphics[width=0.5\linewidth]{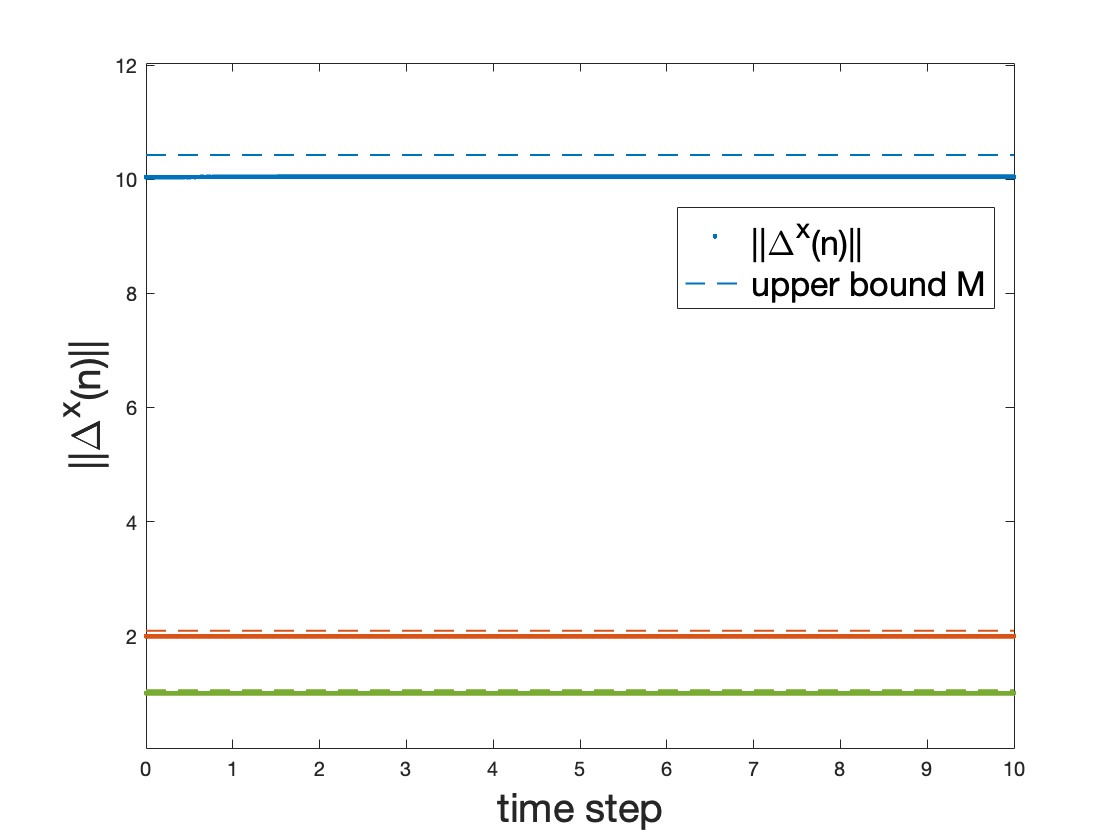}
    \includegraphics[width=0.5\linewidth]{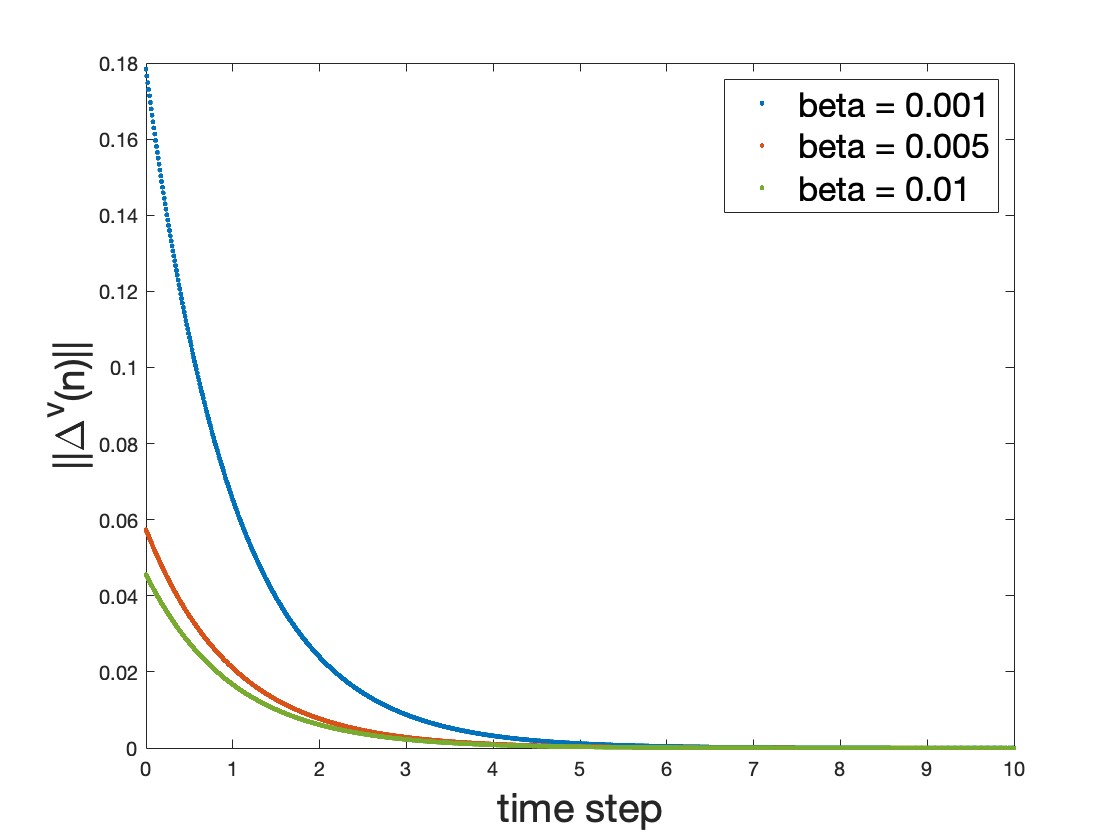}
    \captionof{figure}{Deviation of position (left) and velocity (right) when $\|\Delta^x(0)\|_F$ is between $ 0.9\cdot M$ and $0.95\cdot M$; ($\beta = 0.001, 0.005, 0.01$) with ($M = 0.9606, 0.9352, 0.9032$).}
    \label{fig:deviation_close}
\end{figure}

\begin{figure}[ht!]
    \hspace*{-0.5cm}
    \includegraphics[width=0.5\linewidth]{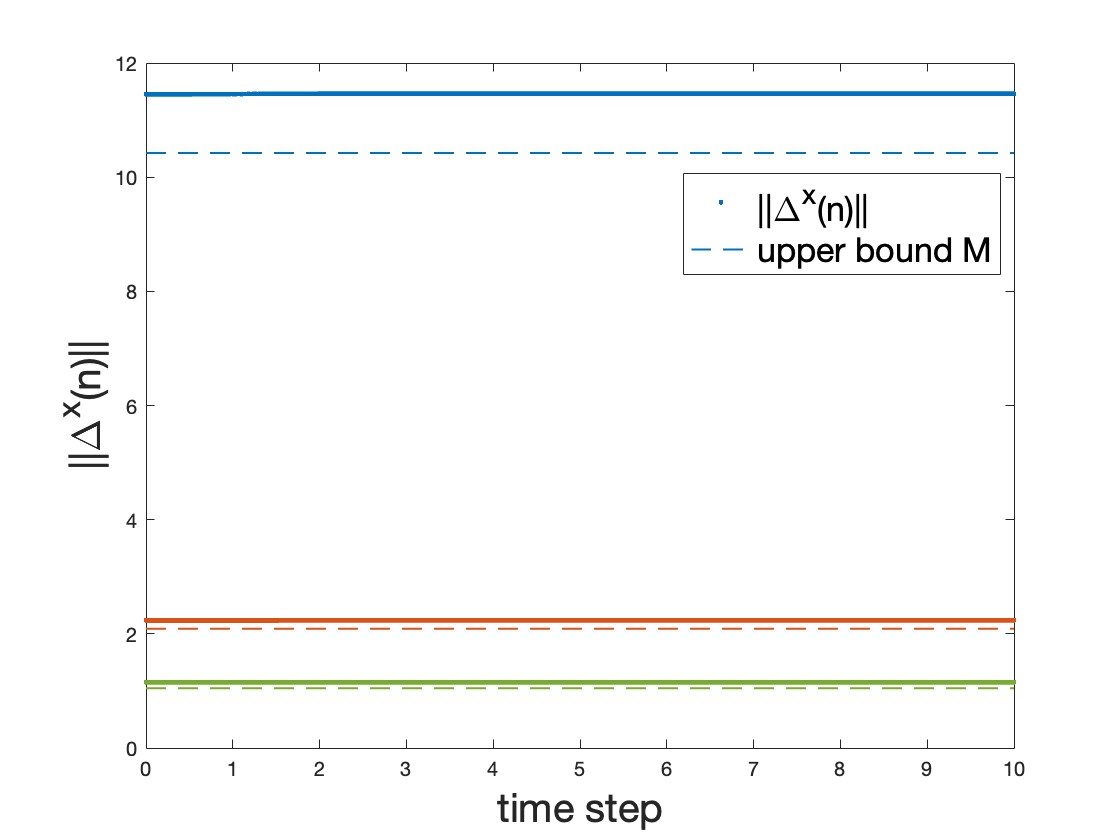}
    \includegraphics[width=0.5\linewidth]{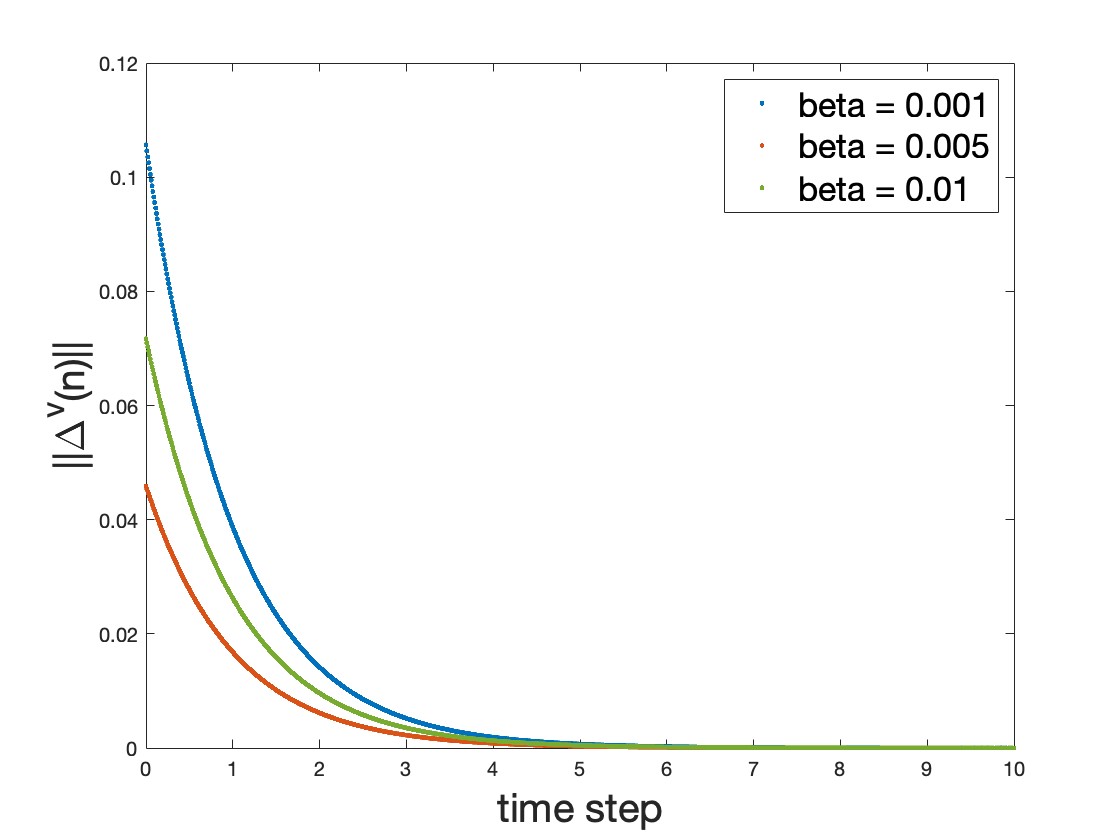}
    \captionof{figure}{Deviation of position (left) and velocity (right) when $\|\Delta^x(0)\|_F$ is between $1.05\cdot M$ and $1.1 \cdot M$; ($\beta = 0.001, 0.005, 0.01$) with $M=(1.0684, 1.100, 1.0791)$.}
    \label{fig:deviation_close_violated}
\end{figure}
\pagebreak
\newpage

Next, we numerically investigate the $l_2$ stability result \cref{thm: stability} under the same setting as the experiments for the flocking estimate. The numerical results (see Figure~\ref{fig:Delta_allbetas}) are consistent with our theoretical analysis: 
The shape discrepancy between two pairs of solutions $(X, V)$, and $(\overline{X}, \overline{V})$ are controlled uniformly in time by its initial value. 

\begin{figure}[ht!]
    \hspace*{-0.5cm}
    \includegraphics[width=0.7\linewidth]{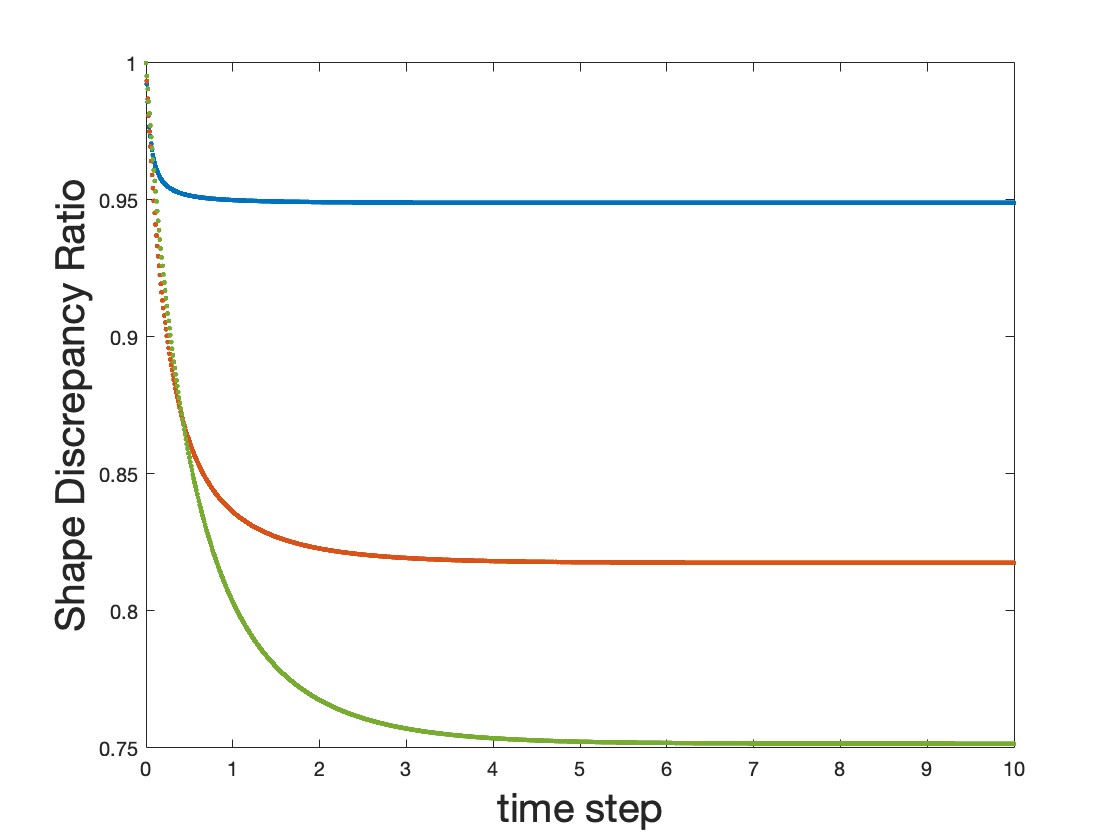}
    \captionof{figure}{Shape Discrepancy Ratio $ \frac{\|\Delta^x(n) - \Delta^{\overline{x}}(n)\|_F + \|\Delta^v(n) - \Delta^{\overline{v}}(n)\|_F}{\|\Delta^x(0) - \Delta^{\overline{x}}(0)\|_F + \|\Delta^v(0) - \Delta^{\overline{v}}(0)\|_F}$ for ($\beta = 0.001, 0.005, 0.01$).}
    \label{fig:Delta_allbetas}
\end{figure}
\pagebreak
\newpage

\section{Conclusion} \label{sec:6}
\setcounter{equation}{0} 
In this paper, we have addressed three quantitative estimates for the discrete MT model, corresponding to a forward Euler discretization of the continuous MT model. Compared to the CS model in which the forcing term is skew-symmetric in the index exchange map, the flocking force in the MT model does not have this skew-symmetric property. Hence, the total momentum for the MT model is not conserved. This lack of conservation law poses technical challenges for the flocking and stability analysis. For example, the standard energy estimates cannot be used for the continuous MT model as it is. 

Fortunately, however, the nonlinear functional approach introduced in Ha-Liu \cite{H-Liu} can be applied with the spatial and velocity diameters as suitable nonlinear functionals. In \cite{MT14}, such a nonlinear functional approach was adapted to prove the flocking result for the continuous MT model. 

For the discrete MT model, we provide the following three quantitative estimates. First, we show that the discrete model can exhibit asymptotic flocking under a suitable framework in terms of initial data and system parameters. Second, we show that the discrepancy between the discrete and continuous solutions tends to zero uniformly in time. For this uniform transition from the discrete model to the continuous model, we used the standard finite-time transition result as well as the asymptotic flocking estimates for both the discrete and the continuum models. 
Since the flocking estimate on the discrete level is a conditional result, the uniform-in-time transition relies on the same constraints for the initial data. In contrast, the finite-time transition can be shown for a general class of initial data. Thus, there is a trade-off between using a uniform-in-time transition with a limited choice of admissible initial data and a finite-time transition with a limited time horizon. Finally, we show that the difference in shape discrepancy between two solutions measured in the $\ell^2$-distance can be controlled by their initial value. 

Our numerical investigations in Section~\ref{sec:5} suggest that the conditions for flocking to occur may be weakened. Noting that unconditional flocking for a slowly decaying communication function $\phi$ was obtained for the continuous model \cite{MT14} in Proposition 2.9, it is a natural next step to seek an unconditional flocking result in the discrete setting as well.

Here, we have chosen an explicit first-order Euler discretization scheme as a starting point for developing the methodology for uniform-in-time flocking and stability. Other discretizations are of interest such as implicit schemes and higher-order schemes; a fourth-order Runge-Kutta scheme for instance is practically desirable from an implementation perspective. This work opens up avenues to develop a uniform-in-time theory for much more general classes of numerical algorithms.

\appendix

\section{Reindexing estimate}\label{sec:App-A}
\setcounter{equation}{0}
In this appendix, we provide a detailed proof of \eqref{eq: sum_ineq} in \cref{sec:3}. The argument follows the proof strategy in \cite[Appendix A]{Ha_Zhang_CS_continuous_transition}.
\begin{lemma}\label{lem:reindex}
Consider the scalar sequence $\{ \|\Delta^x(n) \|_F\}_{n\in\N}$ given via solutions to the discrete MT model~\eqref{eq: discrete_MT}.
Assume there exists a positive constant $n^\infty = n^{\infty}(M) > 0$ such that 
    \begin{equation}\label{eq: assump-appendix}
        \|\Delta^x(n) \|_F \leq M \quad \forall~n < n ^{\infty} \quad \text{ and } \quad \|\Delta^x(n ^{\infty}) \|_F > M\,.
    \end{equation} 
Then it is possible to reindex the set
\[ \{\|\Delta^x(n)\|_F:\|\Delta^x(n)\|_F\geq \|\Delta^x(0)\|_F\}_{n=0}^{n^\infty} \]
into $\{ \|\Delta^x_q \|_F\}_{q = 0}^{K}$ so that it satisfies 
\[ \|\Delta^x(0)\|_F = \|\Delta^x_0\|_F, \quad \|\Delta^x(n^{\infty})\|_F = \|\Delta^x_K\|_F, \quad \|\Delta^x_q\|_F < \|\Delta^x_{q +1}\|_F. \]
and
\[ \sum_{q=0}^{K-1} \left(\|\Delta^x_{q+1}\|_F - \|\Delta^x_{q}\|_F \right)\psi(\|\Delta^x_q\|_F) \leq \sum_{n=0}^{n^\infty-1} \left(\|\Delta^x(n+1)\|_F-\|\Delta^x(n)\|_F\right) \psi(\|\Delta^x(n)\|_F).
\]
\end{lemma}

\begin{proof}
Recall that $\psi(s):= 1 - \phiLip N s$ is a decreasing function. Now, we claim that the sequence $\{\|\Delta^x_q\|_F\}_{q=0}^K$ is a monotone increasing sequence, if necessary, by re-indexing the set 
\[ \Big \{\|\Delta^x(n)\|_F:\|\Delta^x(n)\|_F\geq \|\Delta^x(0)\|_F \Big \}_{n=0}^{n^\infty}.\] In other words, we have
\begin{equation} \label{App-A}
\|\Delta^x(0)\|_F = \|\Delta^x_0\|_F < \|\Delta^x_1\|_F < \cdots < \|\Delta^x_K\|_F = \|\Delta^x(n^\infty)\|_F.
\end{equation}
In what follows, we verify \eqref{App-A} in two steps. \newline

\noindent $\bullet$ \textbf{Step A} (Refinement):~Note that the finite sequence $\{\|\Delta^x(n)\|_F\}_{n=0}^{n^\infty}$ might not be monotonic. So when we move from $\|\Delta^x(n)\|_F$ to $\|\Delta^x(n+1)\|_F$, we might fly over other values in $\{\|\Delta^x(n)\|_F\}_{n=0}^{n^\infty}$. To get rid of such situations, we refine the discrete-time steps $\{0,h,\ldots, n^\infty h\}$ as follows. \newline

Suppose there exists some $ n' <n^\infty$ such that 
\begin{align}
\begin{aligned} \label{A-1}
& \mbox{Either}~~\|\Delta^x(n)\|_F < \|\Delta^x(n')\|_F < \|\Delta^x(n+1)\|_F \\
& \hspace{1.5cm}  \text{or}\quad \|\Delta^x(n)\|_F > \|\Delta^x(n')\|_F > \|\Delta^x(n+1)\|_F.
\end{aligned}
\end{align}
Then we add a new grid point $\overline{n}\in(n, n+1)$ such that 
\[ \|\Delta^x(\overline{n})\|_F = \|\Delta^x(n')\|_F. \]
and remove $n'$ from the previous index set. Note that we do not require that $\|\Delta^x(\overline{n})\|_F$ corresponds to the shape discrepancy of a set of particle solutions $\{x_i(\overline{n})\}_{i=1}^N$ of the discrete MT model. Here, we are simply relabeling a set of scalars.
We repeat this process until the relation \eqref{A-1} does not happen for a new index set $\{l\}_0^I$ with 
\[ 0\leq l \leq I, \quad  \|\Delta^x(0)\|_F = \|\Delta^x_0\|_F \quad \mbox{and} \quad \|\Delta^x(I)\|_F = \|\Delta^x(n^\infty)\|_F. \]
Note that $\{\|\Delta^x(n)\|_F\}_{n=0}^{n^\infty}$ and $\{\|\Delta^x(l)\|_F\}_{l=0}^I$ are identical as sets. For the new index set $\{l\}_0^I$ and corresponding sequence $\{\|\Delta^x(l)\|_F\}_{l=0}^{I}$, we claim
\begin{align}\label{A-2}
\begin{aligned} 
	&\sum_{l=0}^{I-1} \left(\|\Delta^x(l+1)\|_F-\|\Delta^x(l)\|_F\right) \psi(\|\Delta^x(l)\|_F)\\
	&\hspace{1cm}\leq
	\sum_{n=0}^{n^\infty-1} \left(\|\Delta^x(n+1)\|_F-\|\Delta^x(n)\|_F\right) \psi(\|\Delta^x(n)\|_F).
\end{aligned}
\end{align}
If we have a new grid point $\overline{n}\in(n, n+1)$, then we have
\begin{align*}
	&\left(\|\Delta^x(n+1)\|_F - \|\Delta^x(\overline{n})\|_F\right) \psi (\|\Delta^x(\overline{n})\|_F) + 
	\left(\|\Delta^x(\overline{n})\|_F - \|\Delta^x(n)\|_F\right) \psi (\|\Delta^x(n)\|_F)\\
	&\hspace{1cm}\leq \left(\|\Delta^x(n+1)\|_F - \|\Delta^x(\overline{n})\|_F\right) \psi (\|\Delta^x(n)\|_F) + 
	\left(\|\Delta^x(\overline{n})\|_F - \|\Delta^x(n)\|_F\right) \psi (\|\Delta^x(n)\|_F)\\
	&\hspace{1cm}=\left(\|\Delta^x(n+1)\|_F-\|\Delta^x(n)\|_F\right) \psi(\|\Delta^x(n)\|_F).
\end{align*}
Here we used that $\psi$ is a decreasing function. We can conclude for \eqref{A-2} using this observation several times. It then remains to show that
\begin{align}\label{A-remaingoal}
\sum_{q=0}^{K-1} \left(\|\Delta^x_{q+1}\|_F - \|\Delta^x_{q}\|_F \right)\psi(\|\Delta^x_q\|_F) 
\leq \sum_{l=0}^{I-1} \left(\|\Delta^x(l+1)\|_F-\|\Delta^x(l)\|_F\right) \psi(\|\Delta^x(l)\|_F)\,.
\end{align}
We split the set $\{l\}_0^I$ into two subsets as follows:
\[
 \mathcal{N}^+ :=\{l:\|\Delta^x(l)\|_F\geq \|\Delta^x(0)\|_F\}, \quad  \mathcal{N}^- :=\{l:\|\Delta^x(l)\|_F \leq \|\Delta^x(0)\|_F\}.
\]		
By the construction of $\{l\}_0^I$, two consecutive indices $l$ and $l+1$ should belong to either $\mathcal{N}^+$ or $\mathcal{N}^-$, except when 
\[ \|\Delta^x(l)\|_F = \|\Delta^x(l+1)\|_F = \|\Delta^x(0)\|_F. \]
 Note that the latter is of little importance since in this case
\[\left(\|\Delta^x(l+1)\|_F-\|\Delta^x(l)\|_F\right) \psi(\|\Delta^x(l)\|_F)=0.\]
Therefore, the summation in the left-hand side of \eqref{A-2} can be rewritten as 
\begin{align}
\begin{aligned} \label{A-3}
	&\sum_{l=0}^{I-1} \left(\|\Delta^x(l+1)\|_F-\|\Delta^x(l)\|_F\right) \psi(\|\Delta^x(l)\|_F) \\
	&\hspace{1cm} = \left(\sum_{l,(l+1)\in\mathcal{N}^+}+\sum_{l,(l+1)\in\mathcal{N}^-}\right)  \left(\|\Delta^x(l+1)\|_F-\|\Delta^x(l)\|_F\right) \psi(\|\Delta^x(l)\|_F). 
\end{aligned}
\end{align}

\vspace{0.5cm}

\noindent $\bullet$ \textbf{Step B} (Estimate):~We denote $\mathcal{T}:=\{\|\Delta^x_q\|_F\}_{q=0}^K$. 
Then, we can define a surjective function $\alpha: l \in {\mathcal N}^+~~\mapsto~~\|\Delta^x_q\|_F \in  \mathcal{T}$ and a bijective function $\beta:~
\|\Delta^x_q\|_F \in  \mathcal{T}~~\mapsto~~q  \in \{0,\dots K\}$.
Now we define $\gamma:\mathcal{N}^+ \to \{q\}_0^K$ as $\gamma:=\beta \circ \alpha$, the composition of $\beta$ and $\alpha$, which is obviously surjective. From the construction of the sequence $\{l\}_0^I$, we know that $\|\Delta^x(l)\|_F$ and $\|\Delta^x(l+1)\|_F$ are two adjacent elements in $\mathcal{T}=\{\|\Delta^x_q\|_F\}_{q=0}^K$. In simpler terms, we have
\[\gamma(l+1) - \gamma(l) = \pm1.\]
Therefore, we denote the pre-image of $\gamma$ as $\gamma^{-1}$, and we can rewrite the first summation of \eqref{A-3} as 
\begin{align}\label{A-4}
	\begin{aligned}
		&\sum_{l,(l+1)\in\mathcal{N}^+} \left(\|\Delta^x(l+1)\|_F - \|\Delta^x(l)\|_F\right) \psi(\|\Delta^x(l)\|_F)\\
		&\hspace{0.5cm} =\sum_{q=0}^{K-1} \left(\sum_{\substack{l\in \gamma^{-1}(q)\\l+1\in\gamma^{-1}(q+1)}} + \sum_{\substack{l+1\in \gamma^{-1}(q)\\l\in\gamma^{-1}(q+1)}}\right)\left(\|\Delta^x(l+1)\|_F - \|\Delta^x(l)\|_F\right) \psi(\|\Delta^x(l)\|_F).
	\end{aligned}
\end{align}
Next, we note that both $\{\|\Delta^x(l)\|_F\}_{l=0}^I$ and $\{\|\Delta^x_q\|_F\}_{q=0}^K$ construct a path from $\|\Delta^x(0)\|_F$ to $\|\Delta^x(n^\infty)\|_F$ with 
\begin{equation*}
	\begin{cases}
	 \displaystyle	\|\Delta^x(l+1)\|_F - \|\Delta^x(l)\|_F = \|\Delta^x_{q+1}\|_F - \|\Delta^x_{q}\|_F, &l\in\gamma^{-1}(q) \text{ and }l+1\in\gamma^{-1}(q+1),\\
	\displaystyle	\|\Delta^x(l+1)\|_F - \|\Delta^x(l)\|_F = -\left(\|\Delta^x_{q+1}\|_F - \|\Delta^x_{q}\|_F\right), &l\in\gamma^{-1}(q+1)\text{ and }l+1\in\gamma^{-1}(q).
	\end{cases}
\end{equation*}
In other words, each path from $\|\Delta^x(l)\|_F$ to $\|\Delta^x(l+1)\|_F$ corresponds to a path from $\|\Delta^x_{q}\|_F$ to $\|\Delta^x_{q+1}\|_F$ or from $\|\Delta^x_{q+1}\|_F$ to $\|\Delta^x_{q}\|_F$. Since $\|\Delta^x(n^\infty)\|_F > \|\Delta^x(0)\|_F$ we can check that for fixed $q$, 
\begin{equation*}
	\left|\{l:l\in\gamma^{-1}(q)~\text{and } l+1\in\gamma^{-1}(q+1)\}\right| = \left|\{l:l\in\gamma^{-1}(q+1)~\text{and } l+1\in\gamma^{-1}(q)\}\right| + 1,
\end{equation*}
where $|\cdot|$ gives the cardinality of the set. Therefore, for fixed $q$, we obtain 
\begin{align}
	&\left(\sum_{\substack{l\in \gamma^{-1}(q)\\l+1\in\gamma^{-1}(q+1)}} + \sum_{\substack{l+1\in \gamma^{-1}(q)\\l\in\gamma^{-1}(q+1)}}\right)\left(\|\Delta^x(l+1)\|_F - \|\Delta^x(l)\|_F\right) \psi(\|\Delta^x(l)\|_F)\notag\\
	&\hspace{0.5cm}= \sum_{\substack{l+1\in \gamma^{-1}(q)\\l\in\gamma^{-1}(q+1)}} \left(\|\Delta^x(l+1)\|_F - \|\Delta^x(l)\|_F\right) \Big(\psi(\|\Delta^x(l)\|_F) - \psi(\|\Delta^x(l+1)\|_F)\Big)\notag\\
	&\hspace{2.5cm} + (\|\Delta^x_{q+1}\|_F-\|\Delta^x_q\|_F)\psi(\|\Delta^x_q\|_F)\notag\\
	&\hspace{0.5cm}= \sum_{\substack{l+1\in \gamma^{-1}(q)\\l\in\gamma^{-1}(q+1)}} \left(\|\Delta^x_q\|_F - \|\Delta^x_{q+1}\|_F\right) \Big(\psi(\|\Delta^x_{q+1}\|_F) - \psi(\|\Delta^x_q\|_F)\Big)\notag\\
	&\hspace{2.5cm} + (\|\Delta^x_{q+1}\|_F-\|\Delta^x_q\|_F)\psi(\|\Delta^x_q\|_F)\notag\\
	&\hspace{0.5cm} \geq (\|\Delta^x_{q+1}\|_F-\|\Delta^x_q\|_F)\psi(\|\Delta^x_q\|_F)\,,\label{A-5}
\end{align}
where the last inequality follows from the decreasing property of $\psi$. 
Now we combine \eqref{A-4} and \eqref{A-5} to get
\begin{align} \label{A-6}
	\begin{aligned}
	&\sum_{l,(l+1)\in\mathcal{N}^+} \left(\|\Delta^x(l+1)\|_F - \|\Delta^x(l)\|_F\right) \psi(\|\Delta^x(l)\|_F)\\
	&\hspace{1cm} =\sum_{q=0}^{K-1} \left(\sum_{\substack{l\in \gamma^{-1}(q)\\l+1\in\gamma^{-1}(q+1)}} + \sum_{\substack{l+1\in \gamma^{-1}(q)\\l\in\gamma^{-1}(q+1)}}\right)\left(\|\Delta^x(l+1)\|_F - \|\Delta^x(l)\|_F\right) \psi(\|\Delta^x(l)\|_F)\\
	&\hspace{1cm} \geq\sum_{q=0}^{K-1} (\|\Delta^x_{q+1}\|_F-\|\Delta^x_q\|_F)\psi(\|\Delta^x_q\|_F)\,.
	\end{aligned}
\end{align}
Finally, we use a similar argument as above to derive 
\begin{equation} \label{A-7}
	\sum_{l,(l+1)\in\mathcal{N}^-} \left(\|\Delta^x(l+1)\|_F - \|\Delta^x(l)\|_F\right) \psi(\|\Delta^x(l)\|_F)\geq 0.
\end{equation}
More precisely, reindexing the set
\[ \{\|\Delta^x(l)\|_F:\|\Delta^x(l)\|_F\leq \|\Delta^x(0)\|_F\}_{l=0}^{I} \]
into a monotone increasing sequence $\mathcal{T}':=\{\|\Delta^x_p\|_F\}_{p=0}^{K'}$. Define $\alpha': l \in {\mathcal N}^-~~\mapsto~~\|\Delta^x_p\|_F \in  \mathcal{T}'$, $\beta':
\|\Delta^x_p\|_F \in  \mathcal{T}'~~\mapsto~~p  \in \{0,\dots K'\}$ and $\gamma':\mathcal{N}^- \to \{p\}_0^{K'}$ as $\gamma' := \beta'\circ\alpha'$. Since $\|\Delta^x(n^\infty)\|_F>\|\Delta^x(0)\|_F$, for fixed $p$, 
\begin{equation*}
	\left|\{l:l\in\gamma'^{-1}(p)~\text{and } l+1\in\gamma'^{-1}(p+1)\}\right| = \left|\{l:l\in\gamma'^{-1}(p+1)~\text{and } l+1\in\gamma'^{-1}(p)\}\right|.
\end{equation*}
Therefore, for fixed $p$, we have
\begin{align*}
	&\left(\sum_{\substack{l\in \gamma'^{-1}(p)\\l+1\in\gamma'^{-1}(p+1)}} + \sum_{\substack{l+1\in \gamma'^{-1}(p)\\l\in\gamma'^{-1}(p+1)}}\right)\left(\|\Delta^x(l+1)\|_F - \|\Delta^x(l)\|_F\right) \psi(\|\Delta^x(l)\|_F)\notag\\
	&\hspace{0.5cm}= \sum_{\substack{l+1\in \gamma'^{-1}(p)\\l\in\gamma'^{-1}(p+1)}} \left(\|\Delta^x(l+1)\|_F - \|\Delta^x(l)\|_F\right) \Big(\psi(\|\Delta^x(l)\|_F) - \psi(\|\Delta^x(l+1)\|_F)\Big)\notag\\
	&\hspace{0.5cm}= \sum_{\substack{l+1\in \gamma'^{-1}(p)\\l\in\gamma'^{-1}(p+1)}} \left(\|\Delta^x_p\|_F - \|\Delta^x_{p+1}\|_F\right) \Big(\psi(\|\Delta^x_{p+1}\|_F) - \psi(\|\Delta^x_p\|_F)\Big)\notag\\
	&\hspace{0.5cm} \geq 0,
\end{align*}
where we used the decreasing property of $\psi$ at the last inequality. This leads to
\begin{align*}
	\begin{aligned}
	&\sum_{l,(l+1)\in\mathcal{N}^-} \left(\|\Delta^x(l+1)\|_F - \|\Delta^x(l)\|_F\right) \psi(\|\Delta^x(l)\|_F)\\
	&\hspace{1cm} =\sum_{p=0}^{K'-1} \left(\sum_{\substack{l\in \gamma'^{-1}(p)\\l+1\in\gamma'^{-1}(p+1)}} + \sum_{\substack{l+1\in \gamma'^{-1}(p)\\l\in\gamma'^{-1}(p+1)}}\right)\left(\|\Delta^x(l+1)\|_F - \|\Delta^x(l)\|_F\right) \psi(\|\Delta^x(l)\|_F)\\
	&\hspace{1cm} \geq0\,.
	\end{aligned}
\end{align*}
Combining \eqref{A-3} with \eqref{A-6} and \eqref{A-7} shows \eqref{A-remaingoal}. Then the final estimate in \cref{lem:reindex} follows from \eqref{A-2} and \eqref{A-remaingoal}.
\end{proof}

\printbibliography

\end{document}